\documentclass[11pt, oneside]{amsart}
\usepackage{amsmath}
\usepackage{amsthm}
\usepackage{amssymb} 
\usepackage[mathscr]{euscript}
\usepackage{comment}
\excludecomment{toexclude} 
\usepackage{bm}
\usepackage{fancyhdr}
\usepackage{enumerate}
\usepackage{amscd}
\usepackage[all]{xy}
\usepackage{graphicx}
\usepackage{color}
\usepackage{hyperref}
\hypersetup{
    colorlinks,
    citecolor=black,
    filecolor=black,
    linkcolor=black,
    urlcolor=black
}
\usepackage{tikz}
\usepackage[all]{xy}
\usepackage{graphicx}
\usetikzlibrary{backgrounds}

\usepackage{accents}

\renewcommand{\a}{\alpha}
\renewcommand{\b}{\beta}

\renewcommand{\d}{\delta}
\newcommand{\D}{\Delta}

\newcommand{\Si}{\Sigma}

\newcommand{\cM}{{\mathcal M}}

\newcommand{\cL}{{\mathcal L}}
\newcommand{\cE}{{\mathcal E}}
\newcommand{\cU}{{\mathcal U}}

\newcommand{\cX}{{\mathcal X}}
\newcommand{\cQ}{{\mathcal Q}}

\newcommand{\bR}{\mathbb R}

\newcommand{\be}{\begin{equation}}
\newcommand{\ee}{\end{equation}}
\newcommand{\bes}{\begin{equation*}}
\newcommand{\ees}{\end{equation*}}

\newcommand{\pa}{\partial}
\renewcommand{\to}{\rightarrow}

\newcommand{\Ga}{\mathsf{G}}
\newcommand{\DD}{\mathcal{D}}
\newcommand{\Ric}{\mathrm{Ric}}
\newcommand{\dnu}{\left( \nu'(h)\right)}

\theoremstyle{plain}
\newtheorem{lemma}{Lemma}[section]
\newtheorem{proposition}[lemma]{Proposition}
\newtheorem{theorem}[lemma]{Theorem}
\newtheorem{Theorem}{Theorem}

\newtheorem{corollary}[lemma]{Corollary}
\newtheorem{Corollary}[Theorem]{Corollary}

\newtheorem{conjecture}[Theorem]{Conjecture}

\numberwithin{equation}{section}

\newtheorem{manualtheoreminner}{Theorem}
\newenvironment{manualtheorem}[1]{%
  \renewcommand\themanualtheoreminner{#1}%
  \manualtheoreminner
}{\endmanualtheoreminner}

\newtheorem{manuallemmainner}{Lemma}
\newenvironment{manuallemma}[1]{%
  \renewcommand\themanuallemmainner{#1}%
  \manuallemmainner
}{\endmanuallemmainner}

\theoremstyle{definition}
\newtheorem{remark}[lemma]{Remark}

\newtheorem{definition}[lemma]{Definition}
\newtheorem{Definition}[Theorem]{Definition}

\DeclareMathOperator{\ps}{\tfrac{\partial}{\partial s}}

\DeclareMathOperator{\dt}{\tfrac{d}{d t}} 
\DeclareMathOperator{\ds}{\tfrac{d}{d s}}

\DeclareMathOperator{\tr}{tr} 
\DeclareMathOperator{\Div}{div}
\DeclareMathOperator{\dvol}{d\mathrm{vol}}
\DeclareMathOperator{\da}{d\sigma}

\DeclareMathOperator{\Ker}{Ker}
\DeclareMathOperator{\Range}{Range}
\DeclareMathOperator{\Dim}{\mathrm{dim}}
 \DeclareMathOperator{\Span}{span}
\DeclareMathOperator{\Int}{Int}

\DeclareMathOperator{\C}{\mathcal{C}}

\usepackage[tmargin=1in,bmargin=1in,lmargin=1.5in,rmargin=1.5in]{geometry}
\begin{document}

\title[Static vacuum extensions near a general static vacuum metric]{Static vacuum extensions with prescribed Bartnik boundary data near a general static vacuum metric}

\author{Zhongshan An}
\address{Department of Mathematics, University of Michigan, Ann Arbor, MI 48109 USA}
\email{zsan@umich.edu}
\author{Lan-Hsuan Huang}
\address{Department of Mathematics, University of Connecticut, Storrs, CT 06269, USA}
\email{lan-hsuan.huang@uconn.edu}
\thanks{The second author was partially supported by the NSF CAREER Award DMS-1452477, DMS-2005588, and DMS-2304966.}

\begin{abstract}
We introduce the notions of static regular of type~(I) and type~(II) and show that they are sufficient conditions for local well-posedness of solving asymptotically flat, static vacuum metrics with  prescribed Bartnik boundary data. We then show that hypersurfaces in a very general open and dense family of hypersurfaces are static regular of type~(II).  As applications, we confirm Bartnik's static vacuum extension conjecture for a large class of Bartnik boundary data, including those that can be far from Euclidean and have large ADM masses, and give many new examples of static vacuum metrics with intriguing geometry.   
\end{abstract}

\maketitle

\maketitle
{\small
\tableofcontents
\addtocontents{toc}{\setcounter{tocdepth}{2}} } 
\section{Introduction}\label{se:intro}

Let $n\ge 3$ and $M$ be a smooth $n$-dimensional manifold. For a Riemannian metric $g$ and a scalar-valued function $u$ on $M$, $n\ge 3$, we define the \emph{static vacuum operator}
\begin{align} \label{eq:static}
	S(g, u) := \big(-u \Ric_g  + \nabla^2_g u, \Delta_g u\big).
\end{align}
If $S( g,  u)=0$, $( g,  u)$ is called a  static vacuum \emph{pair}, and $(M, g,  u)$ is called a static vacuum \emph{triple}. We also refer to $ g$  as a static vacuum metric.  When $ u>0$, $(M,  g,  u)$ defines a Ricci flat ``spacetime'' metric $\mathbf{g} = \pm   u^2 dt^2 +  g$  (with either a positive or negative sign in the $dt^2$ factor)  on $\mathbb{R}\times M$ that carries a global Killing vector $\partial_t$. The most well-known examples of asymptotically flat, static vacuum metrics include the Euclidean metric and the family of (Riemannian) Schwarzschild metrics.

When $M$ has nonempty boundary $\partial M$, the induced boundary data $(g^\intercal, H_g)$ in \eqref{eq:bdv} is called the {\it Bartnik boundary data} of $g$, where $g^\intercal$ is the restriction of g (or the induced metric) on the tangent bundle of $\partial M$ and the mean curvature $H_g$ is the tangential divergence of the unit normal vector $\nu$. When $M$ is asymptotically flat, We choose $\nu$ to point towards infinity of $M$. 

The main content of this paper is to solve for asymptotically flat $(g, u)$ to $S(g, u)=0$ with prescribed Bartnik boundary data $(g^\intercal, H_g)$ and confirm the following conjecture of R. Bartnik~\cite[Conjecture 7]{Bartnik:2002} for a wider class of boundary data.
\begin{conjecture}[Bartnik's static extension conjecture]\label{conjecture}
Let $(\Omega_0, g_0)$ be a compact, connected, $n$-dimensional Riemannian manifold with scalar curvature $R_{g_0}\ge 0$ and with nonempty boundary $\Sigma$. Suppose $H_{g_0}$ is strictly positive somewhere. Then there exists a geometrically unique  asymptotically flat, static vacuum manifold $(M, g)$ with boundary $\partial M$ diffeomorphic to $\Sigma$ such that  $(g^\intercal, H_g) =  (g_0^\intercal, H_{g_0})$.
\end{conjecture}

{Note that the mean curvature $H_{g_0}$ of $\Sigma \subset (\Omega_0, g_0)$ is calculated with respect to the \emph{outward} unit normal, specifically, the one consistent with $\partial M \subset (M, g)$.}

The conjecture originated from the Bartnik's program on quasi-local mass~\cite{Bartnik:1989, Bartnik:1993, Bartnik:1997, Bartnik:2002} that goes back to 1989 and has drawn great interest in mathematical relativity and differential geometry in the past two decades; see, for example, the recent survey by M. Anderson~\cite{Anderson:2019}. 

It has been speculated that the conjecture may not hold in general. As observed by Anderson-M. Khuri and by Anderson-J. Jauregui~\cite{Anderson-Khuri:2013, Anderson-Jauregui:2016}, if  $\Sigma$ is an immersed hypersurface in an asymptotically flat, static vacuum triple $(M,\bar g, \bar u)$ that is not \emph{outer} embedded, i.e. $\Sigma$ touches itself from the exterior region $M\setminus \Omega$, then the induced data $(\bar{g}^\intercal, H_{\bar{g}})$ is valid Bartnik boundary data on $\Sigma$, but $(\bar{g}, \bar u)$ in $M\setminus \Omega$ is \emph{not} a valid extension because $M\setminus \Omega$ is not a manifold-with-boundary. Those hypersurfaces are conjectured to be counter-examples to Conjecture~\ref{conjecture} \cite[Conjecture 5.2]{Anderson-Jauregui:2016}. {A very recent work of Anderson \cite{Anderson:2023} provides counter-examples to the conjecture.} 
 Nevertheless, Conjecture~\ref{conjecture} is of fundamental importance toward the structure theory of static vacuum manifolds, so it is highly desirable to prove it even under {additional natural} assumptions. 

Conjecture~\ref{conjecture} can be formulated as solving a geometric boundary value problem $T(g, u) = (0, 0, \tau, \phi)$ for given boundary data $(\tau, \phi)$ where 
\begin{align}\label{eq:bdv}
\begin{split}
	T(g, u) := \begin{array}{l}  
 \left\{ \begin{array}{l} -u \Ric_g  + \nabla^2_g u \\ \Delta_g u 
	\end{array} \right. \quad \mbox{ in } M\\
\left\{ \begin{array}{l} g^\intercal \\ H_g 
	\end{array} \right. \quad \mbox{ on } \partial M.
	\end{array}
\end{split}
\end{align} 
Since the static vacuum operator is highly nonlinear, as the first fundamental step toward Conjecture~\ref{conjecture}, it is natural to make the following conjecture.
\begin{conjecture}[Local well-posedness]\label{co:well-posed}
The geometric boundary value problem~\eqref{eq:bdv} is locally well-posed at every background solution in the sense that if  $(M, \bar g, \bar u)$ is an asymptotically flat, static vacuum triple, 
then for $(\tau, \phi)$ sufficiently close to $(\bar g^{\intercal}, H_{\bar g})$ on $\partial M$, there exists a solution $(g, u)$ to $T(g, u) = (0, 0, \tau, \phi)$ that is geometrically unique near $(\bar g, \bar u)$ and can depend continuously on the boundary data $(\tau, \phi)$. 
\end{conjecture}

In particular, positive results toward Conjecture~\ref{co:well-posed} confirms Conjecture~\ref{conjecture} for large classes of boundary data.

There are partial results toward Conjecture~\ref{co:well-posed} when the background solution $(\bar g, \bar u) = (g_{\mathbb E}, 1)$, the Euclidean pair.   For boundary data $(\tau, \phi)$ sufficiently close to $(g_{S^2}, 2)$, the Bartnik boundary data of $g_{\mathbb E}$ on the unit round sphere in $\mathbb R^3$,  P.~Miao~\cite{Miao:2003} confirmed the local well-posedness under a reflectional symmetric condition. Anderson~ \cite{Anderson:2015} removed the symmetry assumption, based on his work with Khuri~\cite{Anderson-Khuri:2013}.\footnote{Our work is inspired by interesting ideas in \cite{Anderson-Khuri:2013, Anderson:2015}. On the other hand, we are unable to fully verify Theorem 1.1 in \cite{Anderson:2015}; specifically, the claim on p. 3094 asserts that if a kernel element $k=\mathcal L_Z g$ in $M$ with $Z=0$ on $\partial M$, then $Ker (D\Pi) =0$. Note that $Z$ need not be asymptotic to zero, as it can be asymptotic to any Euclidean Killing vector field. Thus, even when $k$ satisfies an elliptic gauge, implying that $Z$ satisfies an elliptic equation, $Z$ need not be zero, which leads to a nontrivial  finite-dimensional kernel. This phenomenon arises from the structure at infinity of asymptotically flat manifolds. In~\cite{An-Huang:2021} and this paper, we handle that extra finite-dimensional kernel. In addition, we provide a rigorous description of the diffeomorphism group acting on asymptotically flat metrics, involving an additional orthogonal gauge (see Lemma~\ref{le:diffeo}), a feature absent in \cite[Theorem 1.1]{Anderson-Khuri:2013}.
 
}
There is a proposed flow approach to axial-symmetric solutions with numerical results by C.~Cederbaum, O.~Rinne, and M. Strehlau \cite{Cederbaum-Rinne-Strehlau:2019}. In our recent work~\cite{An-Huang:2021}, we confirmed Conjecture~\ref{co:well-posed} for $(\tau, \phi)$  sufficiently close to the Bartnik data of $g_{\mathbb E}$ on  wide classes of connected, embedded hypersurfaces $\Sigma$ in Euclidean $\mathbb{R}^n$ for any $n \ge 3$, including
\begin{enumerate}
\item \label{it:hypersurfaces}Hypersurfaces in an open dense subfamily of any foliation of hypersurfaces.
\item Any star-shaped hypersurfaces. 
\end{enumerate}
We also show that the solution $(g,u)$ is geometrically unique in an open neighborhood of $(g_{\mathbb{E}}, 1)$ in the weighted H\"older space. Note that we recently extended Item~\eqref{it:hypersurfaces} to more general families of hypersurfaces in Euclidean space \cite{An-Huang:2022-JMP}.


Those results can be viewed as confirming Conjecture~\ref{co:well-posed} for the Euclidean background. In particular, those static metrics obtained have small ADM masses.  In this paper, we implement several new arguments  and extend our prior results to a general background metric (including any Schwarzschild metrics). Consequently, the new static vacuum metrics obtained can have large ADM masses. To our knowledge, prior to this work, the only known examples of asymptotically flat, static vacuum metrics of large ADM masses are either Schwarzschild metrics (rotationally symmetric) or some Weyl's solutions (axial-symmetric). In addition, this paper presents first results toward Conjecture~\ref{conjecture} for ``large'' boundary data, i.e. those can be far from Euclidean.

Before discussing the main results, we introduce some notations and definitions. 

\noindent {\bf Notation.} Let $(M, \bar g, \bar u)$ be an asymptotically flat, static vacuum triple with $\bar u>0$, possibly with a compact boundary $\partial M$. We will use $\Omega$ to denote a non-empty, bounded,  open subset in $M$ such that $M\setminus \Omega$ is connected and $\partial M\subset \Omega$ if $\partial M\neq \emptyset$. We use $\Sigma:=\partial \big(M\setminus \Omega\big)$ to denote the boundary, and we  always assume that $\Sigma$ is  a connected, embedded, smooth hypersurface in $M$. We say that $(g, u)$ is an \emph{asymptotically flat pair} if the metric $g$ is asymptotically flat and the scalar function $u\to 1$ at infinity. 

Let $L$ denote the linearized operator of $T$ at an asymptotically flat, static vacuum pair $(\bar g, \bar u)$ on $M\setminus\Omega$. That is, for any smooth family of asymptotically flat pairs $(g(s), u(s))$ on $M\setminus\Omega$ with $\big( g(0), u(0) \big) = (\bar g, \bar u)$ and $\big( g'(0), u'(0) \big) = (h, v)$, we define 
\[
L(h, v) := \left. \ds\right|_{s=0} T(g(s), u(s)).
\]
By diffeomorphism invariance of $T$, we trivially have $T(\psi_s^* \bar g, \psi_s^* \bar u) = T(\bar g, \bar u)$ on $M\setminus\Omega$ for any family of diffeomorphisms $\psi_s$ that fix the structure at infinity and boundary, i.e. $\psi_s \in \mathcal D(M\setminus \Omega)$. See Definition~\ref{de:diffeo} for the precise definitions of $\mathcal D(M\setminus \Omega)$ and its tangent space $\mathcal X(M\setminus \Omega)$ at the identity map. Therefore, by differentiating in $s$, $\Ker L $ contains all the ``trivial'' deformations $(L_X \bar g, X(\bar u))$ with vector fields $X\in \mathcal X(M\setminus \Omega)$.

Our first main result says that Conjecture~\ref{co:well-posed} holds, provided that $\Ker L$ is ``trivial''. 
\begin{Theorem}\label{existence}
Let $(M, \bar{g}, \bar{u})$ be an asymptotically flat, static vacuum triple with $\bar{u} > 0$, and let $\Omega$ be as defined in the notation above. Suppose that 
\begin{align} \label{eq:no-kernel}
	\Ker L = \big\{ (L_X \bar g, X(\bar u)): X\in \mathcal X(M\setminus \Omega)\big\}. 
\end{align}
Then there exist positive constants $\epsilon_0,C$ such that for each $\epsilon\in (0, \epsilon_0)$, if $(\tau,\phi)$ satisfies $\|(\tau,\phi)-(\bar g^\intercal,H_{\bar g})\|_{\C^{2,\a}(\Si)\times \C^{1,\alpha}(\Sigma)}<\epsilon$, then there exists an asymptotically flat,  static vacuum pair $(g,u)$ with $\|(g,u)-(\bar g,\bar u)\|_{\C^{2,\a}_{-q}(M\setminus \Omega)}<C\epsilon$ solving the boundary condition $(g^\intercal, H_g)=(\tau, \phi)$ on $\Sigma$. The solution $(g, u)$ is geometrically unique in a neighborhood of $(\bar g, \bar u)$ in $\C^{2,\alpha}_{-q}(M\setminus \Omega)$ and can depend smoothly  on $(\tau, \phi)$.  
\end{Theorem}
We refer the precise meanings of ``geometric uniqueness'' and  ``smooth dependence'' to Theorem~\ref{exist} and the dependence of the constants  $\epsilon_0, C$ to Remark~\ref{re:constant}.

We formulate two \emph{static regular} conditions that imply \eqref{eq:no-kernel}. Along the boundary~$\Sigma$,
we denote  the second fundamental form $A_{g} :=  (\nabla \nu)^\intercal$ and the $\ell$-th $\nu$-covariant derivative 
of the Ricci tensor, $\Ric_g$,  by 
\[
	\big(\nabla^\ell_\nu \, \Ric_g\big)(Y, Z):= (\nabla^\ell \Ric) (Y, Z; {\underbrace{{\nu, \dots, \nu}}_{\text{\tiny$\ell$ times}}}\,) \quad \mbox{ for any vectors } Y, Z. 
\]
We linearize $\nu$, $A_g$, $H_g$ $\Ric_g$, $\nabla^\ell_\nu \, \Ric_g$  and denote their linearizations at a static vacuum pair $(\bar g, \bar u)$ by $\nu'(h)$, $A'(h)$, $H'(h)$, $\Ric'(h)$, $\big(\nabla_\nu^\ell\Ric\big)'(h)$ respectively. See the formulas in Appendix~\ref{se:formula}.

\begin{Definition}[Static regular]\label{de:static}
\begin{itemize}
\item The boundary $\Sigma$ is said to be \emph{static regular of type (I) in $(M\setminus \Omega , \bar g, \bar u)$} if for any $(h, v)\in \C^{2,\alpha}_{-q}(M\setminus \Omega)$  solving $L(h, v)=0$, there exists a nonempty, connected, open subset $\hat \Sigma\subset \Sigma$ with $\pi_1(M\setminus \Omega, \hat \Sigma)=0$ such that $(h, v)$ satisfies 
\begin{align} \label{eq:static1}
v=0, \quad \big(\nu'(h)\big) (\bar u) + \nu (v)=0,  \quad A'(h)=0\qquad \mbox{ on } \hat \Sigma.
\end{align}

\item The boundary $\Sigma$ is said to be \emph{static regular of type (II) in $(M\setminus \Omega , \bar g, \bar u)$} if for any $(h, v)\in \C^{2,\alpha}_{-q}(M\setminus \Omega)$  solving $L(h, v)=0$, there exists a nonempty, connected, open subset $\hat \Sigma\subset \Sigma$ with $\pi_1(M\setminus \Omega, \hat \Sigma)=0$  satisfying that 
\begin{subequations}
\begin{align} 
&\hat \Sigma \mbox{ is an analytic hypersurface, and} \label{eq:static2}\\
& A'(h)=0,  \quad \big(\Ric'(h)\big)^\intercal =0, \quad \big((\nabla_\nu^k\, \Ric)'(h)\big)^\intercal=0 \quad \mbox{ on } \hat \Sigma\label{eq:static3}
\end{align} 
\end{subequations}
for all positive integers $k$. 
\end{itemize}
\end{Definition}
We may refer to the boundary conditions in type (I) as the \emph{Cauchy boundary condition} and the boundary conditions in type (II) as the \emph{infinite-order boundary condition} (noting it involves only conditions on $h$).

\begin{Theorem}\label{th:trivial}
Let $(M, \bar{g}, \bar{u})$ be an asymptotically flat, static vacuum triple with $\bar{u} > 0$, and let $\Omega$ be as defined in the notation above. Let $\hat \Sigma$ be a nonempty, connected, open subset of $\Sigma$ with $\pi_1(M\setminus \Omega, \hat \Sigma)=0$. Let $(h, v)\in \C^{2,\alpha}_{-q}(M\setminus \Omega)$ satisfy $L (h, v)=0$ in $M\setminus \Omega$ such that either \eqref{eq:static1}  holds or both \eqref{eq:static2},\eqref{eq:static3} hold on $\hat \Sigma$.  Then $(h, v) = \big(L_X \bar g, X(\bar u) \big)$ for $X\in \mathcal X(M\setminus \Omega)$. 

As a consequence, if the boundary $\Sigma$ is static regular in $(M\setminus \Omega, \bar g, \bar u)$ of either type (I) or type (II), then \eqref{eq:no-kernel} holds. Conversely, if \eqref{eq:no-kernel} holds, then both \eqref{eq:static1} and \eqref{eq:static3} hold for $\hat \Sigma = \Sigma$. 
\end{Theorem}

We make some general remarks on the conditions appearing above. 
\begin{remark}\label{re:intro}
\begin{enumerate}
\item For any subset $U\subset M$, the condition on the relative fundamental group $\pi_1(M, U)=0$ says that $U$ is connected and the inclusion map $U \hookrightarrow M$ induces a surjection $\pi_1(U)\to \pi_1(M)$. Thus, in the special case that  $M\setminus \Omega$ is simply connected, the condition $\pi_1(M\setminus \Omega, \hat \Sigma)=0$ always holds. 


\item \label{it:nu} The condition  $\big(\nu'(h)\big)(\bar u) + \nu(v)=0$ on $\hat \Sigma$ in \eqref{eq:static1} can be dropped in the following two general situations: (i)  $\hat \Sigma = \Sigma$, or (ii) the mean curvature of $\hat \Sigma$ in $(M, \bar g)$ is not identically zero. See Section~\ref{sec:relaxed}.
\item Both conditions $v=0, ((\nu')(h)) (\bar u) + \nu (v)=0$ in \eqref{eq:static1} can be dropped when  $(\bar g, \bar u) = (g_{\mathbb E}, 1)$ and $\hat \Sigma=\Sigma$. Thus, the definition of static regular of type~(I) here recovers the definition of \emph{static regular} in \cite[Definition 2]{An-Huang:2021} (see \cite[Lemma 4.8]{An-Huang:2021}).

\item The analyticity of $\hat \Sigma$ in \eqref{eq:static2} will be used together with a fundamental fact that a static vacuum pair $(\bar g, \bar u)$ is analytic in $\Int M$ under a suitable choice of coordinate charts by the result of M\"uller zum Hagen~\cite{Muller-zum-Hagen:1970}. See Appendix~\ref{se:analytic}. 
\end{enumerate}
\end{remark}

We verify that a large, open and dense subfamily of {analytic} hypersurfaces in a general background $(M,\bar g,\bar u)$ are static regular of type (II). 
\begin{Definition}\label{def:one-sided}
Let $\delta$ be a positive number and $\{ \Sigma_t \}, t\in [-\delta, \delta]$, be a parametrized family of  hypersurfaces in $M$ so that $\Sigma_t = \partial  (M\setminus \Omega_t)$ and $\Sigma_t, \Omega_t$ satisfy the same requirements for $\Sigma, \Omega$ as specified in Notation above. 
We say that  $\{\Sigma_t\}$ is a {\it smooth one-sided family of hypersurfaces (generated by $Y$) foliating along $\hat\Sigma_t\subset\Sigma_t$  with $M\setminus \Omega_t$  simply connected relative to $\hat \Sigma_t$} if 
\begin{enumerate}
\item   $\hat{\Sigma}_t$ is an open subset of $\Sigma_t$ satisfying $\pi_1(M\setminus \Omega_t,\hat \Sigma_t)=0$.
\item The smooth deformation vector field $Y = \frac{\partial}{\partial t} \Sigma_t$ satisfies $\bar g(Y,\nu)=\zeta\geq 0$ on each $\Sigma_t$, and $\zeta>0$ in $\hat\Sigma_t$.
\end{enumerate}
\end{Definition}


\begin{Theorem}\label{generic}
Let $(M,\bar g,\bar u)$ be an asymptotically flat, static vacuum triple with $\bar u>0$.  Let $\{\Sigma_t\}$ be a smooth one-sided family of hypersurfaces foliating along $\hat{\Sigma}_t$ with $M\setminus \Omega_t$ simply connected relative to $\hat\Sigma_t$. Furthermore, suppose $\hat{\Sigma}_t$ is an analytic hypersurface for every $t$. Then there is an open dense subset $J\subset [-\delta, \delta]$ such that $\Sigma_t$ is static regular of type~(II) in $(M\setminus \Omega_t, \bar g, \bar u)$ for all $t\in J$. 
\end{Theorem}

Together with Theorem~\ref{existence} and Theorem~\ref{th:trivial}, the above theorem confirms Conjecture~\ref{co:well-posed} on local well-posedness at an open and dense subfamily of background solutions $(M\setminus\Omega_t,\bar g,\bar u)$ {whose boundaries are analytic}.
We expect the local well-posedness to hold  without the assumption that $\hat \Sigma_t$ is analytic, which is indeed the case if the background static vacuum pair is Euclidean, see \cite[Theorem 7]{An-Huang:2022-JMP}. For a general background static vacuum pair considered here, the main difficulty is that the system $L(h, v)$ is highly coupled (see the explicit formula in \eqref{eq:L} below) and we cannot derive $v$ is trivial without the analyticity assumption.

 We also note that if $\pi_1(M \setminus \Omega) =0$ (for example, $\Omega$ is a large ball), one can choose to foliate along any open subset $\hat \Sigma_t$ of an analytic hypersurface $\Sigma_t$ in Theorem~\ref{generic}. The above theorems also lead to other applications as we shall discuss below.




The well-known Uniqueness Theorem of Static Black Holes says that an asymptotically flat, static vacuum triple $(M, \bar g, \bar u)$ with $\bar u=0$ on the boundary $\partial M$ must belong to the Schwarzschild family~\cite{ Israel:1968, Robinson:1977, Bunting-Masood-ul-Alam:1987}. Note that a Schwarzschild manifold  has the following special properties:
\begin{enumerate}
\item \label{it:CMC}$\bar u =0$ on the boundary $\partial M$, which implies that $\partial M$ has  zero mean curvature.
\item \label{it:increasing} $(M, \bar g)$ is foliated by stable, constant mean curvature $(n-1)$-dimensional spheres. 
\end{enumerate}

It is tempting to ask whether the Schwarzschild metrics can be characterized by replacing Item~\eqref{it:CMC} with other geometric conditions. For example, it is natural to ask whether one may replace {the minimal surface in Item~\eqref{it:CMC} with a nearby CMC surface:}
\begin{enumerate}
\item[(1$^\prime$)] \label{it:CMC'}$\bar u>0$ is very small on $\partial M $, and $\partial M$ has constant  mean curvature $H_{\partial M}>0$ that is very small. 
\end{enumerate}
We make a general remark regarding  Item~(\ref{it:CMC'}$^\prime$). Any asymptotically flat metric (under mild asymptotic assumptions) has a foliation of stable CMC surfaces in an asymptotically flat end with the mean curvature going to zero at infinity by~\cite{Huisken-Yau:1996} (also, for example, \cite{Metzger:2007, Huang:2010}), but when the metric is static $\bar g$, the static potential $\bar u$ is very close to $1$ there and hence cannot be arbitrarily small.  

Our next corollary says that  under both the assumptions Item~(\ref{it:CMC'}$^\prime$) and Item~\eqref{it:increasing}, the static uniqueness fails, as there are many, geometrically distinct static vacuum metrics  satisfying both assumptions.

\begin{Corollary}\label{co2}
Let $B\subset \mathbb R^n$ be an open unit ball and $M:=\mathbb R^n\setminus B$. Given any constant $c>0$, there exists  many, asymptotically flat,  static vacuum triple $(M, g, u)$  such that
\begin{enumerate}
\item $ 0<u< c$ on the boundary $\partial M$, and the mean curvature $\partial M$ is equal to a positive constant $H \in (0, c)$. 
\item $(M, g)$ is foliated by stable, constant mean curvature $(n-1)$-dimensional spheres. 
\item $(M, g)$ is not isometric to a Schwarzschild manifold. 
\end{enumerate}
\end{Corollary}

We include a fundamental result used in Theorem~\ref{th:trivial} that  is  of independent interest.  Let $h$ be a symmetric $(0,2)$-tensor on a Riemannian manifold $(M, g)$. We say a vector field $X$ is  \emph{$h$-Killing} if  $L_X g =h$. The next theorem shows that if $h$ is analytic, then a local $h$-Killing vector can globally extend on an analytic manifold. It generalizes the classic result of Nomizu for local Killing vector fields when $h=0$.  
\begin{Theorem}[Cf. { \cite[Lemma 2.6]{Anderson:2008}}]\label{th:extension}
Let $(U, g)$ be a connected, analytic Riemannian manifold. Let $h$ be an analytic, symmetric $(0,2)$-tensor on $U$. Let $\Omega \subset U$ be a connected open subset satisfying  $\pi_1(U, \Omega)=0$. Then an $h$-Killing vector field $X$  in~$\Omega$ can be extended to a unique $h$-Killing vector field on the whole manifold~$U$. 
\end{Theorem}
We remark that a similar statement already appeared in the work of Anderson~ \cite[Lemma 2.6]{Anderson:2008} with an outline of proof. We give an alternative proof similar to Nomizu's proof. The proof is included in Section~\ref{se:h}.

\medskip 
\noindent{\bf Structure of the paper:} 
 In Section~\ref{se:prelim}, we provide preliminary results that will be used in later sections. In Section~\ref{se:gauge}, we introduce the gauge conditions essential to study the boundary value problem and discuss the kernel of the linearized problem. In Section~\ref{se:static-regular}, we prove Theorem~\ref{th:trivial} and, along the way, discuss the assumptions of static regular assumptions. Then we complete the proof of Theorem~\ref{existence} in Section~\ref{se:solution}. Theorem~\ref{generic} and Corollary~\ref{co2} are proven in Section~\ref{se:perturbation}. We prove Theorem~\ref{th:extension} in Section~\ref{se:h}, which can be read independently of all other sections.

\section{Preliminaries}\label{se:prelim}

In this section, we collect basic definitions, notations,  facts, and some fundamental results that will be used in the later sections.

\subsection{The structure at infinity}

Let $n\ge 3$ and $M$ denote an $n$-dimensional smooth complete manifold (possibly with nonempty compact boundary) such that there are compact subsets $B\subset M$, $B_1\subset \mathbb{R}^n$, and a diffeomorphism $\Phi: M\setminus B\longrightarrow \mathbb R^n\setminus B_1$. We use the chart $\{ x \}$ on $M\setminus B$ from $\Phi$ and a fixed atlas of $B$ to define the weighted H\"older space $\C^{k,\alpha}_{-q}(M)$ for $k=0, 1, \dots$, $\alpha \in (0, 1)$, and $\frac{n-2}{2} < q < n-2$. See the definition in, for example, \cite[Section 2.1]{An-Huang:2021}. For $f\in  \C^{k,\alpha}_{-q} (M)$, we may also write $f= O^{k,\alpha}(|x|^{-q})$ to emphasize its fall-off rate.

Let  $g$ be a Riemannian metric on $M$. We say that  $(M, g)$ is asymptotically flat (at the rate $q$ and with the \emph{structure at infinity} $(\Phi, x)$) if $g- g_{\mathbb E} \in \C^{2,\alpha}_{-q}(M)$ where $(g_{\mathbb E})_{ij} = \delta_{ij}$ in the chart~$\{x\}$ of $M$ via $\Phi$ from a Cartesian chart of $\mathbb R^n$ and $g_{\mathbb E}$ is smoothly extended to the entire $M$.  The ADM mass of $g$ is defined by 
\begin{align}\label{eq:mass}
	m_{\mathrm ADM}(g) = \frac{1}{2(n-1)\omega_{n-1}} \lim_{r\to \infty} \int_{|x|=r} \sum_{i,j=1}^n \left( \frac{\partial g_{ij}}{\partial x_i} - \frac{\partial g_{ij}}{\partial x_j} \right) \frac{x_j}{|x|} \, d\sigma.
\end{align}
where $d\sigma$ is the $(n-1)$-volume form on $|x|=r$ induced from the Euclidean metric and $\omega_{n-1}$ is the volume of the standard unit sphere $S^{n-1}$.

It is well-known that two structures of infinity of an asymptotically flat manifold $(M, g)$ differ by a rigid motion of $\{x \}$, see~\cite[Corollary 3.2]{Bartnik:1986}. In particular,  translations and rotations are continuous symmetries, generated by the vector fields  
\begin{align*} 
	Z^{(i)} =\partial_i\quad  \mbox{ and }\quad  Z^{(i,j)}= x_i \partial_j - x_j \partial_i \qquad \mbox{ for }i, j=1,\dots, n.
\end{align*}
We can smoothly extend $Z^{(i)}, Z^{(i,j)}$ to the entire $M$ and  denote the vector space of ``Euclidean Killing vectors'' by  
\begin{align} \label{de:Z}
\mathcal Z  = \Span\{ Z^{(i)}\mbox{ and } Z^{(i,j)}\mbox{ for } i, j = 1,\dots, n\}. 
\end{align}

\begin{definition}\label{de:diffeo}
Let  $(M, g)$ be  asymptotically flat.  
\begin{enumerate}
\item Define the subgroup of $\C^{3,\alpha}_{\mathrm{loc}}$ diffeomorphisms of $M\setminus \Omega$ that fix the boundary $\Sigma$ and the structure at infinity as 
\begin{align*}
	\mathscr{D} (M\setminus \Omega)= &\Big\{\,\psi\in\C^{3,\alpha}_{\mathrm{loc}}(M): \psi|_\Sigma = \mathrm{Id}_\Sigma \mbox{ and } \psi(x) - (Ox+a) = O^{3,\alpha}(|x|^{1-q} ) \\
	 &\mbox{ for some constant matrix $O\in SO(n)$ and constant vector $a\in\mathbb R^n$ }\Big \}.
\end{align*}
\item Let $\cX(M\setminus\Omega)$ be the tangent space of $\mathscr D (M\setminus \Omega)$ at the identity map. In other words,  $\cX(M\setminus \Omega)$ consists of vector fields so that\footnote{When $n=3$, the fall-off rate of $X$ is $1-q>0$, and hence $\cX(M\setminus\Omega)$ can be equivalently defined without including $Z=Z^{(i)}$ the translation vectors.}  
\begin{align*}
\cX(M\setminus\Omega)=\Big\{& X\in \C^{3,\alpha}_{\mathrm{loc}}(M\setminus\Omega): X=0 \mbox{ on } \Sigma \mbox{ and }X-Z= O^{3,\a}(|x|^{1-q}) \mbox{ for some $Z\in \mathcal Z$} \Big\}.
\end{align*}
\end{enumerate}
\end{definition}


\subsection{Regge-Teitelboim functional and integral identities}

Recall the static vacuum operator\footnote{In \cite{An-Huang:2021} the notation $S(g,u)$ was used to denote a different (but related) operator, which is $(R'|_g)^*(u)$ below.}
\[
	S(g, u) = (- u \Ric_{g} + \nabla^2_{ g}  u, \Delta_g u).
\]
Let $(\bar{g}, \bar{u})$  be a static vacuum pair; namely $S(\bar g, \bar u) = 0$. The linearization of $S$ at $(\bar g, \bar u)$ is given by 
\begin{align*}
	S'(h, v) =\Big(-\bar u \Ric'(h) + (\nabla^2)'(h) \bar u- v\Ric + \nabla^2 v, \Delta v + (\Delta'(h)) \bar u\Big).
\end{align*}
 We say that $(h,v)$ is a \emph{static vacuum deformation} (at $(\bar g, \bar u)$) if $S'(h, v)=0$. 
 
 Throughout the paper,  we use $S'|_{(g,u)}$ to denote the linearization of S at $(g,u)$ and similar for other differential operators; and we often omit the subscript $|_{(\bar g, \bar u)}$ when linearizing at a static vacuum pair $(\bar g, \bar u)$, as well as the subscript $|_{\bar g}$ in linearizations of geometric operators such as $\Ric_g$, when the context is clear. We refer to Appendix~\ref{se:formula} for explicit formulas of those linearized operators.

Let $\cM (M\setminus \Omega)$ denote the set of \emph{asymptotically flat pairs} in $M\setminus \Omega$, consisting of pairs $(g, u)$ of Riemannian metrics and scalar functions on $M\setminus \Omega$ satisfying 
\begin{align}\label{eq:af-pair}
	\cM (M\setminus \Omega) = \Big\{ (g, u):  (g-g_{\mathbb E}, u-1)\in \C^{2,\alpha}_{-q}(M\setminus \Omega)\Big\}.
\end{align}
The Regge-Teitelboim functional for $(g, u)\in \cM (M\setminus \Omega)$ is defined as 
\[
	\mathscr{F}(g, u) = -2(n-1)\omega_{n-1} m_{\mathrm ADM }(g) + \int_{M\setminus \Omega} u R_g \, \dvol_g.
\]
Although two terms in the definition of $\mathscr F$ are not   individually well-defined for arbitrary $(g, u)\in \cM (M\setminus \Omega)$ (because $R_g$ is not assumed integrable), it is well-known that the functional described above extends to all $(g, u)\in \cM (M\setminus \Omega)$  in a natural way (see e.g. \cite[Theorem 4.1]{Bartnik:2005} and \cite[p. 1660]{Anderson-Jauregui:2016}) which we also explain below: Use the following alternative expression by rewriting the ADM mass surface integral as a volume integral via divergence theorem and rearranging terms:
\begin{align*}
	\mathscr{F}(g, u)&= \int_{M\setminus \Omega}\Big( R_g - \Div_g (\Div_{g_{\mathbb E}} g - d(\tr_{g_{\mathbb E}} g) ) \Big) u \, \dvol_g \\
	&\quad - \int_{M\setminus \Omega} (\Div_{g_{\mathbb E}} g - d(\tr_{g_{\mathbb E}} g)) (\nabla_g u)  \, \dvol_g\\
	&\quad + \int_\Sigma (\Div_{g_{\mathbb E}} g - d (\tr_{g_{\mathbb E}} g))(\nu_g) u\, d\sigma_g
\end{align*}	
where recall that $g_{\mathbb E}$ is a background metric equal to the Euclidean metric on the exterior coordinate chart and $\nu_g$ is the unit normal on $\Sigma$ pointing to infinity.   For integrability of the first volume integral, it is a standard fact that $\Div_g (\Div_{g_{\mathbb E}} g - d(\tr_{g_{\mathbb E}} g) )$ matches the top-order part of $R_g$ and the other terms decay fast enough. The second volume integral is finite because of the assumed decay rates.

We recall the first variation formula. 
\begin{lemma}[{\cite[Proposition 3.7]{Anderson-Khuri:2013} \cite[Lemma 3.1]{Miao:2007} \cite[Proposition 2.2]{Anderson-Jauregui:2016}}]\label{le:first}
Let $(g(s), u(s))$ be a one-parameter  differentiable family of asymptotically flat pairs such that  $(g(0), u(0) ) = (g, u)$ and $(h, v) = (g'(0), u'(0))\in \C^{2,\alpha}_{-q}(M\setminus \Omega)$. Then,
\begin{align*}
	&\left. \ds\right|_{s=0} \mathscr{F}(g(s), u(s)) \\
	&= \int_{M\setminus \Omega} \Big\langle \big(-u \Ric_g+ \nabla^2_g u - (\Delta_g u) g+ \tfrac{1}{2} u R_g g, R_g\big), \big(h, v\big)\Big\rangle_g \, \dvol_g\\
	&\quad +\int_\Sigma \Big\langle \big(uA_g - \nu_g(u) g^\intercal, 2u\big), \big(h^\intercal, H'|_g(h)\big) \Big\rangle_g \, \da_g.
\end{align*}
\end{lemma}

The first variation formula holds for any asymptotically flat pair $(g, u)$. {A static vacuum pair $(\bar g, \bar u)$ is a critical point of the functional~$\mathscr{F}$  among deformations $(h, v)$ where $h$ satisfies $(h^\intercal, H'(h) )=0$.  In the next lemma, we show that if} $R'(h)=0$ in $M\setminus \Omega$, then the deformation $h$ ``preserves the mass'' in the sense that $(\bar g,\bar u)$ is also a critical point of the ADM mass functional among such $h$.

\begin{lemma}\label{le:mass}
Let $(\bar g, \bar u)$ be a static vacuum pair in $M\setminus \Omega$. Suppose $h\in \C^{2,\alpha}_{-q}(M\setminus \Omega)$ satisfies $R'(h)=0$ in $M\setminus \Omega$ and $(h^\intercal, H'(h))=0$ on $\Sigma$. Then 
\[
	\lim_{r\to \infty} \int_{|x|=r} \sum_{i,j=1}^n \left( \frac{\partial h_{ij}}{\partial x_i} - \frac{\partial h_{ij}}{\partial x_j}\right) \frac{x_j}{|x|} \, d\sigma_{\bar g}=0. 
\]
\end{lemma}
\begin{proof}
For any function $v\in \C^{2,\alpha}_{-q}(M\setminus\Omega)$, define $(g(s), u(s) ) =( \bar g, \bar u)+s (h, v)$. Notice that $\left.\ps\right|_{s=0} R_{g(s)} = R'(h)=0$. Thus, 
\[
	 -2(n-1)\omega_{n-1} \left. \ds\right|_{s=0} m_{\mathrm ADM }(g(s))= \left. \ds\right|_{s=0}\mathscr{F}(g(s), u(s)) =0.
\]
In the last equality, we use the first variation formula,  Lemma~\ref{le:first},  and the assumptions on $h$. The integral identity follows \eqref{eq:mass}. 
\end{proof}

{Suppose in addition $(h, v)$ is a static vacuum deformation, i.e. $S'(h, v)=0$, which implies $R'(h)=0$.}  The Laplace equation  for $v$ in $S'(h, v)=0$ gives $\Delta v =O^{0,\alpha} (|x|^{-2-2q})$ and thus by harmonic asymptotics $v(x)= c|x|^{2-n} + O^{2,\alpha}(|x|^{\max\{ -2q, 1-n\}})$ for some real number $c$. We will show  that $v$ also ``preserves the mass'' in that  $c=0$ in Lemma~\ref{le:v-mass} below.

In \cite[Section 3]{An-Huang:2021}, we used the first and second variations of the functional $\mathscr{F}$ to derive several fundamental properties for the static vacuum operators. We list the properties that will be used later in this paper whose proofs in directly extend to our current setting. First, we recall the Green-type identity \cite[Proposition 3.3]{An-Huang:2021}.
\begin{proposition}[Green-type identity]
Let $(g, u)$ be an asymptotically flat pair in $M\setminus \Omega$. For any $(h, v), (k,w)\in \C^{2}_{-q}(M\setminus \Omega)$, we have 
\begin{align}\label{equation:Green}
\begin{split}
	&\int_{M\setminus \Omega}\Big\langle P(h, v),(k,w)\Big\rangle_g  \dvol_g-\int_{M\setminus \Omega}\Big\langle  P(k, w), (h, v)\Big\rangle_g  \dvol_g\\
	&=- \int_\Sigma  \Big\langle  Q(h,v) , \big(k^\intercal, H'|_g(k)\big)\Big\rangle_g  \da_g+\int_\Sigma  \Big\langle Q(k, w), \big(h^\intercal, H'|_g(h)\big)\Big \rangle_g  \da_g.
\end{split}
\end{align}
 The linear differential operators $P, Q$ at $(g, u)$ are defined by\footnote{Technically speaking, the notations $P(h, v), Q(h, v)$ should have the subscript $(g, u)$ to specify their dependence on $(g, u)$, but for the rest of the paper we will only consider the case that $(g, u)$ is an arbitrary but fixed static vacuum pair $(\bar g, \bar u)$, and thus we omit the subscript.}
\begin{align*}
	P(h, v) &= \Big(\big((R'|_g)^*(u) + \tfrac{1}{2} u R_g g\big)', \, R'|_g(h)\Big) - \Big(2\big((R'|_g)^*(u)+ \tfrac{1}{2} u R_g g\big)\circ h, 0\Big) \\
	&\quad + \tfrac{1}{2} (\tr_g h)  \Big((R'|_g)^*(u) + \tfrac{1}{2} u R_g g,\, R_g \Big) \quad \quad  \mbox{ in } M\setminus \Omega
	\\
	Q(h, v)&= \Big( \big(uA_g-\nu_g(u) g^\intercal\big)', \, 2v\Big)  - \Big(2\big(uA_g-\nu_g(u) g^\intercal\big)\circ h^\intercal, 0\Big)\\
	&\quad +\tfrac{1}{2} (\tr_g h^\intercal )\Big(uA_g - \nu_g(u) g^\intercal , \,  2u \Big) \quad \quad \mbox{ on } \Sigma,
\end{align*}
where the prime denotes the linearization at $(g, u)$ and recall the formal $\mathcal L^2$-adjoint operator $(R'|_g)^*(u) := - u\Ric_g + \nabla^2_g u - (\Delta_g u) g$.
\end{proposition}

In the special case that $(g, u)$ is a static vacuum pair $(\bar g, \bar u)$, it is direct to verify that  $P(h, v)=0$ if and only if $S'(h, v)=0$. Thus the Green-type identity leads to the following direct consequence. 
\begin{corollary}[{\cite[Corollary 3.5]{An-Huang:2021}}]\label{co:Green}
Let $(\bar g, \bar u)$ be an asymptotically flat, static vacuum pair in $M\setminus \Omega$.  Suppose that  $(h, v), (k,w)\in \C^{2}_{-q}(M\setminus \Omega)$ are static vacuum deformations and that $(h, v)$ satisfies $h^\intercal=0, H' (h)=0$ on $\Sigma$. Then 
\[
	\int_\Sigma  \Big\langle  Q(h,v) , \big(k^\intercal, H'(k)\big)\Big\rangle_{\bar g}  \da_{\bar g}=0
\]
where $Q(h,v) =  \Big( vA_{\bar g} + \bar u A' (h) - \big(\dnu \bar u + \nu (v)\big) g^\intercal, 2v \Big)$. 
\end{corollary}

Let  $\psi_s$ be a smooth family of diffeomorphisms in  $\mathscr{D} (M\setminus \Omega)$ from Definition~\ref{de:diffeo} (also recall its tangent space $\mathcal{X}(M\setminus \Omega)$ defined there). Computing the first variation of the functional $\mathscr F$ along the family of pull-back pairs $(g(s), u(s))=\psi_s^*(g, u)$ as in \cite[Proposition 3.2]{An-Huang:2021} gives the following identities.

\begin{proposition}[Orthogonality]\label{proposition:cokernel}
Let $(M, g, u)$ be an asymptotically flat triple with $u>0$. For any $X\in \mathcal{X}(M\setminus \Omega)$ and any $(h, v) \in \C^{2}_{-q}(M \setminus \Omega)$,  we have
\begin{align}
	\int_{M\setminus \Omega}  \bigg\langle S(g, u),\, \kappa_0(g, u, X) \bigg \rangle_g\, d\mathrm{vol}_g &=0 \label{equation:cokernel1},
\end{align}
where 
\begin{align}\label{eq:kappa}
	\kappa_0 (g, u, X)&:= \Big( L_X g - \big(\Div_g X+ u^{-1} X(u) \big)g, \, - \Div_g X + u^{-1} X(u)\Big).
\end{align}
Linearizing the identity at a static vacuum pair $(g, u) = (\bar g, \bar u)$ yields
\begin{align}
	\int_{M\setminus \Omega} \bigg\langle S' (h, v),\, \kappa_0 (\bar{g}, \bar{u}, X)\bigg\rangle_{\bar g} \, d\mathrm{vol}_{\bar g} = 0.  \label{equation:cokernel2} 
\end{align}
\end{proposition}
The identities above are important because they will be used to find ``spaces'' complementing to the ranges of the nonlinear operator $S$ and its linearization $S'$.

\subsection{$h$-Killing vectors and the geodesic gauge}\label{se:h1}

In this section we study the situation when a symmetric $(0,2)$-tensor $h$ takes the form $L_X g$. The results here hold in a general Riemannian manifold $(U, g)$, and Corollary~\ref{co:geodesic} below is  used to prove Theorem~\ref{th:trivial}.

Given a symmetric $(0,2)$-tensor $h$, we define the $(1,2)$-tensor $T_h$ by, in local coordinates,  
\[
	(T_h)^i_{jk}= \tfrac{1}{2} (h^i_{j;k} + h^i_{k;j} - h_{jk;}^{\;\;\;\;\; i} )
\]
where the upper indices are all raised by $g$, e.g. $h^i_j=g^{i\ell} h_{\ell j}$. Note that $(T_h)^i_{jk}$ is symmetric in $(j, k)$, and thus we may use  $T_h(V, \cdot)$ to unambiguously denote its contraction with a vector $V$ in the index either $j$ or $k$. We say $X$ is an \emph{$h$-Killing vector} if $L_X g = h$.

\begin{lemma}\label{le:X}
Let $X$ be an $h$-Killing vector field.  Then for any vector $V$, 
\begin{align*}
	\nabla_V (\nabla X) = - R(X, V) + T_h (V,\cdot) 
\end{align*}
where the curvature tensor $R(X, V) := \nabla_X \nabla_V - \nabla_V \nabla_X - \nabla_{[X, V]}$. 
\end{lemma}
\begin{proof}
We prove the identity with respect to a local orthonormal frame $\{ e_1, \dots, e_n \}$ and shall not distinguish upper or lower indices in the following computations.  Write  $V= V_k e_k$ and 
\[
 (\nabla_V (\nabla X) )^i_j = V_k\, g(\nabla_{e_k} \nabla_{e_j} X - \nabla_{\nabla_{e_k} e_j} X, e_i ):= V_k X_{i;jk}.
\]
The desired identity follows from the following identity multiplied by $V_k$:
\begin{align}\label{eq:derivativeX}
	X_{i;jk}=- R_{\ell kj i} X^\ell + \tfrac{1}{2} (h_{ij;k} + h_{ik;j} - h_{jk;i}).
\end{align}
The previous identity is a well-known fact, we include the proof for completeness. By commuting the derivatives, we get 
\begin{align*}
	X_{i;jk} - X_{i;kj} &= R_{kj\ell i} X^\ell\\
	X_{j;ki} - X_{j;ik} &= R_{ik\ell j} X^\ell\\
	X_{k;ij} - X_{k;ji} &= R_{ji\ell k} X^\ell. 
\end{align*}
Adding the first two identities and then subtracting the third one, we get
\begin{align*}
	&X_{i;jk} - X_{i;kj} + X_{j;ki} - X_{j;ik} - X_{k;ij} + X_{k;ji} \\
	&=( R_{kj\ell i} + R_{ik\ell j} - R_{ji\ell k} )X^\ell\\
	& = 2R_{ij\ell k} X^\ell =-2 R_{\ell kj i} X^\ell
\end{align*}
where we use the Bianchi identity to the curvature terms. Rearranging the terms in the left hand side above gives
\begin{align*}
	&2 X_{i;jk} - (X_{i;jk} + X_{j;ik} ) - (X_{i;kj} + X_{k;ij} )+ (X_{j;ki} + X_{k;ji}) \\
	&=2X_{i;jk} - h_{ij;k} - h_{ik;j} + h_{jk;i}.
\end{align*}
This gives \eqref{eq:derivativeX}.
\end{proof}

For a given $h$, in general there does not necessarily exist a corresponding $h$-Killing vector field. (For example, when $h=0$, an $h$-Killing vector field is just a Killing vector field, and a generic Riemannian manifold does not admit any Killing vector field.)  Nevertheless, the next lemma says that it is still possible to find an $X$ such that $h$ and $L_X \bar g$ are equal when they are both contracted with a parallel vector~$V$. It uses an ODE argument motivated by the work of Nomizu~\cite{Nomizu:1960}.

\begin{lemma}\label{le:V}
Let $(U, g)$ be a Riemannian manifold whose boundary $\partial U$ is an embedded hypersurface. Let {$3\le k \le \infty$}, $\Sigma$ be an open subset of $\partial U$, and $h$ be a symmetric $(0, 2$)-tensor in~$U$.  Suppose $g, \Sigma\in {\C^{k}}, h\in {\C^{k-1}}$ (or analytic) in coordinate charts containing $\Sigma$. Let $V\in \C^{k}$ (or analytic) be a complete vector field transverse to $\Sigma$ satisfying $\nabla_V V=0$ in~$U$. Then there is a vector field $X\in  {\C^{k-2}}$ (or analytic) in a collar neighborhood of $\Sigma$ such that  $X=0$ on $\Sigma$ and 
\begin{align}\label{eq:V}
	L_X g (V, \cdot) = h(V, \cdot) \quad \mbox{ in the collar neighborhood of } \Sigma. 
\end{align}
\end{lemma}
\begin{proof}
Since $V$ is complete, given $p\in \Sigma$ we let $\gamma(t)$ be the integral curve of $V$ emitting from $p$, i.e. $\gamma(0)=p$ and $\gamma'(t)=V$.  Let  $\{ e_1, \dots, e_n \}$  be a local orthonormal frame such that $e_n$ is a unit normal to $\Sigma$. 

Consider the first-order linear ODE system for the pair $(X,\omega)$ consisting of a vector field $X$ and a $(1,1)$-tensor $\omega$ along $\gamma(t)$:
\begin{align*}
	\nabla_V X&=\omega (V)\\
	\nabla_V \omega &=  -R(X, V)  + T_h (V, \cdot).
\end{align*}
(We remark that  Lemma~\ref{le:X} implies an $h$-Killing vector $X$  satisfies the ODE system with $\omega(e_i) = \nabla_{e_i} X$.)  We rewrite  the above ODE system in the local orthonormal frame:
\begin{align} \label{eq:system}
\begin{split}
	X_{i;j} V_j &= \omega_{ij} V_j\\
	\omega_{ij;\ell} V_\ell &= - R_{k\ell ji} X_k V_\ell + \tfrac{1}{2} \left( h_{ij,\ell} + h_{i\ell, j} - h_{j\ell, i} \right) V_\ell
\end{split}
\end{align}
where $\omega_{ij} =\omega^i_j$ (the first index is lowered by $g$).

We choose the initial conditions for $X_i$ and  $\omega_{ij}$ at $p\in \Sigma$:
\begin{align}\label{eq:initial}
X_i= 0, \quad 
\omega_{ij}V_i V_j = \tfrac{1}{2} h_{ij}V_i V_j, \quad \omega_{ia}V_i=\omega_{ai}V_i   = h_{ai}V_i, \quad \omega_{ab} = 0 
\end{align}   
where the indices $i,j=1,\dots,n,~a, b=1, \dots, n-1$. Since the coefficients of the ODE are ${\C^{k-2}}$ (or analytic) in $p$,  the vector field $X$ and $\omega$ are defined everywhere in the collar neighborhood of $\Sigma$ by varying $p$ and is ${\C^{k-2}}$ (or analytic) by smooth dependence of the ODE (or Cauchy-Kovalevskaya Theorem).

We first show that in a collar neighborhood of $\Sigma$:
\begin{align}\label{eq:omegaV}
	\omega_{ij}  V_i V_j = \tfrac{1}{2} h_{ij} V_i V_j.
\end{align}
Note that $\omega_{ij}+\omega_{ji}-h_{ij}$ is constant along $\gamma(t)$ because $(\omega_{ij}  + \omega_{ji} - h_{ij} )_{;\ell}V^\ell =0$ by \eqref{eq:system} and symmetry of the curvature tensor. Since $(\omega_{ij}+ \omega_{ji} - h_{ij})V_i V_j = (2\omega_{ij} - h_{ij} )V_i V_j=0$  at $p$ by the initial conditions at $p$, it proves \eqref{eq:omegaV}.

To prove \eqref{eq:V},  observe that $( X_{i;j} + X_{j;i} - h_{ij} )V_j $ satisfies a first-order linear ODE along~$\gamma(t)$, motivated by the result of~\cite[Lemma 2.6]{Ionescu-Klainerman:2013}:
\begin{align*}
	&\big(( X_{i;j} + X_{j;i} -h_{ij} )V_j\big)_{;\ell} V_\ell \\
	&= (\omega_{ij;\ell}+ X_{j;i\ell} - h_{ij;\ell})V_j V_\ell\\
	&=\big (- R_{k\ell ji } X_k +  h_{ij;\ell} -\tfrac{1}{2}  h_{j\ell;i}\big) V_j V_\ell  + X_{j;\ell i } V_j V_\ell +  R_{\ell i k j} X_k V_j V_\ell  - h_{ij;\ell} V_j V_\ell \\
	&= -\tfrac{1}{2}  h_{j\ell;i} V_j V_\ell +  X_{j;\ell i} V_j V_\ell  \\
	&= -\tfrac{1}{2}  h_{j\ell;i} V_j V_\ell +  (X_{j;\ell} V_j V_\ell )_{;i}-X_{j;\ell}(V_{j;i}V_{\ell}+V_jV_{\ell;i})\\
	&= -\tfrac{1}{2}  h_{j\ell;i} V_j V_\ell +  (\omega_{j\ell} V_j V_\ell )_{;i}-(X_{j;\ell}+X_{\ell;j})V_{j;i}V_{\ell}\\
	&= -\tfrac{1}{2}  h_{j\ell;i} V_j V_\ell +  \tfrac{1}{2}(h_{j\ell} V_j V_\ell )_{;i}-(X_{j;\ell}+X_{\ell;j})V_{j;i}V_{\ell}\\
	&=-(X_{j;\ell}+X_{\ell;j}-h_{j\ell})V_{\ell}V_{j;i}
\end{align*}
where in  the third equality we use the Bianchi identity 
\[
	( - R_{k\ell ji } + R_{\ell i k j} ) V_j V_\ell =( R_{k\ell ij} + R_{\ell i kj} ) V_j V_\ell= -R_{ik\ell j} V_j V_\ell =0
\]
and in the second-to-the-last equality we use \eqref{eq:omegaV}. Thus, we have shown that $(X_{i;j} + X_{j;i} -h_{ij} )V_j$ satisfies the first-order linear ODE. Our initial conditions \eqref{eq:initial} imply that at $p$:
\begin{align*}
	(X_{i;j} + X_{j;i} -h_{ij} )V_iV_j &= (2\omega_{ij}- h_{ij} )V_iV_j = 0 \\
	(X_{a; j} + X_{j;a} - h_{aj} )V_j &=  (\omega_{aj} -h_{aj})V_j=0 \qquad \mbox{ for  $a = 1, \dots, n-1$}
\end{align*}
where we use that $X_{j; a}=0$ at $p$ because $X$ is identically zero on $\Sigma$. Since $(X_{i;j} + X_{j;i} -h_{ij} )V_j=0$ at $p$, it is identically zero along the curve $\gamma(t)$, and thus it is identically zero in a collar neighborhood of $\Sigma$. 

\end{proof}

Let $\nu$ be a unit normal vector field to $\Sigma$. We can extend $\nu$ parallelly in a collar neighborhood of $\Sigma$. We say that a symmetric $(0,2)$-tensor $h$ in $U$ satisfies  the \emph{geodesic gauge} in a collar neighborhood of $\Sigma$ if, in the collar neighborhood, $h(\nu, \cdot)= 0$.

By Lemma~\ref{le:V} and letting $V= \nu$ there, we can give an alternative proof to the existence of geodesic gauge in \cite[Lemma 2.5]{An-Huang:2021}, and in particular, we obtain an \emph{analytic} vector field $X$ if the metric is analytic. 

\begin{corollary}[Geodesic gauge]\label{co:geodesic}
Let $(U, g)$ be a Riemannian manifold whose $\partial U$ is an embedded hypersurface. Let $\Sigma$ be an open subset of $\partial U$ and $h$ be a symmetric $(0, 2$)-tensor in $U$.
\begin{enumerate}

\item  Let {$3\le k \le \infty$}. Suppose $g, \Sigma \in {\C^{k}}, h \in {\C^{k-1}}$ in some coordinate charts containing $\Sigma$.  Then there is a vector field $X\in {\C^{k-2}}$ in a collar neighborhood of $\Sigma$ such that  $X=0$ on $\Sigma$ and 
\begin{align*}
	L_X g (\nu, \cdot) = h(\nu, \cdot) \quad \mbox{ in the collar neighborhood of } \Sigma. 
\end{align*}

\item Suppose $\Sigma$ is an analytic hypersurface and $g, h$ are analytic in some coordinate charts containing $\Sigma$.  Then there is an analytic vector field $X$ in a collar neighborhood of $\Sigma$ such that  $X=0$ on $\Sigma$ and 
\begin{align*}
	L_X g (\nu, \cdot) = h(\nu, \cdot) \quad \mbox{ in the collar neighborhood of } \Sigma. 
\end{align*}
\end{enumerate}
\end{corollary}

\section{Static-harmonic gauge and orthogonal gauge}
\label{se:gauge}

Recall the operator $T$ defined in \eqref{eq:bdv}. As already mentioned in Section~\ref{se:intro},  if $(g, u)$ solves $T(g, u) = (0, 0, \tau, \phi)$, then any $\psi$ in the diffeomorphism group $\mathscr{D}(M\setminus \Omega)$ defined in Definition~\ref{de:diffeo} gives rise to another solution $(\psi^* g, \psi^* u)$. To overcome the infinite-dimensional ``kernel'' of $T$, one would like to introduce suitable ``gauges''.

\subsection{The gauges}

Fix a static vacuum pair $(\bar g, \bar u)$ with $\bar u>0$. For any pair $(g, u)$ of a Riemannian metric and a scalar function, we use the Bianchi operator $\b_{\bar g}$ (see \eqref{eq:Bianchi}) to define the covector 
\begin{align*}
	\Ga(g, u) =\b_{\bar g} g+\bar u^{-2}udu-{\bar u}^{-1}g( \nabla_{\bar g}\bar u, \cdot ).
\end{align*}
 Note that of course $\Ga(\bar{g}, \bar{u})=0$. We use  $\Ga'(h, v)$ to denote the linearization of $\Ga(g, u)$ at $(\bar{g}, \bar{u})$ and thus
\[
	\Ga'(h,v) = \b h+ \bar u^{-1} dv + \bar{u}^{-2} v d\bar{u} - \bar{u}^{-1} h(\nabla\bar{u}, \cdot )
\]
where the Bianchi operator and covariant derivatives are with respect to $\bar g$. 

For the rest of this section, we omit the subscript $\bar g$ when computing differential operators with respect to $\bar g$, as well as the subscript $(\bar g, \bar u)$ when linearizing at $(\bar g, \bar u)$. 
\begin{lemma}
Let $(\bar g, \bar u)$ be a static vacuum pair with $\bar u>0$. Then for any vector field $X$, 
\begin{align} \label{eq:gauge}
	\Ga'(L_X \bar g, X(\bar u))=-\D X-\bar u^{-1}\nabla X(\nabla \bar u,\cdot)+\bar u^{-2}X(\bar u)d \bar u =: \Gamma(X).
\end{align}
(Here and the in rest of the paper, we slightly abuse the notation and blur the distinction of a vector and its  dual covector with respect to $\bar g$ when the context is clear.) 
\end{lemma}
\begin{proof}

By the linearization formula,
\begin{align}\label{eq:G}
\begin{split}
\Ga' (L_X \bar g, X(\bar u)) &=\beta (L_X  \bar g) + \bar u^{-1} d(X( \bar u))  - \bar u^{-1} L_X \bar g (\nabla  \bar u, \cdot)+ \bar u^{-2} X( \bar u) d \bar u \\
&=-\D X-\bar u^{-1}\nabla X (\nabla \bar u,\cdot)+\bar u^{-2}X(\bar u)d \bar u + ( - \Ric + \bar u^{-1} \nabla^2 \bar u) (X, \cdot )
\end{split}
\end{align}
where we use \eqref{eq:laplace-beta} for $\beta (L_X  \bar g) = -\Delta X - \Ric(X,\cdot)$ and 
\begin{align}\label{eq:com}
 d(X( \bar u))  - L_X g (\nabla \bar u ,\cdot )&= \nabla^2  \bar u (X, \cdot ) -  \nabla X (\nabla \bar u,\cdot).
\end{align}
Since $(\bar g, \bar u)$ is static vacuum, we can drop the term $ - \Ric + \bar u^{-1} \nabla^2\bar u =0$. 
\end{proof}

Using the operator $\Gamma$ defined in \eqref{eq:gauge}, we define the ``gauged'' subspace of $\mathcal X(M\setminus \Omega)$ from Definition~\ref{de:diffeo}. 
\begin{definition} 
 $(M, \bar g, \bar u)$ be an asymptotically flat, static vacuum triple with $\bar u>0$.  Define $\cX^\Ga(M\setminus\Omega)$ to be the subspace of $\cX(M\setminus\Omega)$ as 
\[
	\cX^\Ga(M\setminus\Omega)=\big\{ X \in \cX(M\setminus\Omega): \Gamma (X)=0 \mbox{ in } M\setminus\Omega \big\}.
\]
\end{definition}

We show below that $\cX^\Ga(M\setminus\Omega)$ is finite-dimensional with the dimension
\[
	N = n+\frac{n(n-1)}{2},
\]
which is the same as the dimension of the space of Euclidean Killing vectors $\mathcal Z$ defined in~\eqref{de:Z}. We begin with a fundamental PDE lemma on the special structure of the operator $\Gamma$. For a given asymptotically flat pair $(g, u)$ with $u>0$, we define the operator on vectors by 
\begin{align} \label{eq:Gamma}
	\Gamma_{(g, u)} (X)= -\Delta_g X - u^{-1}\nabla_g X(\nabla_g u,\cdot) + u^{-2} X(u) du.
\end{align}
When $(g, u)$ is a static vacuum pair $(\bar g, \bar u)$,   $\Gamma_{(\bar g, \bar u)} (X)$ is exactly $\Gamma(X)$ as defined in~\eqref{eq:gauge}.

Recall \eqref{eq:af-pair} that $\cM(M\setminus \Omega)$ consists of asymptotically flat pairs of fall-off rate $q$. 

\begin{lemma}\label{PDE}
Let  $(g, u)$ be a pair of a Riemannian metric and a scalar function satisfying $(g-g_{\mathbb E}, u-1)\in \C^{k,\alpha}_{-q}(M\setminus \Omega)$ with $u>0$ and $\delta$ be a real number and $k\ge 1$ be an integer. Consider 
\[
	\Gamma_{(g, u)} : \big\{ X\in \C^{k,\a}_{\delta}(M\setminus\Omega): X=0 \mbox{ on } \Sigma \big\} \longrightarrow \C^{k-2,\a}_{\delta-2}(M\setminus\Omega).
\] 
Then the following holds:
\begin{enumerate}
\item \label{item:0} For $2-n<\delta < 0$, the map is an isomorphism. Therefore, for any fixed boundary value $Z\in \C^{k,\a}(\Sigma)$, the following map is bijective
\[
	\Gamma_{(g, u)} : \big\{ X\in \C^{k,\a}_{\delta}(M\setminus\Omega): X=Z \mbox{ on } \Sigma \big\} \longrightarrow \C^{k-2,\a}_{\delta-2}(M\setminus\Omega).
\] 
\item \label{item:n} For $0<\delta <1$, the map is surjective and the kernel space is $n$-dimensional, spanned by $\{ V^{(1)}, \dots, V^{(n)}\}$ where $V^{(i)} = \pa_i + O^{k,\alpha}(|x|^{-q})$. 
\end{enumerate}
\end{lemma}
\begin{proof} 

 Note that $\Gamma_{( g, u)}$ has the same Fredholm index as the Laplace-Beltrami operator $\Delta_{g}$. It is a standard fact that the Fredholm index for $X\mapsto \Delta_{g}X$ is $0$ in the first case that $2-n<\delta < 0$ and $n$ in the second case that $0<\delta <1$, at least for asymptotically flat manifolds without boundary (see e.g. \cite[p. 16]{CSCB:1979}, \cite[p. 673]{Bartnik:1986}). For the boundary value problem, one can  find general results in \cite{Lockhart-McOwen:1985}, and we also include a proof below for our setting.  In the following proof, the covariant derivative, inner product, volume form are taken with with respect to $g$, and we often omit the subscripts.

We discuss the case $2-n<\delta < 0$, and the other case of decay rate follows a similar argument.  Suppose $\Delta X=0$. By harmonic expansion $X = O^{k,\alpha}(|x|^{2-n})$ (see, e.g. \cite[Theorem 1.17]{Bartnik:1986}). Then $0=-\int_{M\setminus \Omega} X\Delta X \, \dvol  = \int_{M\setminus \Omega} |\nabla X|^2\, \dvol $ because the boundary term on $\Sigma$ vanishes from the Dirichlet boundary condition and the boundary term at infinity vanishes from the decay rate. It implies that  $\Delta$ has trivial kernel. We show that $\Delta$ is surjective: Using\footnote{Let $\xi(x)$ be a positive smooth weight function such that $\xi(x) = |x|^{2-\delta'-\tfrac{n}{2}}$ outside a compact subset of $M$. The $\cL^2_{\delta'-2}(M\setminus \Omega)$-norm is defined as the sum of the usual $\cL^2$-norm on a compact subset of $B\subset M$ and the weighted norm in the asymptotically flat end $M\setminus B$:
\[
	{\| u \|_{\cL^2_{\delta'-2}(M\setminus \Omega)}= \left(\int_{M\setminus \Omega} \big(|u(x)| \xi(x) \big)^2 \, \dvol\right)^{\frac{1}{2}}.}
\]
} that $\C^{k-2,\alpha}_{\delta-2} (M\setminus \Omega) \subset \cL^2_{\delta'-2}(M\setminus \Omega)$ for any $\delta'$ slightly larger than $\delta$, we write $\C^{k-2,\alpha}_{\delta-2} = \Range \Delta \oplus \mathcal K$ where $\mathcal K$ is $\cL^2_{\delta'-2}$-orthogonal to  $ \Range \Delta$. That is,  $Z\in \mathcal K$ if for all $X\in \C^{k,\alpha}_{\delta}$ with $X=0$ on $\Sigma$, 
\begin{align}\label{eq:cokernel}
	0= \int_{M\setminus \Omega} \xi^2 Z\Delta X \, \dvol.
\end{align}
 Considering compactly supported $X$ yields that $\xi^2 Z$ weakly solves $\Delta (\xi^2  Z)=0$ and thus $ \xi^2 Z= O^{k,\alpha}(|x|^{2-n})$ by elliptic regularity. Integrating \eqref{eq:cokernel} by part, invoking $X=0$ on $\Sigma$, and letting $\nabla_\nu X$ be arbitrary on $\Sigma$ implies that $ \xi^2 Z=0$ on $\Sigma$. Then we can as above conclude that  $Z\equiv 0$ in $M\setminus \Omega$, and thus $\Delta$ is surjective.

For Item~\eqref{item:0}, $\Gamma_{( g, u)}$ has Fredholm index $0$. It suffices to show that the kernel is trivial. Observe the $g$-inner product:
\begin{align*}
	\left\langle X,  \Gamma_{( g, u)}(X) \right\rangle= -\tfrac{1}{2} \Delta |X|^2 - \tfrac{1}{2} u^{-1} {\nabla u} \cdot \nabla |X|_g^2 + |\nabla X|^2 +  u^{-2} (X ( u))^2.
\end{align*}	
Thus, if  $\Gamma_{(g,u)}(X)=0$, then 
\[
 \tfrac{1}{2} \Delta |X|^2 + \tfrac{1}{2} u^{-1} {\nabla u} \cdot \nabla |X|^2 \ge 0.
\]
By strong maximum principal and using $X=0$ on $\Sigma$ and $X\to 0$ at infinity,  $X$ is identically zero. The statement about the general boundary value $Z$ is standard.

For Item~\eqref{item:n}, $\Gamma_{(g,u)}$ has the Fredholm index $n$. By harmonic expansion (e.g. \cite[Theorem 1.17]{Bartnik:1986}), if $\Gamma_{(g,u)}(X)=0$, then 
\[
	X = c^i\pa_i + O(|x|^{-q})
\]
 for some constants $c_i$. It implies  that the kernel is at most $n$-dimensional because  by Item~\eqref{item:0} if $c_1=\dots=c_n=0$, then $X$ is identically zero. Since the Fredholm index is $n$, the dimension of the kernel must be exactly $n$.

\end{proof}

\begin{corollary}\label{cor:dimension} $\dim \cX^\Ga(M\setminus\Omega)=N$. 
\end{corollary}
\begin{proof}
Recall the basis $Z^{(i)}, Z^{(i, j)}$ of the space $\mathcal Z$ defined in \eqref{de:Z}, and, outside a compact set of $M$,
\[
	Z^{(i)} =\partial_i\quad  \mbox{ or }\quad  Z^{(i,j)}= x_i \partial_j - x_j \partial_i \qquad \mbox{ for }i, j=1,\dots, n. 
\]
We compute $\Gamma (Z^{(i)})= O^{1,\alpha}(|x|^{-2-q})$. 
By  Item~\eqref{item:0} of Lemma~\ref{PDE}, there is a unique   $Y= O^{3,\alpha}(|x|^{-q})$  such that 
\begin{align*}
	\Gamma(Y) &=-\Gamma (Z^{(i)}) \quad \mbox{ in } M\setminus\Omega\\
	Y &= -Z^{(i)} \quad \mbox{ on } \Sigma.
\end{align*}
Then $W^{(i)}:= Y+Z^{(i)} \in \cX^\Ga(M\setminus\Omega)$. Similarly, we compute $\Gamma(Z^{(i, j)}) = O^{1,\alpha}(|x|^{-q-1})$ and can solve $W^{(i, j)}\in \cX^\Ga(M\setminus \Omega)$ such that $W^{(i, j)} - Z^{(i, j)}\in O^{3,\alpha}(|x|^{1-q})$. 

It is clear that $W^{(i)}$ and $W^{(i, j)}$ are linearly independent. It remains to show that they span $\cX^\Ga(M\setminus\Omega)$, and thus $\Dim \cX^\Ga(M\setminus\Omega)=N$. Let $X\in  \cX^\Ga(M\setminus\Omega)$. Then $X-W\in O^{3,\alpha}(|x|^{1-q})$ where $W$ is a linear combination of $W^{(i)}$ and $W^{(i,j)}$. We separate the discussions into the case $1-q<0$ and the case $0<1-q<1$ (the latter case  occurs only when $n=3$). 

\begin{itemize}
\item If $1-q<0$, using $\Gamma(X-W)=0$ and Item~\eqref{item:0} of Lemma~\ref{PDE}, we obtain $X=W$.  

\item If $0<1-q <1$ (when $n=3$),  by  Item~\eqref{item:n} of Lemma~\ref{PDE}, we have
\[
	X - W =  c_i W^{(i)} + O^{3,\alpha}(|x|^{-q}). 
\]		
Because $\Gamma (X-W-c_iW^{(i)})=0$, as in the first case we conclude that $X= W+c_i W^{(i)}$. 
\end{itemize}
\end{proof}

After introducing $\Ga$ and the gauged space of vectors $\cX^\Ga$, we now define the gauge conditions in solving the boundary value problem for~$T$. 

\begin{definition}\label{de:gauge}
Let  $(M, \bar{g}, \bar{u})$ be an asymptotically flat, static vacuum triple with $\bar u>0$. 
\begin{enumerate}
\item We say that $(g,u)$ satisfies the {\it static-harmonic gauge} (with respect to $(\bar g,\bar u)$) in  $M\setminus\Omega$ if $\Ga(g,u)=0$ in $M\setminus\Omega$.
\item Fix a positive scalar function $\rho$ in $M$ with $\rho=|x|^{-2}$ on the end $M\setminus K$. We say that $(g, u)$ satisfies an {\it orthogonal gauge} (with respect to $(\bar g, \bar u)$ and $\rho$) in $M\setminus\Omega$ if, for all  $X\in \cX^\Ga(M\setminus \Omega)$, 
\[
	\int_{M\setminus\Omega} \Big\langle \big( (g, u) - (\bar g, \bar u)\big) , \big(L_X \bar g, X(\bar u) \big)\Big \rangle \rho \, \dvol = 0,
\] 
where the inner product and volume form are of $\bar g$. 

\end{enumerate}

\end{definition}

\begin{remark}
\begin{enumerate}
\item  When $(\bar{g}, \bar{u}) = (g_{\mathbb{E}}, 1)$ is the Euclidean pair, $\Ga(g,u)=\b_{g_{\mathbb{E}}} g+ udu$. While the above definition of static-harmonic gauge does not recover our prior definition $\b_{g_{\mathbb{E}}} g+ du=0$ in \cite[Definition 4.2]{An-Huang:2021}, both conditions give  the same \emph{linearized} condition which is sufficient. See also  \cite[Remark 4.3]{An-Huang:2021}. 
\item  If $\bar u>0$, we can define the warped product metrics ${\bf \bar{g}} = \pm \bar{u}^2 dt^2 +\bar{g}$ and ${\bf g} = \pm u^2 dt^2 + g$ on $\mathbb{R}  \times M$. The condition $\Ga(g,u)=0$ is equivalent to requiring that $\bf g$ satisfies the harmonic gauge $\b_{\bf \bar{g}} {\bf g}=0$. See Proposition~\ref{pr:harmonic}. 
\end{enumerate}
 \end{remark}
 
We will soon see in Section~\ref{se:operator} below the static-harmonic gauge will be used to obtain ellipticity for the boundary value problem. The following lemma gives a justification why an orthogonal gauge is needed.

\begin{lemma}\label{le:diffeo}
Let $(M, \bar g, \bar u)$ be an asymptotically flat, static vacuum triple with $\bar u>0$. There is an open neighborhood $\mathcal{U}\subset \cM(M\setminus\Omega)$ of $(\bar g, \bar u)$ and an open neighborhood $\mathscr{D}_0 \subset \mathscr D (M\setminus\Omega)$ of $\mathrm{Id}_{M\setminus \Omega}$ such that for any $(g, u) \in \mathcal U$, there is a unique $\psi \in \mathscr{D}_0$ such that $(\psi^* g, \psi^* u)$ satisfies both the static-harmonic and orthogonal gauge. 
\end{lemma}
\begin{proof}

Denote the weighted $\cL^2$-inner product:
\[
\big\langle (h, v), (k, w)\big\rangle_{\cL^2_\rho} = \int_{M\setminus \Omega} (h, v)\cdot (k, w)\rho \, \dvol
\] 
where the inner product and volume form are with respect to $\bar g$ and  $\rho$ is the weight function in Definition~\ref{de:gauge}. By Corollary~\ref{cor:dimension}, let $\{ X^{(1)}, \dots, X^{(N)} \}$ be an orthonormal basis of $\cX^\Ga(M\setminus \Omega)$ with respect to the $\cL^2_\rho$-inner product in the sense that 
\[
	\Big\langle \big(L_{X^{(i)}} \bar g, X^{(i)} (\bar u) \big ) ,\big( L_{X^{(j)}} \bar g, X^{(j)} (\bar u)\big) \Big\rangle_{\cL^2_\rho} =\delta_{ij}.
\] 

Recall $\mathscr D (M\setminus\Omega)$ and $\cM(M\setminus\Omega)$ defined in Definition~\ref{de:diffeo} and \eqref{eq:af-pair}, respectively. Consider the differentiable map 
\begin{align*}
&F:\mathscr D (M\setminus\Omega)\times \cM(M\setminus\Omega) \longrightarrow \C^{1,\a}_{-q-1}(M\setminus \Omega)\times\bR^N\\
&\qquad \qquad F(\psi,(g,u))=\big({\Ga}(\psi^*g,\psi^*u),(b_1,...,b_N)\big)
\end{align*}
where each number $b_i=\Big\langle \big(\psi^*g-\bar g,\psi^*u-\bar u\big),\big(L_{X^{(i)}}\bar g,X^{(i)}(\bar u)\big)\Big\rangle_{\cL^2_\rho}$. 
Linearizing $F$ in the first argument at $(\mathrm{Id}_{M\setminus\Omega},(\bar g,\bar u))$ gives
\bes
\begin{split}
&D_1F:\cX(M\setminus \Omega)\longrightarrow \C^{1,\a}_{-q-1}(M\setminus \Omega)\times\bR^N\\
&D_1F(X)=\Big({\Ga}'(L_X\bar g,X(\bar u)),\big(c_1(X),...,c_N(X)\big)\Big)
\end{split}
\ees
with $c_i(X)=\Big\langle \big(L_X\bar g,X(\bar u)\big), \big(L_{X^{(i)}}\bar g,X^{(i)}(\bar u)\big)\Big\rangle_{\cL^2_\rho}$.  

Note that $F\big(\mathrm{Id}_{M\setminus \Omega}, (\bar g, \bar u)\big)= (0, 0)$. Once we show that $D_1F$ is an isomorphism, the lemma follows from  implicit function theorem for Banach manifolds (see e.g. \cite[3.3.13 Proposition]{Abraham-Marsden-Ratiu:1983}).  If $D_1 F(X)=0$, then $\Gamma(X)=\Ga' (L_X\bar g,X(\bar u))=0$ and $c_i(X)=0$ for all $i$. It implies that $X\in \cX^\Ga(M\setminus \Omega)$,  and thus $X\equiv 0$ in $M\setminus\Omega$. To see that $D_1 F$ is surjective, for any covector $Z\in \C^{1,\alpha}_{-q-1} (M\setminus \Omega)$ and constants $a_1, \dots, a_N$, there exists $Y\in \cX(M\setminus \Omega)$ solving $\Ga'(L_Y \bar g, Y(\bar u))= \Gamma (Y) = Z$ by Lemma~\ref{PDE}. Let $X= Y+ \big(a_1-c_1(Y)\big) X^{(1)}+\dots + \big(a_N-c_N(Y)\big) X^{(N)}$. Then we have $D_1 F(X) =\big (Z, (a_1,\dots, a_N)\big)$.
\end{proof}

\subsection{The gauged operator}\label{se:operator}

Consider the operator $T$ defined in \eqref{eq:bdv} on the manifold $M\setminus\Omega$:
\begin{align*}
&T: \cM(M\setminus \Omega) \longrightarrow \C^{0,\alpha}_{-q-2}(M\setminus\Omega)\times  \mathcal{B}(\Sigma) \\
	&T(g, u) := \begin{array}{l}  
 \left\{ \begin{array}{l} -u \Ric_g  + \nabla^2_g u \\ \Delta_g u 
	\end{array} \right. \quad \mbox{ in } M\setminus\Omega\\
\left\{ \begin{array}{l} g^\intercal \\ H_g 
	\end{array} \right. \quad \mbox{ on } \Sigma.
	\end{array}
\end{align*}
where $\mathcal B(\Sigma)$ denotes the space of pairs $(\tau, \phi)$ where $\tau \in \C^{2,\alpha}(\Sigma)$ is a symmetric $(0, 2)$-tensor on $\Sigma$ and $\phi\in \C^{1,\alpha}(\Sigma)$ is a scalar-valued function on $\Sigma$. 

Define the \emph{gauged} operator 
\[
	T^\Ga: \cM (M\setminus\Omega) \longrightarrow \C^{0,\alpha}_{-q-2}(M\setminus\Omega)\times \C^{1,\alpha}(\Sigma) \times \mathcal{B}(\Sigma)
\]
and 
\begin{align}\label{rsv}
 T^{\Ga}(g, u)=  \begin{array}{l} \left\{ \begin{array}{l}
-u \Ric_{ g}+\nabla^2_{ g} { u}-u\DD_g {\Ga}(g,u)\\
\D_{g} { u}-{\Ga}(g,u) (\nabla_g u)
\end{array}\right. \quad  {\rm in }~M\setminus \Omega \\
\left\{ \begin{array}{l}
	\Ga(g, u) \\
	g^\intercal\\
	H_g
\end{array}
\right.\quad \mbox{ on } \Sigma
\end{array}.
\end{align}
Recall the notation $\mathcal D_g X = \frac{1}{2} L_X g$ from \eqref{eq:Lie}. 

Obviously the gauged operator $T^\Ga$ becomes the operator $T$  when the ``gauge'' term $\Ga(g, u)$ vanishes in $M\setminus\Omega$, i.e. $(g, u)$ satisfies the static-harmonic gauge. The following lemma relates our desired boundary value problem for $T$  to solving $T^\Ga (g, u)$.  
\begin{lemma}\label{rsv-to-sv}
Let $(M, \bar{g}, \bar{u})$ be an asymptotically flat, static vacuum triple  with $\bar{u}>0$. There is an open neighborhood $\cU$ of $(\bar g,\bar u)$ in $\cM(M\setminus \Omega)$ such that if $(g,u)\in\cU$ and $T^\Ga(g, u) = (0, 0, 0, \tau, \phi)$, then $\Ga(g, u)=0$ in $M\setminus \Omega$ and thus $T(g, u)= (0, 0, \tau, \phi)$. 

\end{lemma}
\begin{proof}

In the following computations,  the volume measure,  geometric operators, as well as $\beta_g, \mathcal D_g$,  are all computed with respect to $g$, and we omit the subscripts $g$ for better readability. Recall the integral identity  \eqref{equation:cokernel1} says the following two terms are $\cL^2$-orthogonal, for any $X\in \mathcal X(M\setminus\Omega)$,
\begin{align*}
S(g, u) &= (-u \Ric+\nabla^2 { u}, \Delta u)\\
\kappa_0(g, u, X) &= \big(2\beta^* X - u^{-1} X(u) g, \, -\Div X + u^{-1} X(u)\big)
\end{align*}
where we re-expressed  $\kappa_0(g, u, X)$ from \eqref{eq:kappa} using the operator $\beta^* = \beta_g^*$, defined in \eqref{eq:Bianchi*}.

Below we write the covector $\Ga =\Ga(g, u)$ for short. Since $(g, u)$ solves $T^\Ga(g, u)=0$, we can substitute $S(g, u)= \big( u \DD \Ga ,  \Ga (\nabla u)\big)$ to get 
\begin{align*}
	0&=\int_{M\setminus\Omega} \Big\langle  S(g, u), \kappa_0 (g, u, X) \Big\rangle \, \dvol \\
	 &=\int_{M\setminus\Omega} \Big\langle\big(u \DD \Ga ,  \Ga  (\nabla u)\big), \big( 2 \beta^* X - u^{-1} X(u) g, \, -\Div X + u^{-1} X(u) \big) \Big \rangle \, \dvol.
\end{align*}	
 Applying integration by parts and varying among compactly supported  $X$, we obtain that $\Ga$ weakly solves
\begin{align*}
	0&=2 \b \big( u\DD \Ga ) -( \Div \Ga) du + d (\Ga(\nabla u) ) + u^{-1} \Ga(\nabla u) du\\
	&=2u \b \DD \Ga- L_{\Ga} g(\nabla u, \cdot )+d(\Ga(\nabla u) ) + u^{-1} \Ga(\nabla u) du\\
	&= -u\Delta \Ga - u\, \Ric (\Ga, \cdot) + \nabla^2 u (\Ga, \cdot)- \nabla\Ga(\nabla u,\cdot) + u^{-1} \Ga (\nabla u) \, du\\
	&= u \Gamma_{(g,u)}(\Ga)- u\Ric (\Ga, \cdot)+ \nabla^2 u(\Ga, \cdot ) =:\hat \Gamma(\Ga),
\end{align*}
where we use $2 \b \big( u\DD \Ga ) = 2u \b (\DD \Ga) - L_{\Ga} g (\nabla u, \cdot) + (\Div \Ga )du $ by \eqref{eq:product} in the second line, $2u \b \DD \Ga = -u\Delta \Ga - u\, \Ric (\Ga)$ by \eqref{eq:laplace-beta} and a similar computation as \eqref{eq:com} in the third line, and the definition of the operator $\Gamma_{(g,u)}$ in \eqref{eq:Gamma} in the last line. 

Recall that  $\Gamma_{(g, u)}$ is an isomorphism by Item~\eqref{item:0} in Lemma~\ref{PDE}. For $(g, u)$ sufficiently close to the static vacuum pair $(\bar{g}, \bar{u})$, we have $u>0$ and $-u\Ric_g +\nabla^2_g u$ small, so  we conclude that the operator $\hat \Gamma$ is also an isomorphism. Together with the boundary condition $\Ga = 0$ on $\Sigma$, we conclude that $\Ga$ is identically zero in $M\setminus \Omega$. 

\end{proof}

In Section~\ref{se:solution} below, we shall solve the gauged boundary value problem \eqref{rsv} near a static vacuum pair via Inverse Function Theorem. As preparation, let us state some basic properties of the linearized operator.

Denote by $L$ and $L^\Ga$ the linearizations of $T$ and $T^{\Ga}$ at a static vacuum pair $(\bar g, \bar u)$, respectively. Explicitly, 
\begin{align}\label{eq:L}
 L(h, v)=  \begin{array}{l} \left\{ \begin{array}{l}
-\bar u \Ric'(h) + (\nabla^2)'(h)  \bar u - v \Ric + \nabla^2 v\\
\Delta v + \big( \Delta'(h) \big)\bar u
\end{array}\right. \quad \mbox{ in } M\setminus \Omega \\
\left\{ \begin{array}{l}
	h^\intercal\\
	H'(h)
\end{array}\right. \quad \mbox{ on } \Sigma
\end{array}
\end{align}
\begin{align}\label{lsv}
 L^{\Ga}(h, v)=  \begin{array}{l} \left\{ \begin{array}{l}
-\bar u \Ric'(h) + (\nabla^2)'(h)  \bar u - v \Ric+ \nabla^2 v-\bar u\DD {\Ga}'(h,v)\\
\Delta v + \big( \Delta'(h) \big)\bar u-{\Ga}'(h,v) (\nabla \bar u)
\end{array}\right.\quad \mbox{in } M\setminus \Omega\\
\left\{ \begin{array}{l}
	\Ga'(h,v) \\
	h^\intercal\\
	H'(h)
\end{array}\right. \quad \mbox{ on } \Sigma
\end{array}
\end{align}
where 
\begin{align*}
&L:\C^{2,\alpha}_{-q}(M\setminus \Omega) \longrightarrow \C^{0,\alpha}_{-q-2}(M\setminus\Omega)\times  \mathcal{B}(\Sigma)\\
&L^\Ga: \C^{2,\alpha}_{-q}(M\setminus \Omega) \longrightarrow \C^{0,\alpha}_{-q-2}(M\setminus\Omega)\times \C^{1,\alpha}(\Sigma) \times  \mathcal{B}(\Sigma). 
\end{align*}
Here and for the rest of this section, the geometric operators are all computed with respect to  $\bar g$ and the linearizations are all taken at $(\bar g, \bar u)$.

From the formulas in Section~\ref{se:formula}, it is direct to check that the first two equations of $L^\Ga(h, v)$ can be expressed as 

\begin{align}\label{eq:S'}
\begin{split}
-\bar u \Ric'(h) + (\nabla^2)'(h)  \bar u - v \Ric+ \nabla^2 v-\bar u\DD {\Ga}'(h,v)&=\tfrac{1}{2} \bar u \Delta h + E_1\\
\Delta v + \big( \Delta'(h) \big)\bar u-{\Ga}'(h,v) (\nabla \bar u)&=\Delta v+ E_2
\end{split}
\end{align}
where lower order terms $E_1, E_2$ are linear functions in $h, v, \nabla h, \nabla v$ and of the order $O^{0,\alpha}(|x|^{-2-2q})$ at infinity.

The linear operators $L$ and $L^\Ga$ also share similar relations as Lemmas~\ref{le:diffeo} and \ref{rsv-to-sv}  for the nonlinear operators $T$ and $T^\Ga$, which we summarize in the next  lemma.

\begin{lemma}\label{le:linear}
Let $(M, \bar{g}, \bar{u})$ be a static vacuum triple with $\bar{u}>0$. Then the following holds:
\begin{enumerate}
\item Let $(h,v)\in \C^{2,\a}_{-q}(M\setminus \Omega)$ solve $L^\Ga(h, v) = (0, 0, 0, \tau ,\phi)$. Then  $\Ga'(h,v)=0$ in $M\setminus \Omega$ and thus $L(h, v)=(0, 0, \tau, \phi)$.\label{it:gauge}
\item For any $(h,v)\in \C^{2,\a}_{-q}(M\setminus \Omega)$, there is a vector field $X\in \mathcal X(M\setminus \Omega)\cap \C^{3,\a}_{1-q}(M\setminus \Omega)$ such that $\Ga' \big(h+L_X\bar g, v+X(\bar u)\big )=0$.\label{it:SH}

\end{enumerate}

\end{lemma}
\begin{proof}
The proof of Item~\eqref{it:gauge} proceeds similarly as the proof of Lemma~\ref{rsv-to-sv} by linearizing those identities   at $(\bar{g}, \bar{u})$.

For Item~\eqref{it:SH},  given $(h, v)$, note ${\Ga}'{(h,v)}= O^{1,\a}(|x|^{-q-1})$. By Item~\eqref{item:n} in Lemma \ref{PDE}, we let $X\in \C^{3,\a}_{1-q}(M\setminus \Omega)$ solving $\Gamma(X)=-\Ga'(h,v)$ in $M\setminus \Omega$ and $X=0$ on $\Sigma$. Recall \eqref{eq:gauge} that $\Gamma(X) = \Ga'(L_X \bar g, X(\bar u))$, and thus we get the desired $X$. 
\end{proof}

 One can proceed as in \cite[Lemma 4.6]{An-Huang:2021} (see also  \cite[Proposition 3.1]{Anderson-Khuri:2013}) to show that the operator $L^\Ga$ is elliptic and Fredholm of index zero. In fact, the operator $L^\Ga$ here has exactly the same leading order terms as the special case considered in  \cite{An-Huang:2021}, and thus the same proof applies verbatim.  Also, recall that  
\begin{align*}
	\Ker L \supseteq \Big\{ \big(L_X \bar g, X(\bar u) \big):X\in \mathcal X (M\setminus \Omega) \Big\}
\end{align*}
where $(L_X \bar g, X(\bar u))$ arises from diffeomorphisms. Therefore, $\Ker L^\Ga$ always contains an $N$-dimensional subspace arising from the ``gauged'' space $\mathcal X^\Ga (M\setminus \Omega)$. We summarize those properties in the following lemma.

\begin{lemma}\label{le:Fred}
The operator 
\[
L^\Ga: \C^{2,\alpha}_{-q}(M\setminus \Omega)\longrightarrow \C^{0,\alpha}_{-q-2}(M\setminus \Omega)\times \C^{1,\alpha}(\Sigma) \times \mathcal B(\Sigma) 
\] 
is elliptic of Fredholm index zero whose kernel $\Ker L^\Ga$ contains an $N$-dimensional subspace: 
\begin{align} \label{eq:set}
	\Ker L^\Ga \supseteq \Big\{ \big(L_X \bar g, X(\bar u) \big):X\in \mathcal X^\Ga (M\setminus \Omega) \Big\}.
\end{align}
\end{lemma}

To close this section, we include a basic fact about analyticity of the kernel elements. It will be used in the proof of Theorem~\ref{th:trivial} below.

\begin{corollary}\label{co:analytic}
Let  $(M, \bar g, \bar u)$ be a static vacuum triple with $\bar{u}>0$. Let $(h, v)$ solve $L(h, v)=0$ in $M\setminus \Omega$. Then the following holds:
\begin{enumerate}
\item There exists $X\in \mathcal X(M\setminus \Omega)\cap \C^{3,\alpha}_{1-q}(M\setminus \Omega)$ such that $(\hat h, \hat v) = (h, v)+ (L_X \bar g, X(\bar u))$ satisfies $L^\Ga(\hat h, \hat v)=0$ in $M\setminus \Omega$ and $(\hat h, \hat v)$ is analytic in $\Int (M\setminus \Omega)$. 
\item Furthermore, if an open subset $\hat \Sigma\subset \Sigma $ is an analytic hypersurface, then $(\hat h, \hat v)$ is analytic up to $\hat \Sigma$.
\end{enumerate}
\end{corollary}
\begin{proof}
The existence of $X$ is by Item~\eqref{it:SH} in Lemma~\ref{le:linear}. It remains to argue that $(\hat h, \hat v)$ is analytic. By \cite{Muller-zum-Hagen:1970} (or Theorem~\ref{th:analytic} below), there is an analytic atlas of $\Int M$ such that $(\bar g, \bar u)$ is analytic. Since $L^\Ga(h, v)$ is elliptic by Lemma~\ref{le:Fred}, by elliptic regularity \cite[Theorem 6.6.1]{Morrey:1966}, we see that $(h, v)$ is analytic in $\Int (M\setminus \Omega)$. If the portion $\hat \Sigma$ of the boundary is analytic, then $(h,v)$ is analytic up to $\hat \Sigma$ by \cite[Theorem 6.7.6$^\prime$]{Morrey:1966}. 

\end{proof}


\section{Static regular and ``trivial'' kernel}\label{se:static-regular}

fd

Throughout this section, we fix a background asymptotically flat, static vacuum triple $(M, \bar g, \bar u)$ with $\bar u>0$.  Recall the assumptions on $\Omega, \Sigma$ in Notation in Section~\ref{se:intro}. In particular, $\overline{M\setminus \Omega} = M\setminus \Omega$ is a proper subset of $\Int M$.
All geometric quantities (e.g. covariant derivatives, curvatures) are computed with respect to $\bar g$ and linearization are taken at $(\bar g, \bar u)$. We often skip labelling the subscripts in $\bar g$ or $(\bar g, \bar u)$ when the context is clear.  

We recall the linearized operator $L$ defined by \eqref{eq:L}. The goal of this section is to prove Theorem~\ref{th:trivial}. It follows directly from Theorem~\ref{th:trivial'} and Corollary~\ref{co:static} below.

\begin{manualtheorem}{\ref{th:trivial}$^\prime$}\label{th:trivial'}
Let $(M,\bar g,\bar u)$ be an asymptotically flat, static vacuum triple with $\bar u>0$. Let $\hat \Sigma$ be a nonempty, connected,  open subset of $\Sigma$ with  $\pi_1(M\setminus \Omega, \hat \Sigma)=0$.  Let $(h, v)\in \C^{2,\alpha}_{-q}(M\setminus \Omega)$ satisfy $L (h, v)=0$ with either one of the following conditions:
\begin{enumerate}
\item \label{co:1} 
\begin{align}\label{eq:static1'}
v= 0, \quad \big( \nu'(h) \big)(\bar u) + \nu(v)=0,  \quad  A'(h)=0 \qquad \mbox{ on }\hat \Sigma.
\end{align}
\item \label{co:2} $\hat \Sigma$ is an analytic hypersurface  and
\begin{align}\label{eq:static3'}
\begin{split}
	 &A'(h)=0,\quad \big( \Ric'(h)\big)^\intercal =0,\quad \big(( \nabla^k_\nu \, \Ric)'(h)\big)^\intercal =0 \qquad \mbox { on } \hat \Sigma\\
	 &\mbox{for all  positive integers $k$}. 
\end{split}
\end{align}

\end{enumerate}
 Then $(h, v) = \big(L_X \bar g, X(\bar u) \big)$ for some $X\in \mathcal X(M\setminus \Omega)$.

\end{manualtheorem}

Recall Definition~\ref{de:static} for static regular of type (I) and type (II). We get the following consequence from the above theorem, which completes the proof of Theorem~\ref{th:trivial}. 
\begin{corollary}\label{co:static}
Let $(M\setminus \Omega, \bar g, \bar u)$ be static vacuum with $\bar u>0$. 
\begin{enumerate}
    \item If $\Sigma$ is static regular in $(M\setminus \Omega, \bar g, \bar u)$ of either type (I) or type (II), then 
\begin{align}\label{eq:kernel2}
		\Ker L = \Big\{ \big(L_X \bar g, X(\bar u) \big):X\in \mathcal X (M\setminus \Omega) \Big\}.
\end{align}
\item \label{it:gauge-kernel}If \eqref{eq:kernel2} holds, then 
\[
	\Ker L^\Ga = \Big\{ \big(L_X \bar g, X(\bar u) \big):X\in \mathcal X^\Ga(M\setminus \Omega) \Big\},
\]
and thus $\Dim \Ker L^\Ga= N$. 

\item \label{it:X} If $(h, v) = \big(L_X \bar g, X(\bar u)\big)$ for some vector $X$  in $M\setminus \Omega$ with $X=0$ on $\hat \Sigma$, then both \eqref{eq:static1'} and \eqref{eq:static3'} hold everywhere on $\hat \Sigma$, i.e. those boundary conditions \eqref{eq:static1'} and \eqref{eq:static3'} are ``gauge invariant''.
\end{enumerate}
\end{corollary}
\begin{proof}

We prove Item~\eqref{it:gauge-kernel}.  By Item~\eqref{it:gauge} in Lemma~\ref{le:linear}, we have $\Ker L^\Ga \subseteq \Ker L$ and thus together with  \eqref{eq:kernel2}
\[
\Ker L^\Ga \subseteq \Big\{ \big(L_X \bar g, X(\bar u) \big):X\in \mathcal X^\Ga (M\setminus \Omega) \Big\}.
\]
The opposite inclusion follows \eqref{eq:set}.

We verify in Item~\eqref{it:X}. Clearly, $X(\bar u)=0$ on $\hat \Sigma$. 
We then verify  $\big(( \nabla^k_\nu \, \Ric)'(L_X \bar g)\big)^\intercal =0$, and the rest equalities can be verified similarly. For $e_a, e_b$ tangential to $\hat \Sigma$,  
\begin{align*}
&\big((\nabla^k_\nu \,\Ric)'(L_X  \bar g) \big) (e_a, e_b)\\
&= \big( L_X (\nabla^k_\nu \, \Ric )\big)(e_a, e_b)\\
& = \big ( \nabla_X (\nabla^k_\nu \, \Ric)\big)(e_a, e_b) +  \big(\nabla^k_\nu \, \Ric\big)(\nabla_{e_a} X, e_b)  +  \big(\nabla^k_\nu \, \Ric\big)({e_a} , \nabla_{e_b} X)\\
&=0.
\end{align*}
\end{proof}

We shall refer to Condition~\eqref{co:1} in Theorem~\ref{th:trivial'} as the \emph{Cauchy boundary condition} and to Condition~\eqref{co:2} as the \emph{infinite-order boundary condition}. We prove Theorem~\ref{th:trivial'} for the Cauchy boundary condition and the infinite-order boundary condition respectively in Section~\ref{se:uniqueness} and Section~\ref{se:uniqueness2} below.  In both sections, we observe that to show $(h, v) = (L_X \bar g, X(\bar u))$, it suffices to show $h$ by itself is $L_X\bar g$.

\begin{lemma}\label{le:trivial-h}
Let $(M\setminus \Omega, \bar g, \bar u)$ be static vacuum with $\bar u>0$ and $(h,v) \in \C^{2,\alpha}_{-q}(M\setminus \Omega)$. Suppose $h = L_X \bar g$ for some locally $\C^3$ vector field $X$ in $M\setminus \Omega$. Then the following holds.
\begin{enumerate}
\item $X-Z\in \C^{3,\alpha}_{1-q}(M\setminus \Omega)$ where $Z\in \mathcal Z$. \label{it:asymptotics}
\item If $h^\intercal=0, H'(h)=0$ on $\Sigma$ and $X=0$ in an open subset of  $\Sigma$, then $X\equiv 0$ on $\Sigma$ and thus $X\in \mathcal X(M\setminus \Omega)$. \label{it:boundary}
\item If $S'(h, v)=0$ in  $M\setminus \Omega$, then $v = X(\bar u)$ in $M\setminus \Omega$. \label{it:v}
\end{enumerate}
Consequently, if $L(h, v)=0$ and $h=L_X\bar g$ for some vector field $X\in \C^{3,\alpha}_{\mathrm{loc}}(M\setminus \Omega)$ with $X=0$ in an open subset of $\Sigma$, then $X\in \mathcal X(M\setminus\Omega)$ and $(h, v) = (L_X \bar g, X(\bar u))$. 
\end{lemma}
\begin{remark}
It is clear in the following proof that the assumption $X=0$ in an open subset of $\Sigma$ in Item~\eqref{it:boundary} can be replaced by a weaker assumption that  $X$ along $\Sigma$ vanishes at infinite order at some point $p\in \Sigma$.
\end{remark}
\begin{proof}
To prove Item~\eqref{it:asymptotics}, we first recall \eqref{eq:derivativeX} that 
\begin{align*}
	X_{i;jk}=- R_{\ell kj i} X^\ell + \tfrac{1}{2} (h_{ij;k} + h_{ik;j} - h_{jk;i}).
\end{align*}
In an asymptotically flat end, let $\gamma(t)$, $t\in [1,\infty)$,  be a radial geodesic ray,  parametrized by arc length, that goes to infinity. Note that by asymptotical flatness, the parameter $t$ and $|x|$ are comparable, i.e. $C^{-1} |x| \le t \le C|x|$ for some positive constant $C$.  Then the previous identity implies that, along $\gamma(t)$,
\[
	\frac{d^2}{dt^2} X(\gamma(t)) = B(t) X(\gamma(t)) + D(t)
\]
where the coefficient matrix $B(t)$ is the restriction of the curvature tensor $\mathrm{Rm}(\gamma(t))$ and $D(t)$ is the restriction  $\nabla h(\gamma(t))$. Consequently, $B(t) = O(t^{-q-2}), D(t)=O(t^{-q-1})$. By elementary ODE estimates, we have $|X| + t|X'|\le Ct$ for some constant $C$. (See, for example, \cite[Lemma B.3]{Huang-Martin-Miao:2018}, where the homogeneous case $B=0$ was proved, but a similar argument extends to our inhomogeneous case here.) By varying  geodesic rays, we obtain $X= O^{3,\alpha}(|x|)$. Together with the elliptic equation $\Delta X + \Ric (X, \cdot) = - \beta h \in \C^{1,\alpha}_{-q-1} (M\setminus \Omega)$ and harmonic expansion, it implies that 
\[
	X_i -c_{ij} x_j \in  \C^{3,\alpha}_{1-q}(M\setminus \Omega)
\]
for some constants $c_{ij}$. Using that $L_X \bar g = h= O(|x|^{-q})$, the leading term of $X$  must correspond to a rotation vector (or zero). This gives the desired asymptotics. 

To prove Item~\eqref{it:boundary}, we decompose $X=\eta\nu+X^\intercal$ on $\Sigma$ where $X^\intercal$ is tangent to~$\Sigma$.  Recall 
\begin{align}\label{eq:vector}
\begin{split}
	(L_X \bar g)^\intercal &= 2\eta A + L_{X^\intercal} \bar g^\intercal\\
	H'(L_X\bar g) &=-\Delta_\Sigma \eta- (|A|^2 + \Ric(\nu, \nu) )\eta + X^\intercal (H).
\end{split}
\end{align}
The assumptions $(L_X g)^\intercal =0$ and $H'(L_X \bar g)=0$ imply that $\eta, X^\intercal$ satisfy a linear  $2$nd-order  elliptic system on $\Sigma$:
\begin{align*}
	\Delta_\Sigma X^\intercal + \Ric_\Sigma(X^\intercal,\cdot) +2 \beta_\Sigma (\eta A)&=0\\
	\Delta_{\Sigma}\eta+\big(|A|^2+\Ric(\nu,\nu)\big)\eta-X^\intercal(H)&=0.
\end{align*}
The first equation is exactly  $\beta_{\Sigma}( L_{X^\intercal}\bar  g^\intercal +2\eta A  )=0$, which follows the assumption $(L_X g)^\intercal =0$ and  the first equation in \eqref{eq:vector}, where  $\beta_{\Sigma}$ denotes the \emph{tangential} Bianchi operator, i.e., $\beta_{\Sigma} \tau = - \Div_{\Sigma}\tau + \tfrac{1}{2} d_\Sigma (\tr_{\Sigma} \tau )$. By unique continuation, both $\eta, X^\intercal$ are identically zero on $\Sigma$.

To prove Item~\eqref{it:v}, using $h=L_X \bar g$ and noting $S'\big (L_X \bar g, X(\bar u)\big ) =0 $, we obtain 
\[
	0 = S'( h ,  v) - S'\big (L_X \bar g, X(\bar u)\big ) = S'(0, v - X(\bar u))=0 \quad \mbox{ in } M\setminus \Omega.
\]
Let  $f := v-X(\bar u)$. The previous identity implies  $-f \Ric + \nabla^2 f=0$ in $M\setminus \Omega$. Noting $f(x)\to 0$ at infinity because both $v, X(\bar u)\to 0$ (by Item~\eqref{it:asymptotics}), we can conclude $f\equiv 0$ (see, for example, \cite[Proposition B.4]{Huang-Martin-Miao:2018}).


\end{proof}

\subsection{Cauchy boundary condition}\label{se:uniqueness}


In this section, we assume $(h, v)\in \C^{2,\alpha}_{-q}(M\setminus \Omega)$ satisfies $L(h, v)=0$ and 
\begin{align} \label{eq:Cauchy-all}
v= 0,\quad  \big(\nu'(h)\big) (\bar u) + \nu(v)=0, \quad  A'(h)=0 \quad \mbox{ on } \hat \Sigma,
\end{align}
where $\hat \Sigma$ is a connected open subset of $\Sigma$ satisfying $\pi_1(M\setminus\Omega, \hat \Sigma)=0$. Since $L(h, v)=0$ implies  $h^\intercal = 0, H'(h)=0$ on~$\Sigma$,  all those  boundary conditions together are referred to as the \emph{Cauchy boundary condition}.  The reason is that if $h$ satisfies the geodesic gauge, then $h^\intercal=0, A'(h)=0$ imply   
\begin{align}\label{eq:Cauchy1}
h=0, \quad \nabla h=0 \quad \mbox{ on } \hat \Sigma
\end{align}
 while $v= 0, \big(\nu'(h)\big) (\bar u) + \nu(v)=0$ imply 
 \begin{align}\label{eq:Cauchy2}
 v=0, \quad \nu(v)=0\quad  \mbox{ on }\hat \Sigma.
 \end{align} 
 Since $(h, v)$ satisfies  a system of $2$nd-order partial differential equations, the conditions \eqref{eq:Cauchy1} and \eqref{eq:Cauchy2} are precisely the classical Cauchy boundary condition. However, Theorem~\ref{th:trivial'} does not follow from the classical theorem on uniqueness because the system $L(h, v)$ is not elliptic under the geodesic gauge.

We outline the proof below. We first show that $h$ is ``locally trivial'' in the sense that $h=L_X \bar g$ in some open subset of $M$,  achieved by suitably extending $(h, v)$ across $\hat\Sigma$. Then we would like to show that $h$ must be globally trivial. In general, a locally trivial $h$ need not be globally trivial, but here we know $h$ is analytic (after changing to the static-harmonic gauge), so we can use Theorem~\ref{th:extension}.  To carry out the argument, it is important to work with the appropriate gauge at each step of the proof.

\begin{proof}[Proof of Theorem~\ref{th:trivial'} (under Cauchy boundary condition)]

We without loss of generality assume that $(h, v)$ satisfies the geodesic gauge, and thus we have 
\begin{align}\label{eq:Cauchy}
    h=0, \quad  \nabla h=0,\quad  v=0, \quad \nu(v)=0\quad \mbox{ on }\hat \Sigma
\end{align}
as discussed earlier. We then extend $(h, v)$ by $(0, 0)$ across $\hat \Sigma$ into some small open subset $U\subset \Omega \cap \Int M$ and $\overline U\cap (M\setminus\Omega)$ is exactly the closure of $\hat\Sigma$. Denote the $\C^1$ extension  pair by $(\hat h, \hat v)$:
\begin{align}\label{eq:hat-h}
	(\hat h, \hat v) = \left\{ \begin{array}{ll}  (h, v) & \mbox{ in } M\setminus \Omega \\ 
	(0,0) & \mbox{ in } \overline U \end{array} \right..
\end{align}
Here, the open subset $U$ is chosen sufficiently small such that the ``extended'' manifold 
\[
\hat M = (M\setminus \Omega) \cup \overline U
\]
 has a smooth embedded hypersurface boundary $\partial \hat M$ and  $\pi_1(\hat M, U)=0$ based on the assumption $\pi_1( M\setminus \Omega, \hat \Sigma) = 0$. Note that $\hat M$ is a closed, proper subset of $\Int M$.

Using the Green-type identity~\eqref{equation:Green}, we verify that  $(\hat h, \hat v)$ solves $S'(\hat h, \hat v)=0$ in $\hat M$ weakly, which is equivalent to showing that $P(\hat h, \hat v)=0$ weakly: For $(k, w) \in \C^{\infty}_c(\hat M)$, we have 
\begin{align*}
    \int_{\hat M} (\hat h, \hat v) \cdot P(k, w)  \dvol &= \int_{ M\setminus \Omega } (\hat h, \hat v) \cdot P(k, w)  \dvol\\
    &=\int_{M\setminus \Omega} P(\hat h, \hat v) \cdot (k, w)  \dvol \\
    &\quad +\int_{\hat \Sigma} Q(h, v)\cdot \big(k^\intercal, H'(k)\big) \da- \int_{\hat \Sigma} Q(k,w)\cdot \big(h^\intercal, H'(h)\big)\, \da\\
    &=\int_{M\setminus \Omega} P(\hat h, \hat v) \cdot (k, w)  \dvol
\end{align*}
where the boundary terms vanish because \eqref{eq:Cauchy} (also note that $(k,w)$ on $\Sigma$ is supported in $\hat \Sigma$). Then we solve for $Y\in \C^{1,\alpha}_{1-q}(\hat M)$ with $Y=0$ on $\partial \hat M$ such that 
\begin{align}\label{eq:tilde-h}
	(\tilde h, \tilde v):= (\hat h, \hat v) + \big(L_Y \bar g, Y(\bar u) \big)
\end{align}
weakly solves the static-harmonic gauge $\Ga(\tilde h, \tilde v)=0$ in $\hat M$. Then $(\tilde h,\tilde v)$ is a solution to the elliptic system $L^\Ga(\tilde h,\tilde v)=0$, and hence $(\tilde h, \tilde v)\in \C^{2,\alpha}_{-q}(\hat M)$ by elliptic regularity. By Corollary~\ref{co:analytic}, $(\tilde h, \tilde v)$ is analytic in $\Int \hat M$, and note $(\tilde h, \tilde v)= (L_Y \bar g, Y ( \bar u))$ in $U$. Applying Theorem~\ref{th:extension} to the analytic manifold $(\hat M, \bar g)$ and analytic $\tilde h$,  the  vector field $Y|_{U}$ restricted in $U$ extends to a vector field $\tilde Y$ in $\hat M$ that is $\tilde h$-Killing in $\hat M$ with $\tilde Y = Y$ in $U$.

Denote $X= \tilde Y - Y$ in $\hat M$. Then $X$ is $\hat h$-Killing in $\hat M$ by \eqref{eq:tilde-h}. Thus, $X$ is $h$-Killing in $M\setminus \Omega$ by \eqref{eq:hat-h} with $X=0$ on an open subset of $\hat \Sigma\subset \Sigma$. By Lemma~\ref{le:trivial-h}, we conclude that $X\in \mathcal X(M\setminus \Omega)$ and $(h, v) = (L_X \bar g, X(\bar u))$.  
\end{proof}

\subsubsection{Relaxed Cauchy boundary condition}\label{sec:relaxed}
In this subsection, we note that the condition $\big(\nu'(h)\big) (\bar u) + \nu(v)=0$ in \eqref{eq:Cauchy-all} can be replaced or removed under mild assumptions as noted in Item~\eqref{it:nu} in Remark~\ref{re:intro}. For the rest of this section, we denote 
\[
	(\nu(u))'=\big(\nu'(h)\big) (\bar u) + \nu(v)
\]
since the right hand side is exactly the linearization of $\nu(u)$ on $\Sigma$.

We begin with general properties. It is well-known that a static vacuum metric $\bar g$ has a Schwarzschild expansion in the harmonic coordinates~\cite{Bunting-Masood-ul-Alam:1987}. We show that a similar expansion holds for the static vacuum deformation. 

\begin{lemma}\label{le:v-mass}
Let $(M, \bar g, \bar u)$ be a static vacuum triple with $\bar u>0$. Let $(h, v)\in \C^{2,\alpha}_{-q}(M\setminus \Omega)$ be a static vacuum deformation.
\begin{enumerate}
    \item \label{it:harmonic} Suppose $\Ga'(h, v)=0$. Then $(h, v)$ has the following expansions near infinity:
\begin{align*}
	h_{ij} (x)&= \tfrac{2}{n-2} c \delta_{ij} |x|^{2-n} + O^{2,\alpha}(|x|^{\max\{-2q, 1-n \}})\\
	v(x) &= -c |x|^{2-n} + O^{2,\alpha}(|x|^{\max\{-2q, 1-n \}})
\end{align*}
for some real number $c$. 
\item \label{it:v-mass}Suppose $h^\intercal =0$ and $H'(h)=0$ on $\Sigma$. Then $v(x)= O^{2,\alpha}(|x|^{\max\{-2q, 1-n \}})$. (Note that $\Ga'(h, v)=0$ is not assumed here.)
\end{enumerate}
\end{lemma}
\begin{proof}
We prove Item~\eqref{it:harmonic}. By \eqref{eq:S'}, the assumption  $S'(h, v)=0$ and $\Ga'(h, v)=0$ imply that $(h, v)$ satisfies $\Delta h , \Delta v\in O^{2,\alpha}(|x|^{-2-2q})$. Thus, $h$ and $v$ have the harmonic expansions (see \cite[Theorem 1.17]{Bartnik:1986}) for some real numbers $c_{ij}$ and $c$:
\begin{align*}
h_{ij} &= \tfrac{2}{n-2} c_{ij}  |x|^{2-n} + O^{2,\alpha}(|x|^{\max\{-2q, 1-n \}})\\
	v &= -c |x|^{2-n} + O^{2,\alpha}(|x|^{\max\{-2q, 1-n \}}).
\end{align*}
To see that $c_{ij}=c\delta_{ij}$, we compute
\begin{align*}
    0&=(\Ga'(h, v))_j=\sum_i \left( -h_{ij,i} + \tfrac{1}{2} h_{ii,j}\right) + v_{,j} + o(|x|^{1-n})\\
    &=\Big(2\sum_i c_{ij}x_i - \sum_i c_{ii} x_j + (n-2) cx_j\Big)|x|^{-n}+ o(|x|^{1-n}).
\end{align*}
A direct computation gives the desired identity. 

For Item~\eqref{it:v-mass}, we use Lemma~\ref{le:linear} to find $X\in \mathcal X(M\setminus \Omega)\cap \C^{3,\alpha}_{1-q}(M\setminus \Omega)$ such that $\Ga'\big(h+L_X\bar g, v + X(\bar u) \big)=0$ and apply Item~\eqref{it:harmonic} to obtain  
\begin{align*}
(h + L_X\bar g)_{ij} &= \tfrac{2}{n-2} c\delta _{ij}  |x|^{2-n} + O^{2,\alpha}(|x|^{\max\{-2q, 1-n \}})\\
	v + X(\bar u) &= -c |x|^{2-n} + O^{2,\alpha}(|x|^{\max\{-2q, 1-n \}})
\end{align*}
for some real number $c$. Use the boundary condition and apply Lemma~\ref{le:mass} to $\big(h+L_X\bar g, v + X(\bar u) \big)$, we compute $c=0$. Since $X(\bar u)$ is of the order $O(|x|^{-2q})$, the result follows. 

\end{proof}

In the next lemma, we derive the ``hidden'' boundary conditions for $(h,v)$ satisfying $L(h, v)=0$.

\begin{lemma}[Hidden boundary conditions]
Let $\hat \Sigma$ be an open subset of $\Sigma$. Suppose $(h, v)\in \C^{2,\alpha}_{-q}(M\setminus \Omega)$ satisfy
\begin{align*}
	&\quad S' (h, v)=0 \quad \mbox { in a collar neighborhood $U$ of $\hat\Sigma$ in } M\setminus \Omega\\
	&\left\{ \begin{array}{l} h^\intercal =0\\ 
	H'(h)=0\end{array} \right.	\quad \mbox { on } \hat \Sigma.
\end{align*}
 Then the following equations hold on $\hat \Sigma$:
\begin{align}
	\Delta_\Sigma v + (  \nu (u))'H + v\Ric (\nu, \nu) - \bar u A'(h) \cdot A&=0\label{eq:hidden1}\\
	A(\nabla^\Sigma v, \cdot )  + A' (h) (\nabla^\Sigma \bar u, \cdot) + \bar u \Div_\Sigma A'(h)- d_\Sigma \big( (  \nu (u))'\big)+ v \Ric (\nu, \cdot) &=0.	\label{eq:hidden2}
\end{align}
Consequently,
\begin{enumerate}
\item \label{item:subset}  Suppose $v=0$ and $A'(h)=0$ on $ \hat \Sigma$. Then $(\nu(u))'H\equiv 0$ on $\hat \Sigma$. 
\item \label{item:v} Suppose $\hat \Sigma=\Sigma$,  $v=0, A'(h)=0$ on $\Sigma$, and $S'(h, v)=0$ everywhere in $M\setminus \Omega$. Then $(\nu(u))'\equiv 0$ on  $\Sigma$. 
\end{enumerate}
\end{lemma}
\begin{proof}
Let $X$ be a vector field with support in $U$ and with support in $\hat \Sigma$ when restricted on $\Sigma$.  Let $(k, w) = (L_X \bar g, X(\bar u))$  in~\eqref{equation:Green}. Since both $(h, v)$ and $(k, w)$ are static vacuum deformations in $U$ and $h$ satisfies $(h^\intercal, H'(h))=0$ on $\hat \Sigma$, we get 
\[
	\int_{ \Sigma} \Big\langle Q(h, v) , \big((L_X\bar g)^\intercal, H'(L_X\bar g)\big) \Big\rangle\, \da =0.
\]
Using again the boundary condition of $h$, we get 
\begin{align*}
	Q(h, v) = 	 \Big( v A + \bar u A' (h) - (\nu(u))' \bar g^\intercal, \, 2v\Big)\quad \mbox{ on } \hat \Sigma. 
\end{align*}
Write  $X = \eta \nu + X^\intercal$ for a scalar function $\eta$ and $X^\intercal$ tangent to $\Sigma$ and recall~\eqref{eq:vector}: 
\begin{align*}
	(L_X \bar g)^\intercal &= 2 \eta A + L_{X^\intercal} \bar g^\intercal\\
	H'(L_X \bar g) &= - \Delta_\Sigma\eta  - \big( |A|^2 + \Ric(\nu, \nu) \big ) \eta + X^\intercal (H).
\end{align*}
By choosing 
$X^\intercal\equiv 0$, we get 
\begin{align*}
0&=\int_{\Sigma} \bigg\langle \Big ( vA + \bar u A' (h) -(\nu(u))' \bar g^\intercal, \, 2v\Big), \Big(2 \eta A , - \Delta_\Sigma\eta -\big( |A|^2 + \Ric(\nu, \nu) \big) \eta  \Big) \bigg \rangle \, \da\\ 
&=\int_{\Sigma} 2\eta\Big (v|A|^2 + \bar u A'(h) \cdot A- (\nu(u))' H - \Delta_{\Sigma} v - (|A|^2 + \Ric(\nu, \nu) )v \Big)\, \da.
\end{align*}
Simplifying the integrand and choosing arbitrary $\eta$ supported in $\hat \Sigma$, we prove \eqref{eq:hidden1}. For \eqref{eq:hidden2}, we let $\eta\equiv 0$
\begin{align*}
	0&= \int_{\Sigma} \bigg\langle \Big ( vA  + \bar u A' (h) -(\nu(u))' g^\intercal, \, 2v\Big), \Big( L_{X^\intercal} \bar g^\intercal,  X^\intercal (H) \Big) \bigg \rangle \, \da.
\end{align*}
Applying integration by part and letting $X^\intercal$ vary, we get, on $\hat \Sigma$, 
\begin{align*}
	0&=A(\nabla^\Sigma v, \cdot) + v \Div_\Sigma A  + A' (h) (\nabla^\Sigma \bar u, \cdot) + \bar u \Div_\Sigma A'(h)- d_\Sigma (\nu(u))'- vdH\\
	&=A(\nabla^\Sigma v, \cdot)  + A' (h) (\nabla^\Sigma \bar u, \cdot) + \bar u \Div_\Sigma A'(h)- d_\Sigma  (\nu(u))'+ v \Ric(\nu, \cdot)
\end{align*}
where we use the Codazzi equation.

Item~\eqref{item:subset} follows directly from \eqref{eq:hidden1} by letting $v=0$ and $A'(h)=0$.

For Item~\eqref{item:v}, we have $d_\Sigma (\nu (u))'=0$ and hence $(\nu(u))'= c$ on $\Sigma$ for some real number $c$ from \eqref{eq:hidden2}. The assumption $S'(h, v)=0$ in $M\setminus \Omega$ says that 
\[
	\Delta v + \big( \Delta'(h) \big)\bar u=0 \quad \mbox{ in } M\setminus \Omega.
\]
We integrate and apply divergence theorem (recall $\nu$ on $\Sigma$ points to infinity) and Lemma~\ref{le:W} below to get 
\begin{align*}
	0&=\int_{M\setminus \Omega}\Big( \Delta v + \big( \Delta'(h) \big)\bar u\Big) \, \dvol\\
	&= -\int_{\Sigma}\big( \nu (v) + \dnu \bar u\big) \, \da = -\int_\Sigma (\nu(u))' \, \da,
\end{align*}
which implies $(\nu(u))'= 0$ on $\Sigma$, where we use  $\int_{S_\infty} \nu(v)=0$ by Item~\eqref{it:v-mass} in Lemma~\ref{le:v-mass} and the integral formula~\eqref{eq:it}
\[
    \int_{M\setminus \Omega} (\Delta'(h))\bar u \, \dvol = -\int_\Sigma \dnu \bar u\, \da
\]
proven in the next lemma. 
\end{proof}

Recall the formulas  \eqref{eq:laplace} and  \eqref{equation:normal}: for a scalar function $f$, 
\begin{align*}
	\big( \Delta'(h) \big)f &= - h^{ij}f_{;ij} +\bar g^{ij}  \left(-(\Div h)_i + \tfrac{1}{2} (\tr h)_{;i}  \right) f_{,j}\\
	\nu'(h) &= - \tfrac{1}{2} h(\nu, \nu) \nu - \bar g^{ab} \omega(e_a) e_b
\end{align*}
where $\omega(\cdot) = h(\nu, \cdot)$ and the indices are raised by $\bar g$.  
\begin{lemma}\label{le:W}
Let $h\in \C^{2,\alpha}_{-q}(M\setminus \Omega)$ be a symmetric $(0,2)$-tensor. Given a scalar function $f$, define the covector $W_f = -h(\nabla f, \cdot) + \frac{1}{2} (df) \tr h$. Then 
\begin{align*}
	\big( \Delta'(h) \big)f + \tfrac{1}{2} (\Delta f) \tr h&= \Div W_f \quad \mbox{ in } M\setminus \Omega,
\end{align*}
and if we further assume that $\tr_\Sigma h=0$ on $\Sigma$, then 
\begin{align*}
	\dnu f &=W_f(\nu) \quad \mbox{ on } \Sigma.
\end{align*}
Consequently, if $f$ is harmonic and is of the order $O^{1}(|x|^{-q})$, then 
\begin{align}\label{eq:it}
	\int_{M\setminus \Omega} \big( \Delta'(h) \big) f \, \dvol = - \int_\Sigma \dnu f \, d\sigma.
\end{align}
\end{lemma}
\begin{proof}
Compute  at a point with respect to normal coordinates of~$\bar g$:
\begin{align*}
	 \Div W_f = (W_f)_{i;i} &= - f_{;ji} h_{ij} -f_{j} h_{ij;i} + \tfrac{1}{2} (\Delta f) \tr h + \tfrac{1}{2} f_{;i} (\tr h)_{;i}.
\end{align*}	
It gives the first identity. For the second identity, 
\begin{align*}
	W_f( \nu)  &= -h(\nabla f, \nu) + \tfrac{1}{2} \nu (f) \tr h\\
	&=-h(\nu, \nu) \nu(f) - h(\nabla^\Sigma f, \nu) + \tfrac{1}{2} \nu (f)\, h (\nu, \nu)\\
	&=-\tfrac{1}{2}  h(\nu, \nu) \nu(f)- h(\nabla^\Sigma f, \nu).
\end{align*}

If $f$ is harmonic, we have 
\begin{align*}
	\int_{M\setminus \Omega} \big( \Delta'(h) \big)f \, \dvol &= \int_{M\setminus \Omega} \Div W_{f} \, \dvol\\
	&=- \int_\Sigma W_{f}(\nu)\, \da\\
	&=-\int_\Sigma \dnu f\, \da
	\end{align*}
where the boundary term at infinity vanishes because  $\nabla f = O(|x|^{-q-1})$ and thus $W_{f}\in O(|x|^{-2q-1})$.
\end{proof}

\subsection{Infinite-order boundary condition} \label{se:uniqueness2}
In this section, we assume $(h, v)\in \C^{2,\alpha}_{-q}(M\setminus\Omega)$ and $(h, v)$ satisfies
\[
	 A'(h)=0,\quad \big( \Ric'(h)\big)^\intercal =0,\quad \big(( \nabla^k_\nu \, \Ric )'(h)\big)^\intercal =0 \quad \mbox { on }\quad \hat \Sigma
\]
for all  positive integers $k$, where $\hat\Sigma$ satisfies $\pi_1(M\setminus \Omega, \hat \Sigma)=0$ and $\hat \Sigma$ is an analytic hypersurface.  It is worth noting that above boundary conditions are imposed only on $h$ (not on $v$ at all).

The following lemma justifies the term ``infinite-order'' boundary condition. Recall that a symmetric $(0,2)$-tensor $h$ is said to satisfy the geodesic gauge in a collar neighborhood of $\hat \Sigma$ if $h(\nu, \cdot) = 0$ in the collar neighborhood, where $\nu$ is the parallel extension of the unit normal to $\hat \Sigma$. 
\begin{proposition}\label{pr:infinite}
Fix an integer $\ell\ge 0$. Suppose $h \in \C^{\ell+2,\alpha}(M\setminus \Omega)$ satisfies the geodesic gauge in a collar neighborhood of $\hat \Sigma$. Then the following two boundary conditions are equivalent
\begin{enumerate}
\item \label{it:infinite1} For all $k=0, 1,\dots, \ell$,
\[
h^\intercal =0, \quad A'(h)=0,
\quad \big( (\nabla^k_\nu \, \Ric )'(h)\big)^\intercal=0\qquad\mbox{ on } \hat \Sigma.
\]
Here, we interpret  the $0$-th covariant derivative as  $ (\nabla^0_\nu \, \Ric )'(h)=\Ric'(h)$.
\item \label{it:infinite2} For  all $k=0, 1, \dots, \ell$, 
\[
	h=0, \quad \nabla h=0, \quad \nabla^{k+2} h=0 \qquad \mbox{ on }\hat \Sigma. 
\]

\end{enumerate}
\end{proposition}

\begin{proof} 
It is clear that Item~\eqref{it:infinite2} implies Item~\eqref{it:infinite1} by the formulas of $A'(h)$ and $\Ric'(h)$ in \eqref{eq:sff} and \eqref{equation:Ricci}. The rest of proof is devoted to prove that Item~\eqref{it:infinite1} implies Item~\eqref{it:infinite2}.

We compute with respect to an orthonormal frame $\{ e_0, e_1,\dots, e_{n-1}\}$ where $e_0 =\nu$ parallel along itself. In the following computations, we let $i, j, k_1,\dots, k_{\ell}, k_{\ell +1} \in \{ 0, 1, \dots, n-1\}$ denote all directions and let $a, b, c\in \{ 1,\dots, n-1\}$ denote only the tangential directions. Since $h$ satisfies the geodesic gauge in a collar neighborhood of $\Sigma$, $h_{0i}$ and any covariant derivatives of $h_{0i}$ vanish in the neighborhood. 

Already noted in \eqref{eq:Cauchy1},  from $h^\intercal=0,A'(h)=0$ and geodesic gauge, we have $h=0,\nabla h=0$ on $\Sigma$. We will prove inductively in $\ell$. 

First, for $\ell=0$, we have $h=0$, $\nabla h=0$,  $\big( (\Ric')(h)\big)^\intercal =0$ on $\hat \Sigma$. We would like to show that $\nabla^2 h=0$ on $\hat \Sigma$.  By geodesic gauge, we automatically have $\nabla^2 h (\nu, e_i)=0$. Since $\nabla h=0$, we have $\nabla_{e_a}\nabla h =0$ and  $\nabla_\nu\nabla_{e_a}h=\nabla_{e_a}\nabla_\nu h +\mathrm{Rm}\circ h =0$. To summarize, when expressed in the local frame, we already have
\[
	h_{0j; k_1k_2}=0,\quad  h_{ij; k_1 a}=0, \quad h_{ij; a k_1 }=0\quad \mbox{ on } \hat \Sigma.
\]
Therefore, it remains to show that $(\nabla_\nu \nabla_\nu h)^\intercal =0$; namely $h_{ab;00}=0$. Recall the formula of $\Ric'(h)$ from \eqref{equation:Ricci}, restricted on the tangential vectors $e_a, e_b$:
\begin{align}\label{eq:Ricci2}
\begin{split}
	\Ric'(h)_{ab} & = -\tfrac{1}{2} h_{ab;ii} + \tfrac{1}{2} (h_{ai; ib} + h_{bi; ia} ) - \tfrac{1}{2} h_{ii; ab} \\
	&\quad + \tfrac{1}{2} (R_{ai} h_{jb} + R_{bi} h_{ja} ) - R_{aijb} h_{ij}.
\end{split}
\end{align}
Since $0=\Ric'(h)_{ab} = 	-\tfrac{1}{2} h_{ab;00}$, we obtain $\nabla^2 h=0$ on $\hat \Sigma$.

We proceed inductively in $\ell>0$. Suppose that Item~\eqref{it:infinite1} holds for $k=0,1,\dots, \ell$ and that we already have $h=0, \nabla h=0, \nabla^{k+2} h =0$ on $\hat \Sigma$ for $k=0, 1, \dots, \ell -1$. We would like to derive $\nabla^{\ell+2} h =0$ on $\hat \Sigma$. We compute $\nabla^{\ell+2}h$ in the following three cases according to the location of $\nu$:
\begin{align*}
	(\nabla^{\ell+2}_{\tiny\mbox{arbitrary}} h) (\nu, e_i)&=h_{0i; k_1 \cdots k_{\ell+1} } = 0 \quad \quad \mbox{ (by geodesic gauge)}\\
	(\nabla_{\tiny\mbox{tangential}} \nabla^{\ell+1}_{\tiny\mbox{arbitrary}} h) (e_a, e_b) &=h_{ab;k_1k_2\cdots k_{\ell}c }\\
	&= \big(h_{ab;k_1k_2\cdots k_{\ell}}\big)_{,c} +(\mbox{Christoffel symbols}) \circ \nabla^{\ell+1} h \\
	& =0 \quad \quad (\mbox{by the inductive assumption } \nabla^{\ell+1} h=0 \mbox{ on } \hat \Sigma),
\end{align*}
It remains to compute $(\nabla_{\nu} \nabla^{\ell+1}_{\tiny\mbox{arbitrary}} h) (e_a, e_b)=0$. If any of those `arbitrary' covariant derivatives is tangential, we can as before commute the derivatives to express it in the form $(\nabla_{\tiny\mbox{tangential}} \nabla^{\ell+1}_{\tiny\mbox{arbitrary}} h) (e_a, e_b)$ plus the contraction between curvature tensors of $\bar g$ and derivatives of $h$ in the order~$\ell+1$ or less. Therefore, it suffices to show 
\begin{align}\label{eq:h-nu}
(\nabla^{\ell+2}_\nu h)(e_a, e_b)=0, \quad \mbox{ or in the local frame } h_{ab;\tiny\underbrace{0\cdots 0}_{\ell+2-\mbox{times}}} = 0.
\end{align}
To prove \eqref{eq:h-nu}, note that from \eqref{eq:Ricci2}, we exactly have 
\[
		(\nabla^{\ell}_\nu (\Ric'(h)) \big)(e_a, e_b) = - \tfrac{1}{2} h_{ab;\tiny\underbrace{0\cdots 0}_{\ell+2-\mbox{times}}} \quad \mbox{ on } \hat \Sigma
\]
because all other terms vanish there. We \emph{claim} that\footnote{The notations unfortunately become overloaded. Recall that $(\nabla^\ell_\nu \Ric )'(h) $ denotes the linearized $\nabla^\ell_\nu \Ric$, and note $\nabla^{\ell}_\nu (\Ric'(h))$ represents the $\ell$th covariant derivative in $\nu$ of $\Ric'(h)$. The claimed identity says that ``taking the $\nu$-covariant derivative'' and ``linearizing'' are commutative under the assumption of $h$. } for each $i=0, 1, \dots, \ell$: 
\begin{align}\label{eq:commutative}
	  (\nabla^\ell_\nu \Ric )'(h) = \nabla^{\ell-i}_\nu \big((\nabla_\nu^i \Ric)'(h)\big) \qquad \mbox{ on } \hat\Sigma
\end{align}
provided $h=0, \nabla h=0, \dots, \nabla^{\ell-i} h=0$ on $\hat \Sigma$.

Once the claim is proven, together with the assumption $ \big((\nabla^\ell_\nu \Ric )'(h) \big)^\intercal=0$ on $\hat \Sigma$ and the above computations, we conclude $\nabla^{\ell+2} h=0$ on $\hat \Sigma$ and complete the proof. 
\begin{proof}[Proof of Claim]
Let $g(s)$ be a differentiable family of metrics with $g(0) = \bar g$ and $g'(0)=h$. By definition, 
\begin{align*}
	 (\nabla^\ell_\nu \Ric )'(h) =\left. \ds\right|_{s=0}\left( \nabla^\ell_{\nu_{g(s)}} \Ric_{g(s)} \right)
\end{align*}
where $\nabla_{\nu_{g(s)}}$ is the $g(s)$-covariant derivative in the unit normal $\nu_{g(s)}$. When the $s$-derivative falls on $\nabla^i_\nu \Ric_{g(s)}$, we get the desired term $\nabla^{\ell -i }_\nu \big(\nabla_\nu^i \Ric)'(h)\big)$. It suffices to see that those terms involving the $s$-derivative on one of the covariant derivative $\nabla_{\nu_{g(s)}}$ must vanish. It follows that  $\left.\ds\right|_{s=0}\nabla_{\nu_{g(s)}}$ is a linear function in only $h$ and $\nabla h$. So $\left.\ds\right|_{s=0}\nabla^{\ell-i}_{\nu_{g(s)}}$ is a linear function in $h$ and the derivatives of $h$ up to the $(\ell-i)$th order, and thus it vanishes on $\hat \Sigma$.  We prove the claim. 
\end{proof}
\end{proof}

\begin{remark}
We will not use the following fact, but nevertheless it is interesting to note that, using similar computations, both items in Proposition~\ref{pr:infinite} are equivalent to the condition: for $k= 1, \dots, \ell$, 
\[
	h^\intercal=0, \quad A'(h)=0,\quad  (\nabla^k_\nu A)'(h)=0 \mbox{ on } \hat \Sigma
\]
where the second fundamental form $A$ is defined in the neighborhood as the second fundamental form of equidistant hypersurfaces to $\hat \Sigma$.  
\end{remark}

\begin{proof}[Proof of Theorem~\ref{th:trivial'} (under infinite-order boundary condition)]
Without loss of generality, we may assume $(h, v)$ satisfies the static-harmonic gauge. 

Note that $\hat \Sigma$ is analytic and $(\bar g, \bar u)$ is analytic in some local coordinate chart $\{x_1, \dots, x_n\}$ near $\hat \Sigma$ by Theorem~\ref{th:analytic}. By Corollary~\ref{co:analytic}, $(h, v)$ is analytic up to $\hat \Sigma$ in $\{ x_1,\dots, x_n\}$.  Then by Corollary~\ref{co:geodesic}, there is an analytic vector field $Y$ in a collar neighborhood $U$ of $\hat \Sigma$  with $Y=0$ on $\hat \Sigma$ such that $\hat h:= h-L_Y \bar g$ satisfies the geodesic gauge in $U$. In particular, $\hat h$ is analytic. 

Recall that $h$ satisfies the infinite-order boundary condition. Item~\eqref{it:X} in Corollary~\ref{co:static} says that subtracting the ``gauge'' term $L_Y \bar g$ does not affect the infinite-order boundary condition, so $\hat h$ obtained above also satisfies the infinite-order boundary condition.  Applying Proposition~\ref{pr:infinite} to $\hat h$, we see that $\hat h$ vanishes at infinite order toward $\hat \Sigma$,  and thus $\hat h\equiv 0$  by analyticity. We then conclude $h\equiv L_Y \bar g$ in $U$.  Applying Theorem~\ref{th:extension}, there is a global vector field $X$ so that $X=Y$ in $U$, $X=0$ on $\hat \Sigma$,  and $L_X \bar g = h$  in $M\setminus \Omega$. By Lemma~\ref{le:trivial-h}, $X\in \mathcal X(M\setminus \Omega)$ and $v = X(\bar u)$. It completes the proof. 
\end{proof}

\section{Existence and local uniqueness}\label{se:solution}

In this section, we prove Theorem~\ref{existence}. We restate the theorem and also spell out the precise meaning of ``geometric uniqueness'' and ``smooth dependence'' of the theorem. Recall the diffeomorphism group $\mathscr D (M\setminus \Omega)$ in Definition~\ref{de:diffeo}. 
\begin{theorem}\label{exist}
Let  $(M, \bar g, \bar u)$ be an asymptotically flat, static vacuum with $\bar u>0$.  Suppose for $L$ defined on $M\setminus\Omega$
\begin{align} \label{eq:no-kernel2}
	\Ker L = \big\{ (L_X \bar g, X(\bar u)): X\in \mathcal X(M\setminus \Omega)\big\}. 
\end{align}
Then there exist positive constants $\epsilon_0,C$  such that for each $\epsilon\in (0, \epsilon_0)$, if $(\tau, \phi)$ satisfies $\|(\tau,\phi)-(g^\intercal,H)\|_{\C^{2,\a}(\Si)\times \C^{1,\a}(\Si)}<\epsilon$, there exists an asymptotically flat, static vacuum pair $(g,u)$ that satisfies both the static-harmonic and orthogonal gauges with $\|(g,u)-(\bar g,\bar u)\|_{\C^{2,\a}_{-q}(M\setminus \Omega)}<C\epsilon$ solving the boundary condition  $(g^\intercal, H_g) = (\tau, \phi) $ on $\Sigma$, and the solution $(g, u)$ depends smoothly on the Bartnik boundary data $(\tau, \phi)$.

Furthermore, the solution $(g, u)$ is locally, geometrically unique. Namely, there is a neighborhood $\mathcal U\subset \cM(M\setminus\Omega)$ of $(\bar g, \bar u)$  and $\mathscr D_0\subset \mathscr D(M\setminus\Omega)$ of the identity map $\mathrm{Id}_{M\setminus \Omega}$  such that if $( g_1, u_1)\in \mathcal U$ is another static vacuum pair with the  same boundary condition, there is a unique diffeomorphism $\psi\in \mathscr D_0$ such that $(\psi^* g_1, \psi^* u_1) = (g, u)$ in~$M\setminus \Omega$. 
\end{theorem}

Recall the operators $T, T^\Ga, L, L^\Ga$ defined in \eqref{eq:bdv}, \eqref{rsv}, \eqref{eq:L}, and \eqref{lsv} respectively.  The proof of Theorem~\ref{exist}  is along a similar line as the special case for Euclidean $(\bar g, \bar u) = (g_{\mathbb E}, 1)$ considered in \cite{An-Huang:2021}. We outline the proof: Under the ``trivial'' kernel assumption~\eqref{eq:no-kernel2} and by Corollary~\ref{co:static}, $\Ker L^\Ga $ is the $N$-dimensional space 
\[
	\Ker L^\Ga =\big\{ \big(L_X \bar g, X(\bar u) \big):X\in \mathcal X^\Ga (M\setminus \Omega) \big\}.
\]
 Since  the operator $L^\Ga$ is Fredholm of index zero by Lemma~\ref{le:Fred}, the cokernel of $L^\Ga$ must be also $N$-dimensional. In Lemma~\ref{le:range} below, we explicitly identify the cokernel by \eqref{equation:cokernel2} in Proposition~\ref{proposition:cokernel}. From there, we construct a modified operator $\overline{T^\Ga}$ whose linearization $\overline{L^\Ga}$ is an isomorphism in Proposition~\ref{pr:modified}.

\begin{lemma}\label{le:range}
Under the same assumptions as in Theorem~\ref{exist}, we have 
\begin{align*}
(\Range L^\Ga)^\perp&=\cQ
\end{align*}
where 
$\mathcal{Q}$ consists of elements $\kappa(X) \in \big(\C^{0,\alpha}_{-q-2}(M\setminus \Omega)\times \C^{1,\alpha}(\Sigma)\times \mathcal{B}(\Sigma)\big)^*$ (the dual space)  where 
\[
\kappa(X)=\Big(2\b^* X \!-\bar u^{-1} X(\bar u) \bar g, -\Div X +\bar u^{-1} X(\bar u),  \big(\!-2\bar u \b^* X + X(\bar u) \bar g\big) (\nu, \cdot ), \, 0, \, 0\Big)
\]
for  $X\in \mathcal X^\Ga(M\setminus \Omega)$. Recall $\b^* = \b_{\bar g}^*$ defined by \eqref{eq:Bianchi*}.

Furthermore, the codomain of $L^\Ga$ can be decomposed as 
\begin{align}\label{eq:decom}
	\Range L^\Ga \oplus \eta \mathcal Q
\end{align}
 where $\eta \mathcal Q=\big\{ \eta \kappa(X): \kappa(X)\in \mathcal Q\big\}
$ and $\eta$ is a positive smooth function on $M\setminus\Omega$ satisfying $\eta=1$ near $\Si$ and $\eta(x)=|x|^{-2}$ outside a compact set. 
\end{lemma}
\begin{remark}
Recall the definition of $\kappa_0(g, u, X)$ in \eqref{eq:kappa}. The first two components of $\kappa(X)$ above are exactly  $\kappa_0(\bar g, \bar u, X)$. 
\end{remark}
\begin{proof}
All the geometric operators are with respect to $\bar g$ in the proof. By Corollary~\ref{co:static}, $\Dim \Ker L^\Ga = N$.   Since $L^\Ga$ has Fredholm index zero by Lemma~\ref{le:Fred} and $\Dim \mathcal Q=N$, we just need to  show $\cQ\subseteq (\Range L^\Ga)^\perp$. Namely, we will show $\big\langle L^\Ga(h, v), \kappa(X)\big\rangle_{\cL^2}=0$ for any $(h, v)$ and $X\in \mathcal X^\Ga(M\setminus \Omega)$, where  $\langle \cdot , \cdot \rangle_{\cL^2}$ denotes the component-wise (standard) $\cL^2$-inner product.


Since $S'(h, v) $ is $\cL^2$-orthogonal to $\kappa_0(\bar{g}, \bar{u}, X)$ by  \eqref{equation:cokernel2} in Proposition~\ref{proposition:cokernel}, we can reduce the computation as follows, where we denote by $Z= \Ga'(h,v)$: 
\begin{align*}
	&\big\langle L^\Ga(h, v), \kappa(X)\big\rangle_{\cL^2} &\\
	&= \int_{M\setminus \Omega} \Big\langle \big( -\bar u\, \DD Z, -Z(\nabla \bar u) \big),  \big( 2\b^* X - \bar u^{-1} X(\bar u) \bar g,\, -\Div X + \bar u^{-1} X(\bar u)\big) \Big\rangle \, \dvol\\
	&\quad +\int_\Sigma\big( -2\bar u \b^* X + X(\bar u)\bar g\big) \big(\nu, Z\big) \, d\sigma\\
	&=\int_{M\setminus \Omega}  Z \Big(\Div \big(2 \bar u  \beta^* X - X(\bar u) \bar g \big) - \big( -\Div X + \bar u^{-1} X(\bar u) \big) d\bar u \Big) \, \dvol
\end{align*}
where we apply integration by parts in the last identity and note that the boundary terms on $\Sigma$ cancel. The last integral is zero because
\begin{align*}
	&\Div \big(2 \bar u  \beta^* X - X(\bar u) \bar g \big) - \big( -\Div X + \bar u^{-1} X(\bar u) \big) d\bar u \\
	&= 2 \bar u  \Div \beta^* X - d\big (X(\bar u)\big)+ 2\beta^* X (\nabla \bar u, \cdot)  - \big( -\Div X + \bar u^{-1} X(\bar u) \big) d\bar u\\
	&	= 2 \bar u  \Div \beta^* X - d\big (X(\bar u)\big)+L_X \bar g(\nabla \bar u, \cdot)  -   \bar u^{-1} X(\bar u)  d\bar u\\
	&= \bar u \Big( \, 2 \Div \beta^* X -\bar u^{-1} d\big (X(\bar u)\big)+\bar u^{-1} L_X \bar g(\nabla \bar u, \cdot)  -   \bar u^{-2} X(\bar u)  d\bar u\Big)\\
	&= -\bar u \, \Gamma(X)=0
\end{align*}
where in the second equality we use $2\beta^* X(\nabla \bar u, \cdot) = L_X \bar g(\nabla \bar u, \cdot) - \Div X d\bar u$ by definition, and  in the last line we use $2\Div \beta^* X =-\beta L_X \bar g$ and the definition of $\Gamma$ (see either \eqref{eq:gauge} or \eqref{eq:G}).

The previous identity also implies that $\Range L^{\Ga}\cap \eta \cQ=0$. To see that, let $\eta \kappa(X)\in \Range L^\Ga\cap \eta\cQ$. Then we have $\langle \eta \kappa(X), \kappa(X)\rangle_{\cL^2}=0$, and thus $\kappa(X)\equiv 0$. 


We verify \eqref{eq:decom}. With respect to the weighted inner product $\cL^2_\eta$, we denote by $\kappa_1, \dots, \kappa_N$ an orthonormal basis of $\cQ$. For any element $f$ in the codomain of $L^\Ga$, one can verify that
\[
	f - \eta \sum_{\ell=1}^N a_\ell \kappa_\ell \in \mathcal Q^{\perp}=\Range L^{\Ga}
\]
where the numbers $a_\ell = \langle f, \kappa_\ell \rangle_{\cL^2}$. We also use that  $\Range L^\Ga$ is a closed subspace, and hence $Q^{\perp}= ((\Range L^\Ga)^\perp)^\perp =\Range L^{\Ga} $. We get the desired decomposition. 
\end{proof}

Let $\rho$ be the weight function in Definition~\ref{de:gauge}. Define a complement to $\Ker L^\Ga$ in $\C^{2,\alpha}_{-q}(M)$ by
\bes
\cE=\left\{(h,v):\int_{M\setminus\Omega}\Big\langle (h,v), \big(L_X\bar g,X(\bar u)\big)\Big \rangle \rho\dvol=0\mbox{ for all }X\in\cX^\Ga (M\setminus \Omega) \right \}.
\ees
In other words, $(h, v)\in \cE$ if and only if $(\bar g +h, \bar u + v)$ satisfies the orthogonal gauge (recall Definition~\ref{de:gauge}). 

To summarize, we have decomposed the domain and codomain of $L^\Ga$ by
\[ 
L^\Ga :\cE \oplus \Ker L^\Ga\longrightarrow \Range L^\Ga \oplus \eta \cQ.
\]

Similarly as in \cite[Section 4.3]{An-Huang:2021},  we define the \emph{modified operator} $\overline {T^\Ga}$ as
\begin{align*}
\overline {T^\Ga}(g,u,W)=
&
\begin{array}{l}
\left\{ 
\begin{array}{l}
-u \Ric_{ g}+\nabla^2_{ g} { u}-u\DD_g {\Ga}(g,u)-\eta\big(2\b^*_gW-u^{-1} W(u)g\big)\\
\D_{ g} {u}-\Ga(g,u) (\nabla_g u) +\eta\, \big(\Div_g W-u^{-1} W(u)\big)\\
\Gamma_{(g,u)}(W)+\big(\!-\Ric_g+u^{-1}\nabla_g^2u\big)(W, \cdot )
\end{array}
\right. ~\mbox{in }M\setminus \Omega \\
\left\{
\begin{array}{l}
\Ga(g,u)+\big(2u\b^*_gW-W(u)g\big)(\nu_g, \cdot )\\
g^\intercal\\
H_g
\end{array}
\right. \quad\mbox{ on }\Si
\end{array}
\end{align*}
where $\Gamma_{(g,u)}(W)$ is as in \eqref{eq:Gamma}. The operator $\overline {T^\Ga}$ is defined on the function space of $(g, u)\in \mathcal U$ and $W\in \widehat{\cX}$. Here, $\mathcal U= \big((\bar g,\bar u)+\cE\big)\cap \cM(M\setminus \Omega)$ consists of asymptotically flat pairs satisfying the orthogonal gauge. The linear space $\widehat \cX$ is defined similarly as $\cX (M\setminus \Omega)$, with the only difference in regularity.\footnote{The slight technicality arises  to avoid a potential ``loss of derivatives'' issue. Some coefficients in the third equation of $\overline {T^\Ga}$ are only $\C^{0,\alpha}$, e.g. $\Ric_g$. If we use the space $\cX(M\setminus \Omega)$ (of $\C^{3,\alpha}$-regularity) instead of $\widehat \cX$, the codomain of the map may still be only in $\C^{0,\alpha}_{-q}$, and the map cannot be surjective.)} Explicitly, 
\begin{align*}
\widehat \cX=\Big\{&X\in \C^{2,\alpha}_{\mathrm{loc}}(M\setminus \Omega):\,  X=0 \mbox{ on } \Sigma \mbox{ and }X-Z= O^{2,\a}(|x|^{1-q}) \mbox{ for some } Z\in \mathcal Z\Big\}.
\end{align*}
In other words,  $\cX (M\setminus \Omega)$
can be view as a subspace of $\widehat\cX$ with $\C^{3,\alpha}$  regularity. With the above choices of function spaces,  we get  
\[
\overline {T^\Ga}: \mathcal U \times \widehat {\mathcal X} \longrightarrow  \C^{0,\a}_{-q-2}(M\setminus \Omega)\times \C^{0,\a}_{-q-1}(M\setminus \Omega )\times \C^{1,\a}(\Si)\times \mathcal{B}(\Sigma).
\]
(Note that we slightly abuse the notation and blur the distinction of $W$ and its covector with respect $g$.)

\begin{proposition}\label{pr:modified}
Under the same assumptions as in Theorem~\ref{exist}, we have 
\begin{enumerate}
\item $\overline {T^\Ga}$ is  a smooth local diffeomorphism at $(\bar g,\bar u,0)$.\label{it:diffeo}
\item If $\overline {T^\Ga}(g,u,W)=(0,0,0,0,\tau,\phi)$, then $W=0$ and $T^\Ga(g,u)=(0,0,0,\tau,\phi)$. \label{it:solution}
\end{enumerate}
\end{proposition}
\begin{proof}

We prove Item~\eqref{it:diffeo}. Since in local coordinates $\overline {T^\Ga}(g, u, W)$ can be expressed as locally bounded, polynomials in $\partial^i g, \partial^j u, \partial^k W$ for $i, j, k=0, 1, 2$, $\overline {T^\Ga}$ is a smooth map. We show that the linearization of $\overline {T^\Ga}$ at $(\bar g,\bar u,0)$, denoted by $\overline{ L^\Ga}$, is an isomorphism, where 
\[
   \overline{ L^\Ga}: \mathcal E\times \widehat {\mathcal X } \longrightarrow  \C^{0,\alpha}_{-q-2}(M\setminus \Omega) \times \C^{0,\alpha}_{-q-1}(M\setminus \Omega)\times \C^{1,\alpha}(\Sigma) \times \mathcal B(\Sigma).
\]For $(h, v)\in \mathcal E$ and  $X\in  \widehat{\cX}$,  $\overline{ L^\Ga} (h, v, X)$ is given by
\bes
\begin{split}
&
\begin{array}{l}
\left\{
\begin{array}{l}
-\bar u \Ric'(h) + (\nabla^2)'(h) \bar u - v\Ric + \nabla^2 v -\bar u \DD \Ga'(h, v) -\eta\Big(2\b^*X-\bar u^{-1} X(\bar u)\bar g\Big)\\
\D v+\D'(h)\bar u-\Ga'{(h,v)}(\nabla \bar u)+\eta\Big(\Div X-\bar u^{-1}X(\bar u)\Big)\qquad \qquad \qquad \mbox{ in }M \setminus \Omega \\
\Gamma(X)
\end{array} \right. \\
\left\{
\begin{array}{l}
\Ga'{(h,v)}+(2\bar u\b^*X-X(\bar u)\bar g)(\nu)\\
h^\intercal\\
H'(h)
\end{array}\right.\quad\mbox{ on }\Si.
\end{array}
\end{split}
\ees
Note that we use $X$ to denote both the vector field and the dual covector, and in the third component recall $\Gamma(X)$ defined ~\eqref{eq:gauge}. Observe that when dropping the third component, we can understand $\overline{L^\Ga}(h,v,X)$ as the sum $L^\Ga(h,v)+\eta \kappa(X)$. 

We show that $\overline{L^\Ga}$ is an isomorphism. To see that $\overline{L^\Ga}$ is surjective: Since the third component $\Gamma(X)$ is surjective by Lemma~\ref{PDE}, we just need to show that $\overline{ L^\Ga}$ is surjective onto other components for those $X$ satisfying $\Gamma(X)=0$, i.e. $X\in \mathcal X^\Ga (M\setminus \Omega)$. It is equivalent to showing that $L^\Ga(h,v)+\eta\kappa(X)$ is surjective for $(h, v)\in \mathcal E$ and for $X\in \cX^\Ga (M\setminus \Omega)$, which follows from Lemma~\ref{le:range}. To see that  $\overline{L^\Ga}$ is injective: If $(h, v)\in \cE$ and $X\in \widehat {\cX}$ solves $\overline{ L^\Ga}(h,v,X)=0$, then $L^\Ga(h,v)+\eta \kappa(X)=0$ and $\kappa(X)\in \cQ$, which implies $L^\Ga (h, v)=0$ and $\kappa(X)=0$ by the decomposition \eqref{eq:decom} in Lemma~\ref{le:range}. From there we can conclude $(h, v)=0$ and $X=0$.

For Item~\eqref{it:solution}, suppose $\overline {T^\Ga}(g,u,W)=(0,0,0,0,\tau,\phi)$.  Then
\begin{align*}
	-u \Ric_{ g}+\nabla^2_{ g} { u}-u \DD_g {\Ga}(g,u)&=\eta\, \big(2\b^*_gW-u^{-1}W(u)g\big)\\
	\D_{ g} {u}-\Ga(g,u) (\nabla_g u) &=-\eta\, \big(\Div_g W-u^{-1} W(u)\big).
\end{align*}
We denote  $\Ga=\Ga(g, u)$ in the following computations. Pair the first equation with $2\b^*_g W- u^{-1} W(u)g$ and the second equation with $-\big(\Div_gW-u^{-1}W(u)\big)$:
\begin{align*}
&\int_{M\setminus \Omega }\eta\,  \Big|2\b^*_gW-u^{-1}W(u)g\big|^2+\eta \, \Big|\Div_g W-u^{-1}W(u)\big|^2\dvol_g\\
&=\int_{M\setminus \Omega}\bigg\langle \Big(-u\Ric_{g}+\nabla^2_{ g} u-u\DD_g\Ga, \Delta_{ g} {u}-\Ga(\nabla_g u)\Big),\\
&\qquad \qquad \qquad\qquad  \big(2\b^*_gW-u^{-1} W(u)g\,, - \Div_g W +u^{-1}W(u)\big)\bigg\rangle\dvol_g\\
&=\int_{M\setminus \Omega}\bigg\langle \Big(-u\DD_g\Ga,-\Ga(\nabla_g u)\Big),\big(2\b^*_gW-u^{-1} W(u)g\,, - \Div_g W +u^{-1}W(u)\big)\bigg\rangle\dvol_g\\
&=\int_\Sigma \Big\langle \Ga, \big(2u\beta^*_g W - W(u) g \big)(\nu_g) \Big\rangle \, d\sigma_g\\
&= \int_\Sigma - |\Ga|^2 \, d\sigma_g \qquad \Big(\mbox{use $-\Ga = \big(2u\beta^*_g W - W(u) g\big)(\nu_g)$ from the equation for $\overline {T^\Ga}$}\Big).
\end{align*}
In the above identities,  we use that in the second equality the $\cL^2$-pairing involving $(-u \Ric_g + \nabla^2_g u, \Delta_g u)$ is zero by \eqref{equation:cokernel1}. Then in the third equality we apply integration by parts and note that the integral over $M\setminus \Omega$ is zero because 
\begin{align*}
	&\Div_g  \big(u \, 2\b^*_gW-W(u)g\big)-\big(-\Div_g W+u^{-1}W(u)\big)d u\\
	&= 2 u \Div_g \beta^*_g W - d(W(u)) + L_W g (\nabla_g u, \cdot) - u^{-1} W(u) du \\
	&=- u \Big( \beta_g L_W g + u^{-1} d(W(u))- u^{-1} L_W g(\nabla_g u, \cdot) + u^{-2} W(u) du \Big)\\
	&=- u \Big(-\Delta_g W - u^{-1}  \nabla_g W(\nabla_g u, \cdot) + u^{-2} W(u) du + \big( -\Ric_g + u^{-1} \nabla^2_g u\big) (W) \Big)\\
	&=- u \Big(\Gamma_{(g,u)}(W) + \big( -\Ric_g + u^{-1} \nabla^2_g u\big) (W) \Big)\\
	&=0
\end{align*}
where we compute similarly as in  \eqref{eq:G} in the third equality and use $\overline {T^\Ga}(g,u,W)=(0, 0, 0, 0, \tau, \phi)$  in the last equality.

The previous integral identity implies $2\b^*_gW-u^{-1}W(u)g=0$ and $\Div W-u^{-1}W(u)=0$, so we conclude  $W\equiv 0$.
\end{proof}

\begin{proof}[Proof of Theorem~\ref{exist}]
From Item~\eqref{it:diffeo} in Proposition~\ref{pr:modified}, $\overline{T^\Ga}$ is a local diffeomorphism at $\big((\bar g, \bar u), 0\big)$. That is, there are positive constants $\epsilon_0, C$ such that for every $0 < \epsilon <\epsilon_0$, there is an open neighborhood $\mathcal U\times \mathcal V$ of $\big((\bar g, \bar u), 0\big)$ in $\big((\bar g, \bar u) + \mathcal E \big)\times \widehat \cX$ with the diameter less than $C\epsilon$ such that $\overline{T^\Ga}$ is a diffeomorphism from $\mathcal U\times \mathcal V$ onto an open ball of radius $\epsilon$ in the codomain of $\overline{T^\Ga}$. 

Therefore, given any $(\tau, \phi)$ satisfying $\|(\tau , \phi) - (\bar g^\intercal, H_{\bar g})\|_{\C^{2,\alpha}(\Sigma)\times \C^{1,\alpha}(\Sigma) } < \epsilon$,  there exists a unique $(g, u, W)$ with $\|(g, u) - (\bar g, \bar u ) \|_{\C^{2,\alpha}_{-q}(M\setminus \Omega)} <C\epsilon$ and $\| W \|_{\C^{2,\alpha}_{1-q}(M\setminus \Omega)} < C \epsilon$ satisfying
\[
	\overline{T^\Ga}(g, u , W) = (0, 0, 0, 0, \tau, \phi)
\]
and depending smoothly on $(\tau, \phi)$.

By Item~\eqref{it:solution} in Proposition~\ref{pr:modified}, we have $T^\Ga(g, u) = (0, 0, 0, \tau, \phi)$. By Lemma~\ref{rsv-to-sv}, $(g, u)$ satisfies the static-harmonic gauge  and $T(g, u)=(0, 0, 0, \tau, \phi)$. By the definition of the complement space $\mathcal E$ and  $(g, u)\in (\bar g, \bar u) + \mathcal E$, $(g, u)$ also satisfies the orthogonal gauge. 
\end{proof}

\begin{remark}\label{re:constant}
The constants $\epsilon_0, C$ in the above proof depend on the global geometry $\Sigma$ in $M\setminus \Omega$. More precisely, by Inverse Function Theorem, the constants depend on the operator norms of $\overline{L^\Ga}$, $\overline{L^\Ga}^{-1}$, as well as the second Frech\'et derivative $D^2\overline{T^\Ga}|_{(g, u)}$ for $(g, u)$ in a neighborhood of $(\bar g, \bar u)$,  between those function spaces as specified above.
\end{remark}

\section{Perturbed hypersurfaces}\label{se:perturbation}

Through out this section, we let $(M,\bar g,\bar u)$ be an asymptotically flat, static vacuum triple with $\bar u>0$. In this section, we prove Theorem~\ref{generic}, which follows directly from Theorem~\ref{th:zero} below, and then Corollary~\ref{co2}.

Let $\{ \Sigma_t\}$ be a smooth one-sided family of hypersurfaces generated by $Y$ foliating along $\hat \Sigma_t\subset \Sigma_t$  with $M\setminus \Omega_t$ simply connected relative to  $\hat \Sigma_t$, as defined in Definition~\ref{def:one-sided}.
We extend $Y$ to entire $M\setminus \Omega $ that is supported in a bounded open subset containing~$\{\Si_t\}$.
Let $\psi_t: M\setminus\Omega\longrightarrow M\setminus\Omega$ be the flow of $Y$. If we denote by $\Omega = \Omega_0$ and $\Sigma = \Sigma_0$, then $\Omega_t = \psi_t (\Omega)$ and $\Sigma_t = \psi_t (\Sigma)$.
Denote the pull-back static pair defined on $M\setminus \Omega$ by 
 \[
 (g_t, u_t) = \psi_t^* \big( \left.\big(\bar g,  \bar u\big)\right|_{M\setminus \Omega_t}\big).
 \]
In that notation $(g_0,u_0) = (\bar g, \bar u)$. 

Recall the operator $L$ defined in \eqref{eq:L}, which is the linearization of $T$ at $(\bar g,\bar u)$ in $M\setminus\Omega$. We shall consider the corresponding family of operators $L_t$ raised by linearization of $T$ at $(g_t,u_t)$ in $M\setminus\Omega$:
 \begin{align*}
 L_t(h, v)=  \begin{array}{l} \left\{ \begin{array}{l}
-u_t \Ric'|_{g_t} (h) + \left.(\nabla^2)'\right|_{g_t} (h) u_t - v \Ric_{g_t} + \nabla^2_{g_t} v\\
\Delta_{g_t} v + \left.\Delta'\right|_{g_t} (h) u_t
\end{array}\right. \quad  {\rm in }~M\setminus \Omega \\
\left\{ \begin{array}{l}
	h^\intercal\\
	H'|_{g_t}(h)
\end{array}\right.\quad {\rm on }~\Sigma.
\end{array}
\end{align*}
As is for $L$, we need to add gauge terms to modify $L_t$ for the sake of ellipticity. To that end, we generalize the static-harmonic gauge $\Ga(g,u)$ with respect to $(\bar g,\bar u)$ defined in Section 3.1 to a gauge term $\Ga_t(g,u)$ with respect to $(g_t,u_t)$ and consider its linearization $\Ga'_t$ at $(g_t,u_t)$:
\begin{align*}
	\Ga_t(g, u) &:= \beta_{g_t} g + u_t^{-2} u du - u_t^{-1} g(\nabla_{g_t} u_t, \cdot )\\
	\Ga_t'(h, v)&:= \beta_{g_t} h + u_t^{-1} dv + v u_t^{-2} du_t - u_t^{-1} h(\nabla_{g_t} u_t, \cdot ).
\end{align*}
Then we define the  family of operators $L^{\Ga}_t$ that have the same domain and co-domain as $L^\Ga$ in \eqref{lsv}, where  
\begin{align*}
 L^{\Ga}_t(h, v)=\begin{array}{l} \left\{ \begin{array}{l}
-u_t \Ric'|_{g_t} (h) + \left.(\nabla^2)'\right|_{g_t} (h) u_t - v \Ric_{g_t} + \nabla^2_{g_t} v-u_t\DD_{g_t}  {\Ga}'_t(h,v)\\
\Delta_{g_t} v + \left.\Delta'\right|_{g_t} (h) u_t-{\Ga}'_t(h,v) (\nabla_{g_t}u_t)
\end{array}\right. \quad  {\rm in }~M\setminus \Omega \\
\left\{ \begin{array}{l}
	\Ga'_t(g, u) \\
	h^\intercal\\
	H'|_{g_t}(h)
\end{array}\right.\quad {\rm on }~\Sigma.
\end{array}
\end{align*}

Note that in our notations $L_0^\Ga=L^\Ga$ and $L_0 = L$.
It is direct to verify that the results in Lemma~\ref{le:linear}, Lemma~\ref{le:Fred} and Corollary~\ref{co:analytic} are also true for $L_t$ and $L_t^\Ga$ with $(\bar g,\bar u)$ replaced by $(g_t,u_t)$.

\medskip
\noindent {\bf {Motivation for the proof of Theorem~\ref{generic}}:}  Most part of this section, except Theorem~\ref{th:zero}, follows closely the approach in \cite[Section 6]{An-Huang:2021}.

To provide motivation, we begin by explaining our approach in a simplified case. Consider a region $\Omega$ with a smooth boundary $\Sigma$ in the Euclidean space $(\mathbb R^n, \bar g)$. For a real number $C$, we define  the elliptic operator $Su=\Delta_{\bar g} u + Cu$. Consider the boundary value problem:
\[
L u=\left\{ \begin{array}{ll} Su& \mbox{ in } \Omega\\
 u & \mbox{ on } \Sigma \end{array}\right..
\]
Now, let $\Omega_t\subset \Omega$, $t\in [-\delta, 0]$ be smoothly inward deformed subsets of $\Omega$ with $\Omega_0=\Omega$, and we assume the diffeomorphisms $\psi_t:\Omega\to \Omega_t$ generated by a vector field $Y$ satisfying $Y|_\Sigma = \zeta \nu$ for some $\zeta>0$ where $\nu$ is the inward unit normal.

We shall show that the operator $L$, when restricted to $\Omega_t$, has a trivial kernel for generic values of $t$. This is equivalent to showing that, for pull-back operators $S_t u:= \Delta_{\psi_t^*\bar g} u + Cu$, the following map:
\[
L_t u=\left\{ \begin{array}{ll} S_tu& \mbox{ in } \Omega\\
 u & \mbox{ on } \Sigma \end{array}\right.
\]
has a trivial kernel for generic values of $t$.


Assume that for $t=a$, the kernel of $L_a$ is not trivial. Consider a nontrivial solution $u$ such that $L_au=0$.  Without loss of generality, we can set $a=0$ and thus we may denote $L_0 = L$. Now, let's assume that there exists nontrivial $u(t)$ to the equation $L_t u(t) = 0$ as $t\to 0$ with $u(0)=u$ and differentiable in $t$. While technically we may only have the convergence as a sequence $t_j\to 0$ (as described in Proposition~\ref{pr:limit}), but the argument below can still be extended. By the invariance of $S_t$ under diffeomorphisms, we deduce that $\psi_t^* u$ also satisfies the equation $S_t ( \psi_t^* u )= 0$. (Here, $\psi_t^* u$ denotes the pull-back of  $u|_{\Omega_t}$.) When we differentiate the equation $S_t(u(t) - \psi_t^* u) = 0$ in $t$, we obtain $S (u'(0) - Y(u) )=0$. This can also be understood as follows: computing the Taylor expansion of $S_t(u(t) - \psi_t^* u) = t S (u'(0) - Y(u)) + O(t^2)$, the coefficient of $t$ must be zero.  Calculating the boundary value, we find that $u'(0) - Y(u) = -Y(u) = -\zeta \nu(u)$ on $\Sigma$. Using the classical Green identity, we obtain
\begin{align*}
	0= \int_\Sigma (u'(0) - Y(u)) \nu(u) \, \da= \int_\Sigma \zeta  \nu(u)^2 \, \da.
\end{align*}
This leads $\nu(u)\equiv 0$ on $\Sigma$. Consequently, since $u=0$ on $\Sigma$,  by unique continuation, we can infer that $u\equiv 0$ in $\Omega$, which contradicts our initial assumption.

In our current context, the geometric PDE that we are dealing with are significantly more challenging due to several factors. Some challenges arise from gauge issues and the coupled  system involving $(h, v)$. Another challenge arises during the application of the Green-type identity. We can only obtain the boundary condition for $h$ and not for $v$. Specifically, the first application of the approach allows us to show $A'(h) = 0$ only. This step significantly complicates the ``unique continuation'' argument when compared to the corresponding result in \cite{An-Huang:2021}. The novel argument presented in Theorem~\ref{th:zero} shows that, by applying  the above argument inductively,  not only does $h$ vanish, but the higher-order derivatives of $h$  also vanish on the boundary. Then we invoke analyticity to establish that $h$ is trivial.

\qed


\medskip
Just as the operator $L^\Ga$, each $L_t^\Ga$ is elliptic and thus Fredholm as in Lemma \ref{le:Fred}. We can use elliptic estimates to show that $\Ker L_t^\Ga$ varies ``continuously'' for $t$ in an open dense subset  $J\subset [-\delta, \delta]$. The arguments follow verbatim as in \cite[Section 6.1]{An-Huang:2021}, so we omit the proof. (Note the notation discrepancy:  $L_t$ and $L_t^{\Ga}$ here correspond respectively to $S_t$ and $L_t$ in \cite{An-Huang:2021}).

\begin{proposition}[Cf. {\cite[Proposition 6.6]{An-Huang:2021}}]\label{pr:limit}
There is an open dense subset $J\subset (-\d, \d)$ such that for every $a\in J$ and every $(h, v)\in \Ker L_a^\Ga$, there is a sequence $\{ t_j \} $ in $J$ with $t_j \searrow a$, $\big(h(t_j), v(t_j)\big)\in \Ker L_{t_j}^\Ga$, and $(p, z)\in \C^{2,\alpha}_{-q}(M\setminus \Omega)$ such that, as $t_j \searrow a$, 
\begin{align*}
	\big( h(t_j), v(t_j) \big) &\longrightarrow(h,v)\\
	\frac{\big( h(t_j), v(t_j) \big) - (h, v)  }{t_j - a}&\longrightarrow (p, z)
\end{align*} 
where the both convergences are taken in the $\C^{2,\alpha}_{-q}(M\setminus \Omega)$-norm. 
\end{proposition}

It is more convenient to consider the above convergence for the ``un-gauged'' operators as in the next corollary.

\begin{corollary}\label{co:limit2}
Let $J\subset (-\d, \d)$ be the open dense subset in Proposition~\ref{pr:limit}. Then for every $a\in J$ and every $(h, v)\in \Ker L_a$, there is a sequence $\{ t_j \} $ in $J$ with $t_j \searrow a$, $\big(h(t_j), v(t_j)\big)\in \Ker L_{t_j}$, and $(p, z)\in \C^{2,\alpha}_{-q}(M\setminus \Omega)$ such that, as $t_j \searrow a$, 
\begin{align*}
	\big( h(t_j), v(t_j) \big) &\longrightarrow(h,v)\\
	\frac{\big( h(t_j), v(t_j) \big) - (h, v)  }{t_j - a}&\longrightarrow (p, z)
\end{align*} 
where the both convergences are taken in the $\C^{2,\alpha}_{-q}(M\setminus \Omega)$-norm. 
\end{corollary}
\begin{proof}
Let $(h, v)\in \Ker L_a$. By Item \eqref{it:SH} in Lemma~\ref{le:linear}, there is a vector field $V\in \mathcal X(M\setminus \Omega)$ such that 
\[
(\hat h, \hat v):=\big(h+L_V \bar g, v+ V (\bar u)\big)
\]
 satisfies $\Ga_a'(\hat h, \hat v)=0$, and thus $(\hat h, \hat v)\in \Ker L^\Ga_a$. By Proposition~\ref{pr:limit}, there exists a sequence $t_j\in J$ with $t_j \searrow a$ and $(\hat h(t_j), \hat v(t_j))\in \Ker L^\Ga_{t_j}$, $(\hat p,\hat z)\in \C^{2,\alpha}_{-q}(M\setminus \Omega)$ such that as $t_j \searrow a$, 
\begin{align*}
	\big(\hat  h(t_j),\hat  v(t_j) \big) &\longrightarrow (\hat h, \hat v)\\
	\frac{\big( \hat h(t_j), \hat v(t_j) \big) - (\hat h, \hat v)}{t_j - a}&\longrightarrow (\hat p, \hat z).
\end{align*} 
We now define 
\begin{align*}
	\big( h(t_j), v(t_j) \big) &=\big( \hat h(t_j) , \hat v(t_j) \big) -\big (L_V g_{t_j}, V( u_{t_j} ) \big)\\
	(p, z) &= (\hat p, \hat z) -\Big(L_V L_Y g_a, V \big(Y(u_a)\big) \Big)
\end{align*}
where recall that $Y$ denotes the deformation vector of $\{ \Sigma_t \}$. It is direct to verify the desired convergences. 
\end{proof}

 For each $t$, we let $\Sigma^+_t\subset \Sigma $ be the subset $\psi_t^{-1}\big(\{p\in\Sigma_t: \zeta(p)>0\}\big)$, and write  $\Sigma^+=\Sigma^+_0$. Note that $\Sigma_t^+$ contains $\psi_t^{-1} (\hat \Sigma_t)$.
\begin{theorem}\label{th:zero}
Let $J\subset (-\d, \d)$ be the open dense subset in Proposition~\ref{pr:limit}. Then for every $a\in J$ and every $(h, v)\in \Ker L_a$, we have
\begin{align*}
	A'|_{g_a}(h)=0,\quad \big(\Ric'|_{g_a}(h)\big)^\intercal=0,\quad \Big(\big(\nabla^k_\nu \, \Ric^\intercal \big)'\big|_{g_a}(h)\Big)^\intercal = 0 \quad  \mbox{ on } \Sigma^+_a
	\end{align*}
for all positive integers $k$. 
\end{theorem}

\begin{proof}

Let $(h, v)\in \C^{2,\alpha}_{-q}(M\setminus \Omega)$ solve $L_a (h, v)=0$. We first prove  $A'|_{g_a}(h)=0$ on $\Sigma$ for all $a\in J$.

We may without loss of generality assume $a=0$ and that $(h, v)$ satisfies the geodesic gauge $h(\nu,\cdot) = 0$  on $\Sigma$. (See \cite[Lemma 2.5]{An-Huang:2021}.)  Let $\big(h(t_j), v(t_j) \big)\in \Ker L_{t_j}$ and $(p, z)$ be from Corollary~\ref{co:limit2}. We \emph{claim} that $\big(p-L_Y h, z-Y(v)\big)$ is a static vacuum deformation, i.e. $S'\big(p-L_Y h, z-Y(v)\big)=0$ in $M\setminus\Omega$, with the Bartnik boundary data:
\begin{align}\label{eq:perturbed-bdry}
\begin{split}
	(p-L_Y h)^\intercal &= -2\zeta A'(h)\\
	H'(p-L_Y h) &= \zeta A\cdot A'(h). 
\end{split}
\end{align}
We recall that  $\zeta=\bar g(Y,\nu)$ on $\Si$ in Definition~\ref{def:one-sided}.

We compare $\big(h(t_j), v(t_j) \big)$ and the pull-back pair $\psi_{t_j}^* (h, v)$. Since they are equal at $t_j=0$ and both satisfy
\[
	S'|_{g_{t_j}}\big(h(t_j), v(t_j) \big)=0 \quad \mbox{ and }\quad  S'|_{g_{t_j}}\big(\psi_{t_j}^* (h, v)\big)=0\quad \mbox{ in } M\setminus \Omega. 
\]	
{We subtract the previous two identities and take the difference quotient:
\begin{align*}
0&=\lim_{t_j\to 0}\frac{1}{t_j} S'|_{g_{t_j}}\big((h(t_j), v(t_j))-\psi_{t_j}^* (h, v) \big)\\
&= \left(\lim_{t_j\to 0} \frac{1}{t_j} (S'|_{g_{t_j}} - S'|_{\bar g}) \right)\big((h(0), v(0))-\psi_0^* (h, v) \big)\\
&\quad +  S'|_{\bar g} \left(\lim_{t_j\to 0} \frac{1}{t_j} \big((h(t_j), v(t_j))-\psi_{t_j}^* (h, v) \big)\right) \\
 &= S'|_{\bar g} \big(p-L_Y h, z- Y(v) \big)
\end{align*}
where we use that $(h(0), v(0))=\psi_0^* (h, v) =(h,v)$.}
For the boundary data, since  $(h^\intercal, H'(h))=0$ and $\big( h(t_j)^\intercal, H'|_{g_{t_j}} (h(t_j)) \big)=0$ for all $t_j$ on $\Sigma$, we have 
\begin{align*}
	p^\intercal &= \lim_{t_j\to 0}\frac{ \big(h(t_j)-h\big)^\intercal}{t_j} = 0\\
	H'(p-L_Yh)& =\lim_{t_j\to 0} \frac{1}{t_j} H'|_{g_{t_j} } \big( h(t_j)- \psi^*_t(h) \big)\\
	&=-\lim_{t_j\to 0} \frac{1}{t_j} \psi_t^*\big(H'|_{g}(h)\big)=-Y(H'(h)),
\end{align*}
where we recall $H'|_{g_{t_j}}(h(t_j))=0$. By \eqref{eq:sff} and \eqref{eq:Ricatti}, we have 
$(L_Y h)^\intercal = 2\zeta A'(h)$ and $Y\big(H'(h)\big)=- \zeta A\cdot A'(h)$, and it completes the proof of the claim. 

We apply the Green-type identity, Corollary~\ref{co:Green}, for $(h, v)$ and $(k, w):= \big(p-L_Y h, z-Y(v) \big)$ and get 
\begin{align*}
	 0 &= \int_\Sigma \Big \langle Q(h, v), \big( k^\intercal, H'(k) \big) \Big\rangle d\sigma\\
	 &= \int_\Sigma \Big \langle \big( vA + \bar u A'(h) - \nu (v) g^\intercal, 2v \big), \big(  -2\zeta A'(h),\zeta A\cdot A'(h) \big) \Big\rangle\, d\sigma\\
	 &= -\int_\Sigma 2\bar u\zeta |A'(h)|^2 \, d\sigma
\end{align*}	
where in the second equality we use the definition of $Q(h, v)$ (noting $\nu'(h)=0$ in geodesic gauge) and in the last equality $g^\intercal \cdot A'(h) = 0$. This shows that $A'(h)=0$ on $\Sigma^+$. 

It follows that for all $(h, v)\in \Ker L_a~(a\in J)$, we have $A'|_{g_a}(h)=0$ on $\Sigma^+_a$, because $A'(h)$ is gauge invariant by Corollary \ref{co:static} Item (3).  So far, the argument follows closely Theorem~7{$^\prime$} in \cite{An-Huang:2021}.

We further observe that for $(k, w)= (p-L_Y h, z- Y(v))$ as defined above, its Bartnik boundary data is zero from \eqref{eq:perturbed-bdry}, and thus $(k,w)\in \Ker L$. Therefore, we must have $A'(k)=0$ on $\Sigma^+$:
\be
\begin{split}\label{limit}
	 0 &= A'(p-L_Yh)\\
	 &= \lim_{t_j\to 0}\frac{1}{t_j}A'|_{g_{t_j}}\big(h(t_j)-\psi_{t_j}^*(h)\big) = \lim_{t_j\to 0}\frac{1}{t_j}A'|_{\psi_{t_j}^*g}(-\psi_{t_j}^*(h))\\
	 &=-\lim_{t_j\to 0}\frac{1}{t_j}\psi_{t_j}^*\big( A'(h) \big)=-L_Y\big(A'(h)\big)=-\zeta\nabla_\nu \big( A'(h) \big).
\end{split}
\ee
In the second line above, we use that $A'|_{g_{t_j}}(h(t_j))=0$ on $\Sigma^+_{t_j}$ for all $t_j$; and since the deformation vector field $Y$ is smooth, $\Si_{t_j}^+\to\Si^+$ as $t_j\to 0$.   
In the last equality we use that $A'(h)=0$ on $\Sigma^+$. Thus we obtain $\nabla_\nu \big(A'(h)\big)=0$ on $\Sigma^+$. By the formula~\eqref{eq:sff} for $A'(h)$ and noticing $h=0, \nabla h=0$ on $\Sigma$ (because of $h^\intercal =0, A'(h)=0$ and geodesic gauge), we obtain 
\bes
(\nabla_\nu^2 h)^\intercal =0\ \ \mbox{on } \Sigma^+.
\ees
It implies $\big(\Ric'(h)\big)^\intercal=0$ by  \eqref{equation:Ricci}. Again, since $\big(\Ric'(h)\big)^\intercal$ is gauge invariant, this holds for all $(h,v)\in \Ker L_a~(a\in J)$.

We proceed to prove inductively in $k$  that  for all $a\in J$ and for all $(h, v)\in \Ker L_a$ we have $ \big((\nabla_\nu^k \Ric )'\big|_{g_a}(h)\big)^\intercal=0$ on $\Sigma^+_a$. In the previous paragraph, we prove the statement for $k=0$, i.e. $\big(\Ric'|_{g_a}(h)\big)^\intercal=0$ on $\Sigma^+_a$ for all $(h,v)\in \Ker L_a$.
Suppose the inductive assumption holds for $k\le \ell$, i.e. for all $a\in J$ and all $(h,v)\in \Ker L_a$, $\big( (\nabla^k_\nu \Ric)'(h) \big)^\intercal=0$ on $\Sigma^+_a$ for $k=0,1,...,\ell$. We prove the statement holds for  $k=\ell+1$. Let $(h, v)\in  \Ker L_a$. We may without loss of generality assume $a=0$ and that $h$ satisfies the geodesic gauge. Let $(p, z)$ be defined as above. Because of \eqref{eq:perturbed-bdry} and $A'(h)=0$, we see that $ \big(p-L_Y h, z-Y(v) \big)\in \Ker L$. Therefore, we can apply the inductive assumption for  $\big(p-L_Y h, z-Y(v) \big)$ and get 
\[
	\big((\nabla^\ell_\nu \Ric )'(p-L_Y h)\big)^\intercal=0 \ \ \mbox{ on } \Sigma^+. 
\]
Similar computations as in \eqref{limit} yield
\begin{align*}
	 0 &= \big((\nabla^\ell_\nu \Ric )'(p-L_Y h)\big)^\intercal= \lim_{t_j\to 0}\frac{1}{t_j}\Big((\nabla^\ell_\nu\Ric )'|_{g_{t_j}}\big(h(t_j)-\psi_{t_j}^*(h) \big) \Big)^\intercal\\
	 & = -\lim_{t_j\to 0}\frac{1}{t_j}\Big((\nabla^\ell_\nu\Ric \big)'|_{g_{t_j}}(\psi_{t_j}^*(h))\Big)^\intercal =-\lim_{t_j\to 0}\frac{1}{t_j}\Big(\psi_{t_j}^*\big( (\nabla^\ell_\nu\Ric )'(h)\big) \Big)^\intercal\\
	   &=-\Big(L_Y\big( (\nabla^\ell_\nu \Ric^\intercal \big)'(h) \big)\Big)^\intercal=-\zeta \Big(\nabla_\nu \big( (\nabla^\ell_\nu \Ric \big)'(h)\big) \Big)^\intercal.
\end{align*}
It implies that $\Big(\nabla_\nu \big( (\nabla^\ell_\nu \Ric \big)'(h)\big) \Big)^\intercal=0$ on $\Sigma^+$, and thus  $\big((\nabla_\nu^{\ell+1} \Ric)'(h)\big)^\intercal=0$ on $\Sigma^+$ by \eqref{eq:commutative}. 
 \end{proof}

\begin{proof}[Proof of Theorem~\ref{generic}]
So far we have not used the assumption that $\pi_1(M\setminus \Omega_t, \hat \Sigma_t)=0$ nor that $\hat \Sigma_t$ is analytic. Using those assumptions and Theorem~\ref{th:zero}, we see that for $t\in J$, $\Sigma_t$ is static regular of type (II) in $(M\setminus \Omega_t, \bar g, \bar u)$. 
\end{proof}
For the rest of this section, we discuss how  Corollary~\ref{co2} follows directly from Theorem~\ref{existence} and Theorem~\ref{generic}.

\begin{proof}[Proof of Corollary~\ref{co2}] 
Fix the background metric as a Schwarzschild manifold $(\mathbb R^n\setminus B_{r_m},  g_m,  u_m)$, where $m>0$ and 
\[
	r_m = (2m)^{\frac{1}{n-2}}, \quad g_m = \left( 1- \tfrac{2m}{r^{n-2}}\right)^{-1} dr^2 + r^2 g_{S^{n-1}}, \quad u_m = \sqrt{1-\tfrac{2m}{r^{n-2}}}.
\]
Denote the $(n-1)$-dimensional spheres $S_r = \{ x\in \mathbb{R}^n: |x| = r \}$. Then the manifold is foliated by strictly stable CMC hypersurfaces $\{ S_r\}$ and note $H_{g_m}=0$ on $S_{r_m}$. Also note that the family $S_r$ obviously gives a one-sided family of hypersurfaces as in Definition~\ref{def:one-sided}. We also note that each $S_r$ is an analytic hypersurface in the spherical coordinates $\{r, \theta_1,\dots, \theta_{n-1}\}$ in which $g_m$ is analytic.

Given any $c>0$,  we can find some $\delta>0$ such that $S_{r}$ has mean curvature less than $c$, for $r\in (r_m+\delta, r_m+2\delta)$ and $u_m < \frac{c}{2}$. We apply Theorem~\ref{generic} and obtain that for $r$ in an open dense subset $J\subset (r_m+\delta, r_m+2\delta)$, $S_r$ is static regular of type (II) in $(\mathbb R^n\setminus B_r, \bar g_m, \bar u_m)$. Fix $r \in J$ and we denote the boundary $S_r$ by $\Sigma$. By Theorem~\ref{existence}, there exists positive constants $\epsilon_0, C$ such that for any $\epsilon\in (0, \epsilon_0)$ and for any  $(\tau, \phi)$ satisfying $\|(\tau, \phi) - ( g_m^\intercal, H_{g_m})\|_{\C^{2,\alpha}(\Sigma)\times \C^{1,\alpha}(\Sigma)} < \epsilon$, there is a static vacuum pair $(g, u)$ such that $\|(g, u) - (g_m,  u_m)\|_{\C^{2,\alpha}_{-q}(\mathbb R^n \setminus B_{r})}\le C\epsilon$. By choosing $\epsilon$ small, we have $|u-u_m|<\frac{c}{4}$ and thus $u<c$.  

In particular, we pick the prescribed mean curvature $\phi= H_{g_m}< c$ on $\Sigma$ and the prescribed metric $\tau$ not isometric to the standard metric of a round $(n-1)$-dimensional sphere for any radius. Since the background Schwarzschild manifold has a foliation of strictly stable CMC, for $\epsilon$ sufficiently small, the metric $g$ is also foliated by strictly stable CMC $(n-1)$-dimensional spheres. However, such $g$ cannot be rotationally symmetric. To see that, we suppose on the contrary that $g$ is rotationally symmetric, then by uniqueness of CMC hypersurfaces of \cite{Brendle:2013}, the boundary $\Sigma$ must be a round sphere of some radius. It contradicts to our choice of $\tau$.  

We can vary $\tau$ to get many, asymptotically flat, static vacuum  metrics that are not isometric to one another. 
\end{proof}

\section{Extension of local $h$-Killing vector fields}\label{se:h}

Let $h$ be a symmetric $(0,2)$-tensor on a Riemannian manifold $(M, g)$. We say a vector field $X$ is  \emph{$h$-Killing} if 
\[
	L_X g =h. 
\]
The following result extends the classical result of Nomizu \cite{Nomizu:1960} for  $h\equiv 0$. 

\begin{manualtheorem}{\ref{th:extension}}
Let $(M, g)$ be a connected, analytic Riemannian manifold. Let $h$ be an analytic, symmetric $(0,2)$-tensor on $M$. Let $U\subset M$ be a connected open subset satisfying  $\pi_1(M, U)=0$. Then every $h$-Killing vector field $X$  in $U$ can be extended to a unique $h$-Killing vector field on the whole manifold $M$. 
\end{manualtheorem}

Given a symmetric $(0,2)$-tensor $h$, recall in Section~\ref{se:h1} we define the $(1,2)$-tensor $T_h$ by, in local coordinates,  
\[
	(T_h)^i_{jk}= \tfrac{1}{2} (h^i_{j;k} + h^i_{k;j} - h_{jk;}^{\;\;\;\;\; i} )
\]
where the upper index $h^i_j$ is raised by $g$, and note that $(T_h)^i_{jk}$ is symmetric in $(j, k)$. The formula from Lemma~\ref{le:X} is the main motivation for defining the ODE system~\eqref{eq:ODE} below.
\begin{manuallemma}{\ref{le:X}}
Let $X$ be an $h$-Killing vector field. Then for any vector field $V$, we have 
\begin{align*}
	\nabla_V (\nabla X) = - R(X, V) + T_h (V) 
\end{align*}
where $R(X, V) := \nabla_X \nabla_V - \nabla_V \nabla_X - \nabla_{[X, V]}$. 
\end{manuallemma}

Let $p\in U$ and let $\Omega\subset M$ be 
a neighborhood of $p$ covered by the geodesic normal coordinate chart.  We shall extend $X$ to a unique $h$-Killing vector field in $\Omega$. For any point $q\in \Omega$, let  $\gamma(t)$ be the geodesic connecting $p$ and $q$.  We denote  $V= \gamma'(t)$. 

Consider the following inhomogeneous, linear ODE system for  a vector field $\hat X$ and a $(1,1)$-tensor $\omega$  along $\gamma(t)$:
\begin{align}\label{eq:ODE}
\begin{split}
	\nabla_{V} \hat X &= \omega (V)\\
	\nabla_V \omega&= -R(\hat X, V) + T_h (V).
\end{split}
\end{align}
Let $\hat X, \omega$ be the unique solution to \eqref{eq:ODE} with the initial conditions at $\gamma(0)=p$:
\begin{align*}
	\hat X(p) = X(p) \quad \mbox{ and } \quad  \omega(p)= (\nabla X )(p).
\end{align*}
On the other hand, by Lemma~\ref{le:X} $X, \nabla X$ also solve~\eqref{eq:ODE}  with the same initial conditions. By uniqueness of ODE we have $\hat X = X$ and  $\omega = \nabla X$ on the connected segment of $\gamma(t)\cap U$ containing the initial point $\gamma(0)$.

By varying the point $q$, the vector field $\hat{X}$ is defined everywhere in $\Omega$ and is smooth by smooth dependence of ODE. Moreover, since $(M,g)$ is analytic, its geodesics are analytic curves, and hence $X$ is analytic along $\gamma(t)$.  We summarize the above construction as the following lemma. 
\begin{lemma}\label{le:Xhat}
Let $(M, g)$ be a smooth manifold such that $\Int M$ is analytic. Let $h$ be an analytic symmetric $(0,2)$-tensor in $\Int M$ and smooth in $\overline M$, and $X$ be an $h$-Killing vector field in an open subset  $U\subset \Int M$. For $p\in U$, let $\Omega\subset M$ be a geodesic normal neighborhood of $p$. Then there is a smooth vector field $\hat X$ and a smooth $(1,1)$-tensor $\omega$ in $\Omega$ such that $\hat X=X, \omega = \nabla X$ in a neighborhood of~$p$ in $U$, and $(\hat X,  \omega)$ solves  \eqref{eq:ODE} along each $\gamma(t)$ and is analytic in $t$. 
\end{lemma}

After extending $X$ to $\hat X$ in $\Omega$ as above, we show that ${\hat X}$ is $h$-Killing in $\Omega$. As in Lemma \ref{le:V}, one may use Cauchy-Kovalevskaya Theorem to say that $\hat X$ also depends analytically on angular variables (not only on $t$) in $\Omega$ and hence $L_{\hat X} g=  h$ in $\Omega$ by analyticity. 
 Alternatively, we give another proof that is similar to Nomizu's original argument in the next proposition.

\begin{proposition}\label{pr:extension}
Let $\hat{X}, \omega$ be from Lemma~\ref{le:Xhat}. Then ${\hat X}$ is $h$-Killing in $\Omega$ and $\hat X = X$ everywhere in $U$. 
\end{proposition}
\begin{proof}
We  first show that 
\begin{align}\label{eq:omega}
g(\omega(Y), Z)+ g(\omega(Z),  Y) = h(Y, Z) 
\end{align}
for any vectors $Y, Z$ at an arbitrary point $q\in\Omega$. We can extend $Y, Z$ analytically in $t$ along the geodesic $\gamma(t)$ from $p$ to $q$.  Thus, the left hand side of \eqref{eq:omega}  along $\gamma(t)$ is analytic. We also know that \eqref{eq:omega} holds on $\gamma(t)$ for $t$ sufficiently small because $\omega = \nabla X$ in a neighborhood of $p$ and $X$ is $h$-Killing. Thus, \eqref{eq:omega} holds along the whole path $\gamma(t)$ by analyticity, and in particular at $\gamma(t_0)$. 

Next, 
we \emph{claim} that given an arbitrary vector $Y$ at a point in $\Omega$, say $\gamma(t_0)$ for some geodesic $\gamma$ starting at $p$, if we extend $Y$  such that $[V, Y]=0$ along $\gamma$, then $\nabla_Y \hat X - \omega(Y )$ along~$\gamma(t)$ is analytic in $t$. (We remark that clearly $\nabla_V \hat X$ is already analytic along $\gamma$, so the main point in the following proof is to show that it holds for general $Y$.) Once the claim is proven, we get 
\begin{align}\label{eq:X}
\nabla \hat X = \omega \quad\mbox{ in } \Omega.
\end{align}

\begin{proof}[Proof of Claim]
Note that along $\gamma(t)$,  $[V, Y] = \nabla_V Y - \nabla_Y V=0$ becomes a first-order linear ODE for $Y$ with analytic coefficients along $\gamma(t)$, and thus $Y$ is analytic along $\gamma(t)$. We show that $\nabla_Y \hat X - \omega(Y )$, together with $\nabla_Y \omega - R(Y, \hat X)$,  solves the following ODE system:
\begin{align}\label{eq:ODE2}
\begin{split}
	\nabla_V \big(\nabla_Y \hat X - \omega(Y )\big) &= \big( \nabla_Y \omega - R(Y, \hat X)\big) V- T_h (V, Y)\\
	\nabla_V \left(\nabla_Y \omega - R(Y, \hat X) \right) &= - R\big(\nabla_Y \hat X - \omega (Y), V\big) \\
	&\quad -R(\omega(Y), V)  + \nabla_Y (T_h(V))+ R(V, Y) \omega \\
&\quad + (\nabla_{\hat X} R) (V, Y)- R(Y, \omega(V)).
\end{split}
\end{align}
Note that the inhomogeneous term $T_h (V, Y)$ and those inhomogeneous  terms in in the 3rd and 4th lines above are all analytic along $\gamma(t)$. Since any solutions to above  the linear ODE system~\eqref{eq:ODE2} with analytic coefficients and analytic inhomogeneous terms are analytic along $\gamma(t)$, in particular, $\nabla_Y \hat X - \omega(Y )$ is analytic along $\gamma(t)$. 
 
The computations are similar to \cite[Proof of Theorem 4]{Nomizu:1960}. We include the proof for completeness. For the first identity, 
\begin{align*}
	&\nabla_V \left(\nabla_Y \hat X - \omega(Y )\right)  \\
	&= \nabla_Y \nabla_V \hat X  + R(V, Y)\hat X - (\nabla_V \omega)(Y) - \omega(\nabla_V Y)\\
	&=\nabla_Y (\omega (V))   + R(V, Y) \hat X + R(\hat X, V) Y - T_h (V, Y)  - \omega(\nabla_V Y)\quad \mbox{(by \eqref{eq:ODE})}\\
	&=(\nabla_Y \omega )(V)  - R(Y, \hat X) V  - T_h (V, Y)\quad \mbox{(by Bianchi identity)}
\end{align*}
where we also use $[V, Y]=0$.

To prove the second identity, we compute each term below:
\begin{align*}
	\nabla_V \nabla_Y \omega &= \nabla_Y \nabla_V \omega + R(V, Y) \omega \\
	&=\nabla_Y \big(-R(\hat X, V) + T_h(V)\big)+ R(V, Y) \omega\\
	&= -(\nabla_Y R)(\hat X, V) - R(\nabla_Y \hat X, V) - R(\hat X, \nabla_Y V)  + \nabla_Y (T_h(V))+ R(V, Y) \omega\\
	\nabla_V (R(Y, \hat X)) &= (\nabla_V R)(Y, \hat X) + R(\nabla_V Y, \hat X) + R(Y, \nabla_V \hat X)\\
	&= (\nabla_V R)(Y, \hat X) - R(\hat X, \nabla_V Y) + R(Y, \omega(V)).
\end{align*} 
Subtracting the previous two identities and noting that the terms $ - R(\hat X, \nabla_Y V)$ and  $R(\hat X, \nabla_V Y)$ cancel out by $[V, Y]=0$, we derive  
\begin{align*}
&\nabla_V \nabla_Y \omega -\nabla_V \big(R(Y, \hat X)\big)\\
&=   -(\nabla_Y R)(\hat X, V) - R(\nabla_Y \hat X, V)  + \nabla_Y (T_h(V))+ R(V, Y) \omega\\
&\quad - (\nabla_V R)(Y, \hat X) - R(Y, \omega(V))\\
&= - R(\nabla_Y \hat X -\omega(Y), V)  -R(\omega(Y), V)  + \nabla_Y (T_h(V))+ R(V, Y) \omega \\
&\quad + (\nabla_{\hat X} R) (V, Y)- R(Y, \omega(V))\quad \mbox{(by Bianchi identity)}.
\end{align*}
Rearranging the terms gives the second identity in \eqref{eq:ODE2}.
\end{proof}

Lastly, \eqref{eq:omega} and \eqref{eq:X} together imply that $\hat X$ is $h$-Killing in $\Omega$. Since $U$ is connected and $X-\hat X$ is Killing in $U$ which is identically zero in an open subset, we have $X = \hat X$ in $U$. 
\end{proof}

We have shown how to extend the $h$-Killing vector field $X$ in a geodesic normal neighborhood. Using the assumption that $\pi_1(X, U)=0$, we show how to extend $X$ globally and complete the proof of Theorem~\ref{th:extension}. 

\begin{proof}[Proof of Theorem~\ref{th:extension}]
For any point $q$ in $M$, we let $\gamma$ be a path from  $p\in U$ to $q$. The path $\gamma(t)$ is covered by finitely many geodesic normal neighborhoods. We extend $X$ at $q$  along the path by Proposition~\ref{pr:extension}. For any other path $\tilde{\gamma}$ from $U$ to $q$ that is sufficiently close to $\gamma$, $\tilde \gamma$ is also covered by the same collection of neighborhoods, and thus the extension along $\tilde{\gamma}$ gives the same definition of $X$ at $q$.

To show that the definition of $X$ at $q$ doesn't depend on the paths from $U$ to~$q$, we use the assumption that $\pi_1(M, U)=0$. Since any path $\tilde \gamma $ from $U$ to $q$ is homotopic to $\gamma$  relative to $U$, there are finitely many paths from $U$ to $q$, say $\gamma_1=\gamma, \gamma_2,  \gamma_3, \dots, \gamma_k = \tilde{\gamma}$, such that each pair of consecutive paths, $\gamma_i$ and $\gamma_{i+1}$,  can be covered by the same collections of neighborhoods. Thus the extension of $X$ at $q$ is the same on each pair and thus on all those paths. That completes the proof. 
\end{proof}
\begin{appendix}

\section{Formulas of (linearized) geometric operators}\label{se:formula}

Given a Riemannian manifold  $(U, g)$, the {\it Bianchi operator} $\b_{ g}$  and its {\it adjoint operator} $\b_g ^*$ are defined by, for a symmetric $(0,2)$-tensor $h$ and a vector field $X$, 
\begin{align}
	\b_{g} h&= -\Div_{g} h + \tfrac{1}{2} d \mathrm{tr}_{g} h\label{eq:Bianchi}\\
	\b^*_g X&= \tfrac{1}{2} \big(L_X g - (\Div_g X )g \big).\label{eq:Bianchi*}
\end{align}
We write the Lie derivative of the Riemannian metric $g$ along a vector field $X$ by 
\begin{align}\label{eq:Lie}
	\DD_g X&= \tfrac{1}{2} L_X  g.
\end{align}
We record the following basic identities:
\begin{align}
	\beta_g \DD_g X &= -\tfrac{1}{2} \Delta_g X - \tfrac{1}{2} \Ric(X, \cdot) \label{eq:laplace-beta}\\
	\beta_g (fh)&= f \beta_g h - h(\nabla f, \cdot) + \tfrac{1}{2} (\tr_g h) df\quad \mbox{ for a scalar function $f$}.  \label{eq:product} 
\end{align}

We frequently use the linearization  of geometric quantities or operators. 
In a smooth manifold $U$, let $g(s)$ be a smooth family of Riemannian metrics with $g(0) =  g$ and $ g'(0)  = h$. 
We define the linearization of the Ricci tensor at $g$ by  $\Ric'|_g (h):=\left.\dt\right|_{s=0} \Ric_{g(s)}$. The other linearized quantities  are denoted in the similar fashion. For example, 
\begin{align*}
	(\nabla^2)'|_g(h) := \left.\dt\right|_{t=0} \nabla^2_{g(t)} \quad \mbox{ and } \quad \Delta'|_g(h) := \left.\dt\right|_{t=0} \Delta_{g(t)}.
\end{align*}
We often omit the subscripts $|_{g}$ when the context is clear.

Let $\Sigma\subset U$ be a hypersurface and $\nu$ be the unit normal to $\Sigma$. Consider a local frame $\{ e_0, e_1, \dots, e_{n-1}\}$ such that $e_0$ is the parallel extension of $\nu$ along itself near $\Sigma$. We list those linearized quantities in the local frame (where $f$ is an arbitrary scalar function), see \cite[Section 2.1]{An-Huang:2021}:
\begin{align}
\begin{split}\label{equation:Ricci}
	(\Ric'|_g(h))_{ij} &=-\tfrac{1}{2} g^{k\ell} h_{ij;k\ell} + \tfrac{1}{2} g^{k\ell} (h_{i k; \ell j} + h_{jk; \ell i} ) - \tfrac{1}{2} (\tr h)_{;ij}  \\
	&\quad + \tfrac{1}{2} (R_{i\ell} h^\ell_j + R_{j\ell} h^\ell_i )- R_{ik\ell j} h^{k\ell}
\end{split}\\
R'|_g(h) &=-\Delta (\tr h ) +\Div \Div h - h \cdot \Ric_g\notag\\
	\big((\nabla^2)'|_g(h)  f \big)_{ij}&= \tfrac{1}{2} g^{k\ell}f_{,\ell}  \big( h_{ij;k} - h_{jk;i} - h_{ik;j} \big)\label{eq:derivative}\\
	\big( \Delta'|_g(h) \big)f &= -  g^{ik}  g^{j\ell}f_{;ij}h_{k\ell} + g^{ij} \left(-(\Div_{g} h)_i + \tfrac{1}{2} (\tr_{g} h)_{;i}  \right) f_{,j}. \label{eq:laplace}
\end{align}

We write $\nu = \nu_g$ for short. We let $\omega(e_a) = h(\nu, e_a)$ be the one-form on the tangent bundle of $\Sigma$ and $\{\Sigma_t\}$ be the foliation by $g$-equidistant hypersurfaces to $\Sigma$.
For tangential directions $a, b, c \in \{ 1, \dots, n-1\}$ on $\Sigma$, we have 
\begin{align}
	\nu'|_g(h) &= - \tfrac{1}{2} h(\nu, \nu) \nu - g^{ab} \omega(e_a) e_b\label{equation:normal}\\
	A'|_g(h) & = \tfrac{1}{2} (L_{\nu} h)^\intercal - \tfrac{1}{2} L_\omega g^\intercal - \tfrac{1}{2} h(\nu, \nu) A_g\label{eq:sff}\\
	&= \tfrac{1}{2}( \nabla_{\nu} h)^\intercal+ A_g\circ h - \tfrac{1}{2} L_\omega g^\intercal- \tfrac{1}{2} h(\nu, \nu) A_g \notag \\
	H'|_g(h) &=\tfrac{1}{2} \nu(\tr h^\intercal) - \Div_\Sigma \omega - \tfrac{1}{2} h(\nu, \nu) H_g\notag
\end{align}
where $(A\circ h)_{ab} = \frac{1}{2} (A_{ac}h^c_b + A_{bc} h_a^c)$. 
In the special case that $h = L_X g$ along $\Sigma$ where the vector field $X = \eta \nu + X^\intercal$ for some tangential vector $X^\intercal$ to $\Sigma$, we get
\begin{align*}
	H'|_g(L_X g) &= - \Delta_\Sigma \eta - (|A|^2+\Ric(\nu, \nu))\eta + X^\intercal (H).
\end{align*}

We also have the following ``linearized'' Ricatti equation:
\begin{align}
	\begin{split}\label{eq:Ricatti}
		\nu (H'|_g(h))& =  - \tfrac{1}{2} R'(h) + \tfrac{1}{2} R'^\Sigma|_{g^\intercal} (h^\intercal) + A_g \cdot (A_g \circ h)  \\
		&\quad - A_g \cdot A'(h) -H_g H'(h) \\	
	&\quad + \tfrac{1}{4}\big( - R_g +R^\Sigma_g -  |A_g|^2 - H_{g}^2\big)h(\nu, \nu)\\
	&\quad 	- \tfrac{1}{2}\Delta_\Sigma h(\nu, \nu)  + g^{ab} \omega (e_a) e_b(H_g).
	\end{split}
\end{align}
In the third equation, $\tr h^\intercal$ denotes the tangential trace of $h$ on $\Sigma_t$, defined in a collar neighborhood of $\Sigma$. For the last equation,  the dot $\cdot$ means the $g$-inner product, $(A_g\circ h)_{ab} = \tfrac{1}{2} (A_{ac} h^c_{b}+A_{bc} h^c_{a})$, and   $R^\Sigma_g, \Delta_\Sigma$ are respectively the scalar curvature and the Laplace operator of  $(\Sigma, g^\intercal)$.

\section{Spacetime harmonic gauge and analyticity}

Let $(M, g)$ be a Riemannian manifold and let $u>0$ be a scalar function on $M$. We define the spacetime metric $\mathbf g$ on $\mathbf{N}:=\mathbb{R}\times M$ by (in either Riemannian or Lorentzian signature)
\begin{align}\label{eq:spacetime}
	\mathbf{g} = \pm u^2 dt^2 + g.
\end{align}
Recall that we say the pair $(g, u)$ is static vacuum if $S(g, u)=0$ as defined in \eqref{eq:static}. When $u>0$, the condition is equivalent to that the spacetime metric satisfies $\Ric_{\bf g}=0$. Thus, we may also refer such $\mathbf g$ as a static vacuum (spacetime) metric. 

\subsection{Harmonic gauge}
Fix a general background triple $(M, \bar g, \bar u)$ and thus the background spacetime metric $\bar {\mathbf g} = \pm \bar u^2 dt^2 + \bar g$ on~$\mathbf {N}$. Recall the Bianchi operator $\beta_{\bar {\mathbf{g}}} $ sends any symmetric $(0,2)$-tensor ${\bm k}$  on $\mathbf {N}$ to the covector $\beta_{\bar {\mathbf{g}}}{\bm k}$,  defined by 
\[
	\beta_{\bar {\mathbf{g}}} {\bm k}=  - \Div_{\bar {\mathbf{g}}} {\bm k}+ \tfrac{1}{2} \bm{d}  (\tr_{\bar {\mathbf{g}}}{\bm k}).
\]

The following proposition explains the motivation behind the definition of the static-harmonic gauge in Definition~\ref{de:gauge}. Note that this fact is not used anywhere else in the paper. 

\begin{proposition}\label{pr:harmonic}
For $\mathbf {g} $ taking the form \eqref{eq:spacetime}, 
\begin{align*}
	\beta_{\bar {\mathbf{g}}} \mathbf {g} =  \beta_{\bar g} g + \bar u^{-2} u du -\bar u^{-1} g(\nabla_{\bar g} \bar u, \cdot).
\end{align*}
That is, in the coordinate $t$ of $\mathbb{R}$ and a local coordinate chart  $\{  x_1,\dots, x_n\}$ of $M$
\begin{align*}
	\beta_{\bar {\mathbf{g}}} \mathbf {g} (\partial_t)&=0\\
	\beta_{\bar {\mathbf{g}}} \mathbf {g} (\partial_a ) &=\beta_{\bar g} g(\partial_a )  + \bar u^{-2} u \partial_a u -\bar u^{-1} g(\nabla_{\bar g}\bar  u, \partial_a )\qquad \mbox{ for $a=1,\dots, n$}.
\end{align*}

\end{proposition}
\begin{proof}
Let $\bm{\nabla}, \nabla$ be the covariant derivatives of $\bar{\mathbf g}, \bar g$, respectively.  We have
\begin{align*}
	\bm{\nabla}_{\partial t} \partial_t &=  \mp \bar u \nabla  \bar u\\
	\bm{\nabla}_{\partial a} \partial_t = \bm{\nabla}_{\partial t} \partial_a&=  \bar u^{-1}  {\partial_a}  \bar u\, \partial_t\\
		\bm{\nabla}_{\partial a} \partial_b &= \nabla_{\partial_a} \partial_b.
\end{align*}
Then we compute $\Div_{\bar {\mathbf{g}}} {\mathbf g}$:
\begin{align*}
	(\Div_{\bar {\mathbf g}} {\mathbf g})(\partial_t ) &= 0\\
	(\Div_{\bar {\mathbf g}}  {\mathbf g})(\partial_a ) &= \bar u ^{-1}  g(\nabla \bar u, \partial_a) - \bar u^{-3} u^2 \partial_a \bar u + (\Div_{\bar g} g)(\partial_a).
\end{align*}
Next we compute the trace term:
\begin{align*}
	\tr_{\bar {\mathbf{g}}} \mathbf g &= \bar u^{-2} u^2 + \tr_{\bar g} g\\
	\bm{d} (\tr_{\bar {\mathbf{g}}} \mathbf g ) &=-2 \bar u^{-3} u^2d\bar u + 2\bar u^{-2} u du + d (\tr_{\bar g} g). 
\end{align*}
Combining the above identities give
\begin{align*}
	\beta_{\bar {\mathbf{g}}} \mathbf {g} &= -\Div_{\bar {\mathbf{g}}} \mathbf{g} + \tfrac{1}{2} \bm d (\tr_{\bar {\mathbf{g}}} \mathbf{g} )\\
	&=\beta_{\bar g} g -\bar u ^{-1}  g( \nabla \bar u, \cdot) + \bar u^{-3} u^2 d \bar u  - \bar u^{-3} u^2 d\bar u + \bar u^{-2} u du \\
	&=\beta_{\bar g} g -\bar u ^{-1}  g( \nabla \bar u, \cdot)  + \bar u^{-2} u du.
\end{align*}
\end{proof}

If $M$ has nonempty boundary $\partial M$, then on the boundary $\partial \mathbf N := \mathbb{R}\times \partial M$, the Cauchy boundary data $({\mathbf g}|_{\partial \mathbf N}, A_{\bf g})$ of $\partial \bf N \subset (\mathbf N, \mathbf g)$ can be expressed in terms of $ u, \nu( u)$ and  $(g^\intercal, A_g)$ of $\partial M \subset (M, g)$. 

\begin{proposition}
On the boundary $\partial \bf N$, we have 
\begin{align*}
	{\bf g} |_{\partial  \bf N} &=  \pm u^2 dt^2  + g^\intercal\\
	A_{\bf g} &= \pm u \nu ( u) dt^2 + A_g.
\end{align*}
Consequently, the corresponding linearizations at $\bar {\bf g}$ along the deformation ${\bf h}$, which is generated by an infinitesimal deformation $(h,v)$ at $(\bar g, \bar u)$, are given by 
\begin{align*}
	{\bf h} |_{\partial  \bf N} &=  \pm 2 \bar u v dt^2  + h^\intercal\\
	A'|_{\bar {\bf g}}  (\bf h)&= \pm \big( v \nu (\bar u) + \bar u (\nu(u))' \big) dt^2 + A'|_g(h).
\end{align*}
\end{proposition}
\begin{proof}
The identity for the induced metric $\bf g$ is obvious. For $A_{\bf g}$, we compute  in local frame $\{ \partial_t, \nu, e_1, \dots, e_{n-1}\}$ where  $\nu$ is the unit normal to $\partial  \bf N$ (which coincides with the unit normal for $\partial M$ and $e_1, \dots, e_{n-1}$ are tangential to $\partial M$. Then 
\begin{align*}
	A_{\bf g} (\partial_t, \partial_t) &= - {\bf g}( \nu, {\bm \nabla}_{\partial_t} \partial_t ) = \pm u \nu(u)\\
	 A_{\bf g} (e_a, e_b) &=  - {\bf g}( \nu, {\bm \nabla}_{e_a} e_b) =  A_g (e_a, e_b) \quad \mbox{ for } a, b = 1, \dots, n-1
\end{align*}
and all other components of $A_{\bf g}$ are zero. 
\end{proof}
\subsection{Analyticity}\label{se:analytic}

 We say a scalar function $f$ is (real) analytic in the coordinate chart $\{ x_1,\dots, x_n \}$ on  a manifold $M$ if for each $p\in M$, there is a neighborhood $U$ of $p$  such that
\[
	f(x) = \sum_{|I|=0, 1, \dots, } \frac{1}{|I|!} \partial^I f(p) (x-p)^I \quad \mbox{ for all } x\in U
\]
where $I$ is a multi-index. 

A tensor $h$ is said to be analytic in the coordinate chart $\{ x_1,\dots, x_n \}$ if all of its components are analytic. 
A Riemannian manifold $(M,g)$ is called analytic if it can be covered by coordinate charts $\{U_i\}$ where $g$ is analytic in each chart; and similarly, a hypersurface embedded in $(M,g)$ is called analytic if it is analytic in each $U_i$.

A classical result of M\"uller zum Hagen~\cite{Muller-zum-Hagen:1970} says that if  the spacetime $(\bf N, \bar {\bf g})$ is static vacuum, then  $\bar {\bf g}
$ is analytic in harmonic coordinates. Below, we state and prove a version of the corresponding result for the time-slice $(M, \bar g, \bar u)$. 

Let $(M, \bar g)$ be a Riemannian manifold with a scalar function $\bar u>0$  on $M$.  A coordinate chart $(x_1,\dots, x_n)$ on  $(M, \bar g, \bar u)$ is called \emph{static-harmonic} if 
\[
	\Delta x_k + \bar u^{-1} \nabla \bar u \cdot \nabla x_k=0 \mbox{ for all } k =1, \dots, n.
\]
Here and  below, the covariant derivatives and curvatures are all with respect to $\bar g$. In coordinates $\{ x_1,\dots, x_n\}$ we denote the Christoffel symbols $\Gamma^k:=\bar g^{ij} \Gamma^k_{ij} $ and compute 
\[
	\Delta x_k + \bar u^{-1} \nabla \bar u \cdot \nabla x_k = -\Gamma^k + \bar u^{-1} \bar g^{ik} \frac{\partial \bar u}{\partial x_i}.
\]
Therefore, a local coordinate chart $\{ x_1, \dots, x_n\}$ is static-harmonic if and only if 
\[
-\Gamma^k+ \bar u^{-1} \bar g^{ik} \frac{\partial \bar u}{\partial x_i} =0.
\]

\begin{theorem}\label{th:analytic}
Let $(M, \bar g, \bar u)$ be a static vacuum triple with $\bar u>0$. Then $\bar g$ and  $\bar u$ are analytic in  static-harmonic coordinates in  $\Int M$. 
\end{theorem}
\begin{remark}
As a direct consequence of \cite[Theorem 2.1]{DeTurck-Kazdan:1981}, $\bar g$ and $\bar u$ are analytic in harmonic coordinates and in geodesic normal coordinates  in $\Int M$. 
\end{remark}
\begin{proof}
Let $\{ y_1, \dots, y_n\}$ be an arbitrary local coordinate chart about the point $p\in \Int M$. By standard elliptic theory~\cite[p. 228]{Bers-John-Schechter:1979}, in a neighborhood of $p$ there are solutions $x_1, \dots, x_n$ of 
\begin{align*}
	&\Delta x_j - \bar u^{-1} \nabla \bar u \cdot \nabla x_j=0  \\
	&x_j (p)=0 \; \mbox{ and }\; \frac{\partial x_j}{\partial y_i } (p) = \delta_{ij}.
\end{align*}
Those functions $\{ x_1,\dots, x_n\}$ are the desired static-harmonic coordinates. 

The static vacuum pair satisfies 
\begin{align}\label{eq:static-A}
\begin{split}
	-  \Ric+\bar u^{-1} \nabla^2 \bar u&=0\\
	\Delta \bar u&=0.
\end{split}
\end{align}
Recall the well-known formula:
\[
	\Ric_{ij} =  -\frac{1}{2} \bar g^{rs} \frac{\partial^2 \bar g_{ij}}{\partial x_r \partial x_s} + \frac{1}{2} \left(\bar g_{ri} \frac{\partial \Gamma^r}{\partial x_j} + \bar g_{rj} \frac{\partial \Gamma^r}{\partial x_i } \right)+ \mathcal O(\bar g,\partial \bar g).
\]
where $ \mathcal O(\bar g,\partial \bar g)$ denotes the terms involving at most one derivative of the metric $\bar g$. In particular, in the static-harmonic coordinates:
\begin{align*}
	&-  \Ric_{ij}+\bar u^{-1} \bar u_{;ij} \\
	&= \frac{1}{2} \bar g^{rs} \frac{\partial^2 \bar g_{ij}}{\partial x_i \partial x_j} - \frac{1}{2} \left( \bar g_{ri} \frac{\partial }{\partial x_j} \left(\bar u^{-1} \bar g^{\ell r} \frac{\partial \bar u}{\partial x_\ell} \right) + \bar g_{rj} \frac{\partial }{\partial x_i} \left(\bar u^{-1} \bar g^{\ell r} \frac{\partial \bar u}{\partial x_\ell} \right) \right) \\
	&\quad + \bar u^{-1}  \bar u_{;ij} +  \mathcal O(\bar g,\partial \bar g)\\
	&= \frac{1}{2} \bar g^{rs} \frac{\partial^2\bar  g_{ij}}{\partial x_i \partial x_j} +  \mathcal O( \bar g,\partial \bar g, \bar u, \partial \bar u).
\end{align*}
Thus, \eqref{eq:static-A} is a quasi-linear elliptic system in static-harmonic coordinates, so the solutions $\bar g_{ij}, \bar u$ are analytic in static-harmonic coordinates. 
\end{proof}

\end{appendix}

\section*{Data Availability Statement}
All data generated or analysed during this study are included in this published article.



\bibliographystyle{amsplain}
\bibliography{../2020}
\end{document}